\documentclass[11pt,reqno]{amsart}
\usepackage{hyperref}

\allowdisplaybreaks[4]

\usepackage{amsmath}
\usepackage{amssymb}
\usepackage{mathrsfs}
\usepackage{amsthm}
\usepackage{bm}

\usepackage{booktabs}
\usepackage{float}
\usepackage{graphicx,epstopdf}
\usepackage[caption=false]{subfig}
\allowdisplaybreaks

\graphicspath{{Fig/}}
\usepackage{geometry}
\usepackage{color}
\theoremstyle{plain} 
\newtheorem{lemma}{Lemma}[section]
\newtheorem{thm}[lemma]{Theorem}

\newtheorem{coro}[lemma]{Corollary}
\newtheorem{remark}[lemma]{Remark}
\newtheorem{prop}[lemma]{Proposition}

\newcommand{\rmd}{\mathrm d}

\begin{document}
	\title[]{Random attractors and almost-sure stability under discretization of a stochastic autoparametric system}
	\author[Chuchu Chen, Jialin Hong, Yibo Wang]{Chuchu Chen, Jialin Hong, Yibo Wang*}
\address{State Key Laboratory of Mathematical Sciences, Academy of Mathematics and Systems Science, Chinese Academy of Sciences, Beijing 100190, People's Republic of China, 
\and 
School of Mathematical Sciences, University of Chinese Academy of Sciences, Beijing 100049, China}
\email{chenchuchu@lsec.cc.ac.cn; hjl@lsec.cc.ac.cn; wangyb@amss.ac.cn}
\thanks{This work is funded by the National Key R\&D Program of China under Grant (No. 2024YFA1015900 and No. 2020YFA0713701),  by the National Natural Science Foundation of China (No. 12031020, No. 12461160278, and No. 12471386), and by Youth Innovation Promotion Association CAS}
\subjclass{37M22, 37M25, 60H35}
\thanks{*Corresponding author} 
	\begin{abstract} 
		For a stochastic autoparametric block-and-pendulum system, the long-time dynamics exhibit two fundamental features: 
		the almost-sure stability of the single mode solution, characterized by its Lyapunov exponent, and the global asymptotic dynamics when this single mode solution  loses stability. This naturally raises the question of whether these dynamical features are preserved under discretization, since such preservation is essential for the resulting discrete system to faithfully capture the qualitative behavior of the continuous system. To address this question, we first establish the existence of a random attractor for the continuous system subject to multiplicative stochastic excitation, providing a rigorous characterization of the global asymptotic dynamics. We then propose a numerical discretization that induces a discrete random dynamical system and prove the convergence of its random attractor to the continuous one as the step size tends to zero. In addition, we show that the numerical Lyapunov exponent of the single mode solution has the same sign as its continuous counterpart for sufficiently small step sizes, thus preserving the corresponding almost-sure stability or instability classification. These results demonstrate that the proposed discretization captures both the global asymptotic dynamics and the stability characteristics of the underlying stochastic autoparametric system. 
	\end{abstract}
	\keywords{Single mode solution $\cdot$ Random attractor $\cdot$ Lyapunov exponent $\cdot$ Stability $\cdot$ Discrete random dynamical system}
	\maketitle
	\section{Introduction} 
	Autoparametric systems  constitute a classical class of nonlinear oscillators in which energy can be transferred between interacting modes through internal resonance \cite{Vyas2001}. 
	Representative examples include block-and-pendulum models \cite{Bajaj1994,Banerjee1996,Hatwal1983}, block-and-beam models \cite{Haxton1972,Yan2017}, and other related nonlinear mechanical models. 
	A characteristic feature of these systems is the existence of a single mode solution, where the primary component oscillates while the secondary component stays at rest. The loss of stability of this single mode solution is associated with the onset of autoparametric energy transfer from the primary component to the secondary one, which underlies the vibration-absorption mechanism of these systems. In the presence of stochastic excitation, both the stability threshold and the ensuing long-time behavior become path-dependent, raising fundamental questions concerning almost-sure stability and global asymptotic behaviors.

	In this paper, we consider a stochastic block-and-pendulum autoparametric system consisting of a primary block-spring-damper subsystem coupled to a pendulum as the secondary subsystem. Its dimensionless form is given by 
	\begin{equation}\label{SDE}
		\left\{
		\begin{aligned}
			&\ddot{\ell}(t) + 2\zeta_1 \dot{\ell}(t) + \kappa_{1} \ell(t) - \gamma \big( \ddot{\vartheta}(t) \sin \vartheta(t) + \dot{\vartheta}^2(t) \cos \vartheta(t) \big) = \sigma \dot{W}(t), \\
			&\ddot{\vartheta}(t) + 2\zeta_2 \dot{\vartheta}(t) + (\kappa_{2} - \ddot{\ell}(t)) \sin \vartheta(t) = 0,
		\end{aligned} \right.
	\end{equation} 
	where the constants $\zeta_1$ and $\zeta_2$ are scaled damping coefficients, $\kappa_{1}$ is the dimensionless stiffness of the spring, $\kappa_{2}$ is a dimensionless parameter associated with the frequency of the undamped pendulum, and $\sigma$ is the scaled noise intensity. The parameter $\gamma = m_2/(m_1 + m_2)\in(0,1)$ represents the mass ratio, where $m_1$ and $m_2$ are the masses of the block and the pendulum bob, respectively.
    When recast into first-order state-space form, system \eqref{SDE} is a stochastic differential equation (SDE) driven by multiplicative noise; see \eqref{system} below. 
	A schematic illustration of the configuration is shown in Figure \ref{Fig:system}. 
	
	System \eqref{SDE} possesses a single mode solution of the form $(\ell(t),\dot{\ell}(t),\vartheta(t),\dot\vartheta(t)) = (\eta(t),\dot{\eta}(t),0,0)$, 
	under which the dynamics reduce to a linear stochastic oscillator for the primary subsystem: 
	\begin{equation*}
		\ddot{\eta}(t) + 2\zeta_1 \dot{\eta}(t) + \kappa_{1} \eta(t) = \sigma \dot{W}(t). 
	\end{equation*} 
	The single mode solution takes values in the invariant manifold $\mathcal{N}_0=\{(\ell,\dot{\ell},0,0):\ell,\dot{\ell}\in\mathbb R\}$, which corresponds to motions for which the pendulum remains vertical. 
	A basic issue is whether $\mathcal{N}_0$ is stable with respect to transverse perturbations in $(\vartheta,\dot{\vartheta})$, namely, whether small pendulum-angle perturbations away from $\mathcal{N}_0$ decay or grow. 
	Linearization of \eqref{SDE} along the single mode solution yields a system of parametrically excited stochastic equations \eqref{linearization}. Its top Lyapunov exponent $\lambda_{0}$ provides a criterion for the transverse almost-sure stability of the single mode solution: $\lambda_{0}<0$ implies stability, whereas $\lambda_{0}>0$ implies instability. 
	In the past decades, the stability of the single mode solution has received considerable attention; see, e.g., \cite{Ariaratnam1991,Namachchivaya2007}. 
	Recently, Baxendale and Namachchivaya \cite{Baxendale24SIAM} derived rigorous asymptotic estimates for $\lambda_{0}$ and established explicit criteria determining the sign of $\lambda_{0}$. These results provide a rigorous almost-sure stability theory for the single mode solution.

	\begin{figure} 
		\centering
		\includegraphics[scale = 0.4]{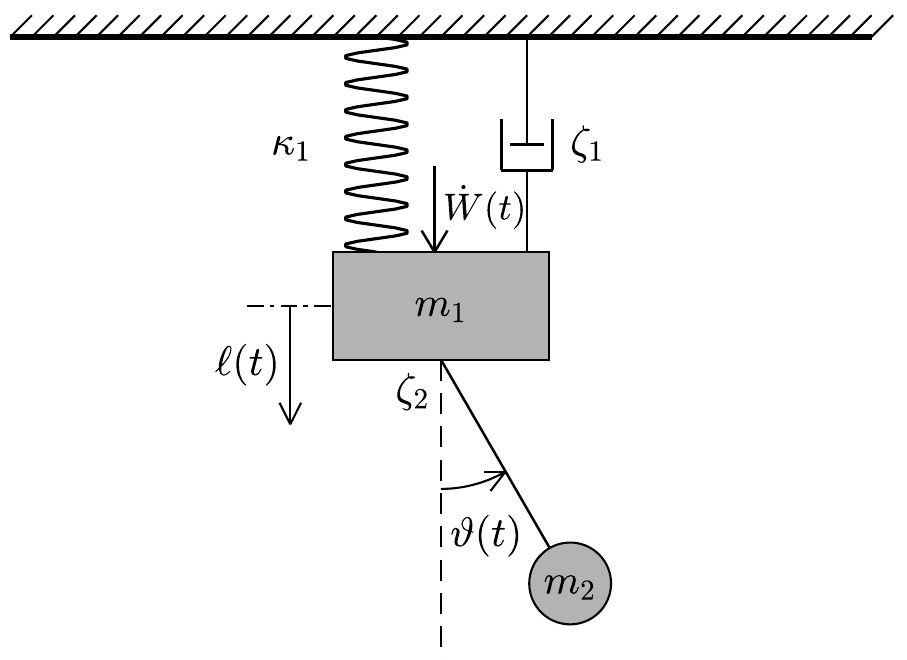}  
		\caption{Autoparametric system: block and pendulum}
		\label{Fig:system}
	\end{figure}

	While the Lyapunov exponent characterizes the local transverse stability near the invariant manifold $\mathcal{N}_0$, it is also natural to ask what global asymptotic behavior may emerge once trajectories leave a neighborhood of $\mathcal N_0$. 
	In particular, 
    after the onset of autoparametric activation, 
    trajectories may evolve far away from the invariant manifold and approach nontrivial random invariant sets. Understanding such post-activation dynamics naturally leads to the framework of random dynamical systems (RDSs) and random attractors. The theory of random attractors has been extensively developed for RDSs; see, e.g., \cite{arnoldRDS,Crauel1994}. 
    However, to the best of our knowledge, random attractors have not been established for stochastic autoparametric systems of the type considered here. 
	This observation raises the following question: 
	\begin{itemize}
		\item[(Q1)] Does the stochastic system \eqref{SDE} generate an RDS possessing a random attractor that captures its global asymptotic dynamics?
	\end{itemize}

Numerical approximation plays an essential role in the investigation of such long-time dynamical features, which raises the question of whether these features are faithfully captured under time discretization. Considerable progress has been made in the approximation of long-time properties of stochastic systems, including invariant measures, ergodicity, and random attractors; see, e.g., \cite{Caraballo2024,Chen20252,Chen2025,Han2017,Li2019}.  
Nevertheless, the simultaneous approximation of the random attractor and the preservation of Lyapunov-exponent-based almost-sure stability under discretization remain much less understood, particularly for stochastic autoparametric systems. At the numerical level, time discretization gives rise to a computable discrete RDS, making it important to establish a rigorous connection between its long-time dynamics and those of the underlying continuous RDS. This leads to the second question:
    
	\begin{itemize}
		\item[(Q2)] 
        Can the random attractor of the continuous RDS be approximated by that of the discrete RDS, while preserving the almost-sure stability classification of the single mode solution?
	\end{itemize}

	In this work, we give an affirmative answer to both questions. For (Q1), we prove that the continuous RDS generated by \eqref{SDE} admits a random attractor, yielding a global description of its asymptotic dynamics. To the best of our knowledge, this is the first result establishing a random attractor for a stochastic autoparametric system of the type considered here.
	For (Q2), we propose a discretization that induces a discrete RDS and establish two qualitative approximation properties: the random attractor of the discrete system converges to that of the continuous system as the step size tends to zero, and the numerical Lyapunov exponent of the single mode solution preserves the sign of its continuous counterpart for sufficiently small step sizes. Thus, the discretization captures both the global asymptotic dynamics and the almost-sure stability classification of the continuous system. 
	
	At the continuous level, the main difficulty in proving the existence of a random attractor lies in the noise structure of the corresponding first-order SDE. Although the stochastic forcing enters additively in the implicit second-order formulation \eqref{SDE}, resolving the coupled accelerations gives rise to state-dependent diffusion coefficients. This multiplicative-noise structure obscures the dissipative estimates required to establish pullback absorption.
	To overcome this obstacle, we introduce a transformation that converts the system into an equivalent SDE with additive noise while retaining the key dissipative structure. This enables us to establish a tempered random absorbing set for the continuous RDS and hence the existence of a random attractor.
	
	At the discrete level, two additional difficulties arise. First, discretization does not automatically preserve the dissipative structure available in the continuous setting, making pullback absorption nontrivial. To address this issue, we introduce an auxiliary transformation tailored to the numerical scheme and recover a suitable discrete dissipative structure, which yields a random absorbing set and ultimately allows us to establish the convergence of the discrete random attractor.  
	Second, since the numerical Lyapunov exponent $\lambda^\tau$ is defined through the asymptotic growth rate of the variational dynamics, it is inherently a long-time quantity and cannot be inferred directly from finite-time convergence alone. 
	Our strategy combines ergodic arguments, stationary measure approximation, and strong convergence estimates to quantify the difference between $\lambda^\tau$ and $\lambda_0$, implying the sign consistency for sufficiently small $\tau$. 
	The strategy may also be useful for comparing numerical and continuous Lyapunov exponents in broader classes of dissipative stochastic systems.

	The paper is organized as follows. Section~\ref{Sec:preliminary} collects notation and basic facts, including the stationary measure of the model and background results on Lyapunov exponents. Section~\ref{Sec:attractor} establishes the existence of a random attractor for the continuous RDS. Section~\ref{Sec:numerical attractor} introduces the numerical method and proves the existence and convergence of a random attractor for the discrete RDS. Section~\ref{Sec:numerical lambda} analyzes the numerical Lyapunov exponent for the single mode solution and proves its sign preservation. 
	Section~\ref{Sec:proof} presents the proofs of propositions in Section \ref{Sec:numerical lambda}. 
	Finally, Section~\ref{Sec:experiment} reports numerical experiments.
	
	\section{Preliminary}\label{Sec:preliminary}
	In this paper, we denote by $\|\cdot\|$ the Euclidean norm of vectors in $\mathbb{R}^d$, $d\geq1$, and by $\langle \bm{x},\bm{y} \rangle = \sum_{i=1}^{d}x_iy_i$ the inner product of $\bm{x}=(x_1,x_2, ..., x_d)^{\top}$ and $\bm{y}=(y_1,y_2, ..., y_d)^{\top}$, where $\bm{x}^{\top}$ represents the transpose of a vector $\bm{x}$.  
	The norm of a matrix $Q$ is defined by $\|Q\| := \sqrt{\text{tr}(Q^\top Q)}$, where $\text{tr}(\cdot)$ represents the trace.  
	We recall a fundamental algebraic fact: if a symmetric matrix $Q\in\mathbb{R}^{d\times d}$ is strictly positive definite, then there exists a constant $c>0$ such that $\bm{x}^{\top}Q\bm{x}\geq c\|\bm{x}\|^2$ for all vectors $\bm{x}\in\mathbb{R}^{d}$. 
	We shall frequently use the following elementary estimate: for any $\kappa>0$, if 
	$0<\delta<\sqrt{\kappa}$, then
	\begin{equation}\label{Est:inequality}
		\frac{1}{2} \left(x^2+\kappa v^2\right)\leq \kappa v^2+x^2+\delta vx \leq \frac{3}{2}(x^2+\kappa v^2), \qquad \forall \, v, x \in \mathbb{R}. 
	\end{equation}

	Throughout this paper, we use $C$ to denote an unspecified positive and finite constant, which may vary from one line to another but is always independent of the discretization parameters. 
	Constants depending on
	certain parameters $a,b,...$ are written as $C(a,b,...)$.

	In order to introduce the notion of random attractor, we will briefly review those concepts in the theory of  RDS which are relevant for our case.  
	We refer the reader to \cite{arnoldRDS} for a more detailed and systematic treatment.  
	In the sequel, $\mathbb{T}=\mathbb{R}$ or $\mathbb{Z}$ denotes the time set. An RDS on the state space $\mathbb{R}^{d}$ with time $\mathbb{T}$ consists of the following two components:  
	
	(i) \textit{Model of the noise}: Let $(\Omega,\mathcal{F},\mathbb{P},\theta)$ be a metric dynamical system, i.e., $(\Omega,\mathcal{F},\mathbb{P})$ is a probability space, and $\theta$ is a flow of mappings $\{\theta_{t}\}_{t\in\mathbb{T}}$ on $\Omega$ (i.e., $\theta_{0}=\text{id}_{\Omega}$, $\theta_{t+s} = \theta_{t}\circ\theta_{s}$ for all $t, s\in\mathbb{T}$), which leaves the measure $\mathbb{P}$ invariant.  For simplicity, we assume that $\theta$ is ergodic.

	(ii) \textit{Model of the system perturbed by noise}: Let $\varphi$ be a cocycle over $\theta$, that is,  a measurable mapping: $\mathbb{T}\times\Omega\times\mathbb{R}^{d} \rightarrow \mathbb{R}^{d}$, $(t,\omega,x)\mapsto\varphi(t,\omega,x)$, such that $(t,x)\mapsto\varphi(t,\omega,x)$ is continuous for all $\omega\in\Omega$ and $\varphi$ satisfies the cocycle property: 
	\begin{equation*}
		\varphi(0,\omega,\cdot) = \text{id}_{\mathbb{R}^{d}}(\cdot), \quad \varphi(t+s,\omega,\cdot) = \varphi(t,\theta_{s}\omega,\cdot)\circ\varphi(s,\omega,\cdot), \quad \text{for all} \ t, s \in\mathbb{T} \ \text{and} \ \omega\in\Omega, 
	\end{equation*}
	where $\circ$ means the composition of mappings. 
	Then $\varphi$ is said to be a continuous RDS (resp. a discrete RDS) on the state space $\mathbb{R}^{d}$ over a metric dynamical system $(\Omega,\mathcal{F},\mathbb{P},\theta)$ with time $\mathbb{T}=\mathbb{R}$ (resp. $\mathbb{T}=\mathbb{Z}$). 
	
	Next, we introduce the concept of random attractor.  A random set $D:\Omega\rightarrow\mathcal{B}(\mathbb{R}^d)$ is called tempered if $D(\omega)\subset B_{\rho(\omega)}(0)$ for all $\omega\in\Omega$, where $B_{\rho(\omega)}(0)$ denotes a ball centered at $0\in\mathbb{R}^d$ with radius $\rho(\omega)$ and $\rho$ is a tempered random variable, i.e., $\lim_{t\rightarrow\infty}\frac{1}{t}\log^+\rho(\theta_{-t}\omega)=0$. We denote by $\mathcal{D}$ the set of all compact tempered sets. 
	Let $A\in\mathcal{D}$ be a random compact set satisfying $\varphi(t,\omega,A(\omega))=A(\theta_{t}\omega)$ for all $t>0$ and $\omega\in\Omega$. If for all $D\in\mathcal{D}$,  
	\begin{equation*}
		\lim\limits_{t\rightarrow\infty} d\left(\varphi\left(t,\theta_{-t}\omega,D(\theta_{-t}\omega)\right) , A(\omega)\right) = 0, \quad \text{for all} \ \omega \in \Omega, 
	\end{equation*}
	then $A$ is said to be a random attractor for $\varphi$ (see, e.g., \cite{Caraballo2012,Han2017,Wang2015}). Here $d(B_1,B_2)=\sup\{\inf\{\|x-y\|:y\in B_2\}: x\in B_1\}$ denotes the Hausdorff semi-distance between subsets $B_1$ and $B_2$ in $\mathbb{R}^d$. 
	The random attractor provides the geometric description of the asymptotic regime of RDS as $t\rightarrow\infty$.

	\subsection{Stationary measure for the autoparametric system}
	Since $0<\gamma<1$, the system \eqref{SDE} can be rewritten as  
	\begin{equation}\label{SDE1} 
		\ddot{\ell}(t) = f_1(\ell,\dot{\ell},\vartheta,\dot{\vartheta})  + \frac{\sigma\dot{W}(t)}{1-\gamma\sin^2\vartheta(t)} 
		\qquad \text{and} \qquad 
		\ddot{\vartheta}(t) =f_2(\ell,\dot{\ell},\vartheta,\dot{\vartheta}) + \frac{\sigma \sin \vartheta(t) \, \dot{W}(t)}{1 - \gamma \sin^2 \vartheta(t)},   
	\end{equation}
	where the nonlinear terms $f_1, f_2: \mathbb{R}^4\rightarrow\mathbb{R}$ are given by 
	\begin{equation}\label{nonlinear f}
		\begin{aligned}
			f_{1}(\ell,\dot{\ell},\vartheta,\dot{\vartheta}) &:= \frac{-2\zeta_1 \dot{\ell} - \kappa_{1} \ell + \gamma \dot{\vartheta}^2 \cos \vartheta - 2\gamma\zeta_2 \dot{\vartheta} \sin \vartheta - \gamma \kappa_{2} \sin^2 \vartheta}{1 - \gamma \sin^2 \vartheta}, \\[1mm]
			f_{2}(\ell,\dot{\ell},\vartheta,\dot{\vartheta}) &:= \frac{-2\zeta_2 \dot{\vartheta} - \kappa_{2} \sin \vartheta - 2\zeta_1 \dot{\ell} \sin \vartheta - \kappa_{1} \ell \sin \vartheta + \gamma \dot{\vartheta}^2 \sin \vartheta \cos \vartheta}{1 - \gamma \sin^2 \vartheta}. 
		\end{aligned}
	\end{equation} 
	
	The second order system \eqref{SDE1} can be reformulated as the first-order system of SDEs for the process $(\ell, \dot{\ell}, \vartheta, \dot{\vartheta})$ taking values in the space $\mathcal{N} := \mathbb{R}^2 \times \mathbb{R}/(2\pi\mathbb{Z}) \times \mathbb{R}$.  Hereafter, \(\vartheta\) is understood modulo \(2\pi\), and we always take \(\vartheta\in[0,2\pi)\) as the representative.
	We introduce
	\begin{equation*}
		\bm{v}(t) = \left(v_1(t), v_2(t), u_1(t), u_2(t)\right)^{\top} = \big( \ell(t), \dot{\ell}(t), \vartheta(t), \dot{\vartheta}(t) \big)^{\top}, 
	\end{equation*}
	and define two mappings $\bm{f}, \bm{g}: \mathcal{N} \rightarrow \mathbb{R}^4$ as
	\begin{equation}\label{nonlinear fg}
		\bm{f}(\bm{x}) := \left(x_2, \ f_1(\bm{x}), \ x_4, \ f_2(\bm{x})\right)^{\top} 
		\quad \text{and} \quad 
		\bm{g}(\bm{x}) := \frac{\sigma}{1 - \gamma \sin^2 x_3} 
		\left(0, \ 1, \  0, \ \sin x_3 \right)^{\top},  
	\end{equation} 
	where $\bm{x}=(x_1,x_2,x_3,x_4)^{\top}\in\mathcal{N}$ and $f_1$, $f_2$ are given by \eqref{nonlinear f}. 
	Then \eqref{SDE1} gives the following SDE driven by multiplicative noise: 
	\begin{equation}\label{system} 
		\rmd \bm{v}(t) = \bm{f}(\bm{v}(t)) \rmd t + \bm{g}(\bm{v}(t)) \rmd W(t). 
	\end{equation} 
	
	The single mode solution takes values in the invariant subset $\mathcal{N}_0 = \{(v_1, v_2, 0, 0) : v_1, v_2 \in \mathbb{R}\} \subset \mathcal{N}$, and the unstable single mode solution (corresponding to the inverted pendulum configuration) takes values in the invariant subset $\mathcal{N}_\pi = \{(v_1, v_2, \pi, 0) : v_1, v_2 \in \mathbb{R}\} \subset \mathcal{N}$. 
	Let $\nu_0$ and $\nu_\pi$ be the corresponding stationary measures for the process $\bm{v}(t)$ on $\mathcal{N}_0$ and $\mathcal{N}_\pi$, respectively. 
	If $\bm{v}(0) \in \mathcal{N}_0$ (resp. $\bm{v}(0) \in \mathcal{N}_\pi$), then $\bm{v}(t)$ converges to $\nu_0$ (resp. $\nu_{\pi}$) in distribution. 
	One can verify that $\nu_{0}=\tilde{\nu}_{0}\times\delta_{(0,0)}$, where $\tilde{\nu}_0$ is a Gaussian measure with mean $(0,0)^\top$ and covariance matrix $\frac{\sigma^2}{4\zeta_{1}} \left(\genfrac{}{}{0pt}{}{1/\kappa_{1} \ 0}{\ \, 0 \ \ \ 1}\right)$, and $\delta_{*}$ is the Dirac measure. 
	In \cite{Baxendale24SIAM}, the authors conjectured that when the Lyapunov exponent of the single mode solution $\lambda_{0}>0$, there is also a stationary measure $\hat{\nu}$ on $\mathcal{N}\backslash (\mathcal{N}_{0}\cup\mathcal{N}_{\pi})$. 
	However, the rigorous proof of the existence of $\hat{\nu}$ has not yet been established in the literature.
	In this paper, we mainly focus on the dynamics associated with the stationary measure $\nu_{0}$. 
	
	\subsection{Lyapunov exponents}
	Linearizing \eqref{system} along the solution $\bm{v}$ 
	yields the variational equation for $\tilde{\bm{v}} \in \mathbb{R}^{4}$: 
	\begin{equation*}
		\rmd \tilde{\bm{v}}(t)	= \left(D \bm{f}(\bm{v})\right) 
		\tilde{\bm{v}} \, \rmd t + \left(D \bm{g}(\bm{v})\right) 
		\tilde{\bm{v}} \, \rmd W(t), 
	\end{equation*} 
	where $D\bm{f}$ represents the Jacobian matrix of $\bm{f}$. 
	Define the top Lyapunov exponent of the four-dimensional system \eqref{system} by   
	\begin{equation*} 
		\lambda := \lim\limits_{t\rightarrow\infty} \frac{1}{t} \log \|\tilde{\bm{v}}(t)\|. 
	\end{equation*}
	Let $\nu$ be a stationary measure of \eqref{system}. Since $\|D\bm{f}(\bm{x})\| \leq C(\|\bm{x}\|^2+1)$, $\|D\bm{g}\| \leq C$, and $D((D\bm{g})\bm{g}) = 0$, it follows that $\|D\bm{f}(\cdot)\| \in L^1(\nu)$ (see, \cite[Theorem 2.1 (iii)]{Baxendale24SIAM}). By the multiplicative ergodic theorem (see \cite[Theorem 4.2.13]{arnoldRDS}), the Lyapunov exponent $\lambda$ exists for $\mathbb{P}$-a.s. $\omega\in\Omega$, for $\nu$-a.e. $\bm{v}(0)\in \mathbb{R}^4$, and for all $\tilde{\bm{v}}(0)\in \mathbb{R}^4$.    
	
	In what follows, we introduce the Lyapunov exponent $\lambda_{0}$ of the single mode solution.  
	Linearizing \eqref{system} along the single mode solution $(\eta(t), \dot{\eta}(t), 0, 0)$ gives a first order system 
	\begin{equation*}
		\rmd \tilde{\bm{v}}(t)
		= \left(D \bm{f}(\eta,\dot{\eta},0,0)\right) 
		\tilde{\bm{v}} \rmd t 
		+ \left(D \bm{g}(\eta,\dot{\eta},0,0)\right) 
		\tilde{\bm{v}} \rmd W(t), 
	\end{equation*} 
	which, by denoting $(\bar{\ell}, \dot{\bar{\ell}}, \beta, \dot{\beta})^{\top} = \tilde{\bm{v}}$, is equivalent to 
	\begin{equation}\label{variational SDE}
		\begin{aligned}
			\ddot{\bar{\ell}}(t) + 2 \zeta_1\dot{\bar{\ell}}(t) + \kappa_{1} \bar{\ell}(t) &= 0, \\
			\ddot{\beta}(t) + 2 \zeta_2\dot{\beta}(t) + \left( \kappa_{2}+2 \zeta_1\dot{\eta}(t) + \kappa_{1} \eta(t) \right) \beta(t) &= \sigma \beta(t) \dot{W}(t). 
		\end{aligned}
	\end{equation}
	In fact, \eqref{variational SDE} can be obtained alternatively by formally writing $\ell(t)=\eta(t)+\varepsilon\bar{\ell}(t)$ and $\vartheta(t) = 0+\varepsilon\beta(t)$ in \eqref{SDE} and letting $\varepsilon\rightarrow0$. This means that the stability of the single mode solution is determined by the long-time behavior of the linearized process $\{(\bar{\ell}(t),\dot{\bar{\ell}}(t),\beta(t),\dot{\beta}(t))\}_{t\geq0}$. The first equation in \eqref{variational SDE} is damped and unforced, hence $\|(\bar{\ell}(t),\dot{\bar{\ell}}(t))\|\rightarrow0$ as $t\rightarrow\infty$. Therefore, it suffices to consider the parametric excitation of $\{(\beta(t),\dot{\beta}(t))\}_{t\geq0}$ caused by the single mode vibration $\eta(t)$. We consider the long-time growth or decay rate for the process $(\beta,\dot{\beta})$ given by
	\begin{equation}\label{linearization}
		\begin{aligned}
			\ddot{\eta}(t) + 2 \zeta_1\dot{\eta}(t) + \kappa_{1} \eta(t) &= \sigma \dot{W}(t), \\
			\ddot{\beta}(t) + 2 \zeta_2\dot{\beta}(t) + \left( \kappa_{2}+2 \zeta_1\dot{\eta}(t) + \kappa_{1} \eta(t) \right) \beta(t) &= \sigma \beta(t) \dot{W}(t), 
		\end{aligned}
	\end{equation}
	and define the top Lyapunov exponent of the single mode solution by 
	\begin{equation}\label{lambda0}
		\lambda_{0} := \lim\limits_{t\rightarrow\infty} \frac{1}{t} \log \| (\beta(t),\dot{\beta}(t)) \|. 
	\end{equation}
	Recall that $\nu_{0}$ is the stationary measure for the process $\{\bm{v}(t)\}_{t\geq0}$ on $\mathcal{N}_{0}$. 
	The multiplicative ergodic theorem implies that $\lambda_{0}$ exists for $\mathbb{P}$-a.s. $\omega\in\Omega$, for $\nu_{0}$-a.e. $(\eta(0),\dot{\eta}(0),0,0)^{\top} \in \mathcal{N}_{0}$, and for all $(\bar{\ell}(0), \dot{\bar{\ell}}(0), \beta(0), \dot{\beta}(0))^{\top}\in \mathbb{R}^4$. 
	
	We next derive Furstenberg--Khasminskii's formula for $\lambda_{0}$. 
	Let $B(t) = (\beta(t),\dot{\beta}(t))^{\top}$. Then the second equation in \eqref{linearization} can be written as the following two-dimensional linear SDE  
	\begin{align*}
		\rmd B(t) = \begin{pmatrix}
			0 & 1 \\ -\kappa_{2}-2\zeta_1\dot{\eta}(t)-\kappa_{1} \eta(t) & -2\zeta_2 
		\end{pmatrix} B(t) \rmd t + \begin{pmatrix}
			0 & 0 \\ \sigma & 0 
		\end{pmatrix} B(t) \rmd W(t). 
	\end{align*}
	By introducing the polar coordinate $B(t) = \|B(t)\| (\cos\psi(t) , \sin\psi(t))^{\top}$ and applying It\^o's formula, we obtain 
	\begin{equation}\label{xi}
		\begin{aligned}
			\rmd \left(\log \|B(t)\|\right) &= Q(\eta(t),\dot{\eta}(t),\psi(t)) \rmd t + \sigma\sin\psi(t)\cos\psi(t) \rmd W(t), \\
			\rmd \psi(t) &= h(\eta(t),\dot{\eta}(t),\psi(t)) \rmd t + \sigma\cos^2\psi(t) \rmd W(t), 
		\end{aligned}
	\end{equation}
	where, by denoting $\tilde{\mathcal{N}} := \mathbb{R}^2 \times \mathbb{R}/(\pi\mathbb{Z})$, the functions $Q, h: \tilde{\mathcal{N}}\rightarrow\mathbb{R}$ are given by 
	\begin{equation}\label{Qh}
		\begin{aligned}
			Q(\eta,\dot{\eta},\psi) &= \left( 1-\kappa_{2}-2\zeta_1\dot{\eta}-\kappa_{1} \eta \right) \sin\psi\cos\psi - 2\zeta_2 \sin^2\psi + \frac{\sigma^2}{2} \cos^2\psi \cos 2\psi, \\
			h(\eta,\dot{\eta},\psi) &= -1 + \left( 1-\kappa_{2}-2\zeta_1\dot{\eta}-\kappa_{1} \eta \right) \cos^2 \psi - 2\zeta_2\sin\psi \cos\psi - \sigma^2 \sin\psi \cos^3\psi. 
		\end{aligned}
	\end{equation}
	By the definition of $\lambda_{0}$, i.e., \eqref{lambda0}, we have
	\begin{equation*}
		\lambda_{0} 
		= \lim\limits_{t\rightarrow\infty} \frac{1}{t}  \int_{0}^{t}
		Q(\eta(s),\dot{\eta}(s),\psi(s)) \rmd s 
		+ \lim\limits_{t\rightarrow\infty} \frac{1}{t} \int_{0}^{t} \sigma\sin\psi(s)\cos\psi(s) \rmd W(s). 
	\end{equation*}
	By \cite[Proposition 4.1]{Baxendale24SIAM}, the process $\{(\eta(t),\dot{\eta}(t),\psi(t))\}_{t\geq0}$ on $\tilde{\mathcal{N}}$ has a unique stationary measure $\nu$. 
	In addition, $\int_{0}^{t} \sigma\sin\psi(s)\cos\psi(s) \rmd W(s)$ is a continuous martingale with finite quadratic variation, which leads to 
	\begin{equation*}
		\mathbb{P}\left(\frac{1}{t} \int_{0}^{t} \sigma\sin\psi(s)\cos\psi(s) \rmd W(s) \rightarrow 0 \ \text{as} \ t \rightarrow \infty\right)=1. 
	\end{equation*}
	Therefore, the Lyapunov exponent $\lambda_{0}$ exists as an almost-sure limit for all initial conditions $(\eta(0),\dot{\eta}(0),\beta(0),\dot{\beta}(0))$ with $(\beta(0),\dot{\beta}(0))\neq0$ and satisfies the following Furstenberg--Khasminskii's formula: 
	\begin{equation*}
		\lambda_{0} 
		= \int_{\tilde{\mathcal{N}}} Q(\eta,\dot{\eta},\psi) \rmd\nu(\eta,\dot{\eta},\psi) =: \nu(Q). 
	\end{equation*}

	\subsection{Estimate of $\lambda_{0}$}
	In this subsection, we first review some known results regarding $\lambda_{0}$, and then derive an estimate based on the exponential ergodicity of the process $\{(\eta(t),\dot{\eta}(t),\psi(t))\}_{t\geq0}$. 
	
	In \cite{Baxendale24SIAM}, the authors derived asymptotic formulas for $\lambda_{0}$ corresponding to the regimes of small forcing and of small forcing combined with small damping, respectively. 
	In the following, we adapt the corresponding results to the notation and setting of this paper. 
	
	(i) \textit{Small forcing}: By replacing $\sigma\leadsto\varepsilon\sigma$ for small $\varepsilon$, the corresponding Lyapunov exponent $\bar{\lambda}_{0}(\varepsilon)$ satisfies that when $\zeta_{2}^2<\kappa_{2}$, 
	\begin{equation*}
		\bar{\lambda}_{0}(\varepsilon) = -\zeta_{2} +  \frac{2\varepsilon^2\sigma^2(\kappa_{2}-\zeta_{2}^2)}{\big(\kappa_{1}-2\sqrt{\kappa_{2}-\zeta_{2}^2}\big)^2 + 16\zeta_{1}^2(\kappa_{2}-\zeta_{2}^2)} + O(\varepsilon^4) \qquad \text{as } \varepsilon \rightarrow0. 
	\end{equation*}  
	
	(ii) \textit{Small forcing and small damping}: By replacing $\sigma\leadsto\varepsilon\sigma$ and $\zeta_{2}\leadsto\varepsilon^2\zeta_{2}$ for small $\varepsilon$, the corresponding Lyapunov exponent $\hat{\lambda}_{0}(\varepsilon)$ satisfies that 
	\begin{equation}\label{small damping}
		\hat{\lambda}_{0}(\varepsilon) = \varepsilon^2\left(-\zeta_{2} + \frac{2\kappa_{2}\sigma^2}{(\kappa_{1}-4\kappa_{2})^2+16\zeta_{1}^2\kappa_{2}} \right)+ O(\varepsilon^4) \qquad \text{as } \varepsilon \rightarrow0. 
	\end{equation} 
	
	From \eqref{small damping} we know that, in the limit of small forcing and damping, we can take a proper $\zeta_{2}$ such that $\hat{\lambda}_{0}(\varepsilon)>0$. In this case, the single mode solution is unstable, and the authors conjectured that there exists an additional stationary measure on $\mathcal{N}\backslash (\mathcal{N}_{0}\cup\mathcal{N}_{\pi})$ (see \cite[Conjecture 2.1]{Baxendale24SIAM}). 
	Although a rigorous proof is still lacking, one can intuitively conjecture that $\lambda_{0}>0$ implies exponential amplification of small perturbations to the single mode solution, reflecting the extreme sensitivity to initial conditions. This scenario will be illustrated numerically in Section \ref{Sec:experiment}.
	
	We next derive the estimate of $\lambda_{0}$ via the ergodic limit of the process $\{(\eta(t),\dot{\eta}(t),\psi(t))\}_{t\geq0}$. 
	Let $\mathcal{L}_{\eta\psi}$ be the infinitesimal generator of the process $\{(\eta(t),\dot{\eta}(t),\psi(t))\}_{t\geq0}$: 
	\begin{equation*}
		\mathcal{L}_{\eta\psi} := \dot{\eta} \partial_{\eta} - \left(\kappa_{1}\eta+2\zeta_1\dot{\eta}\right) \partial_{\dot{\eta}} + h(\eta,\dot{\eta},\psi) \partial_{\psi} + \frac{\sigma^2}{2} \partial_{\dot{\eta}\dot{\eta}} + \frac{\sigma^2\cos^4\psi}{2} \partial_{\psi\psi} + \sigma^2 \cos^2\psi\partial_{\dot{\eta}\psi}, 
	\end{equation*}
	where the function $h$ is given by \eqref{Qh}. 
	Define the Lyapunov function $\Gamma_{1}:\tilde{\mathcal{N}}\rightarrow[1,\infty)$ as 
	\begin{equation*}
		\Gamma_{1}(\eta,\dot{\eta},\psi) := \kappa_{1}\eta^2 + \dot{\eta}^2 + \psi^2 + \delta_{1}\eta\dot{\eta} + M, \quad \ \text{with } 0<\delta_{1}<  \frac{2\zeta_1\kappa_{1}}{\kappa_{1}+\zeta_1^2}, \ M\geq1, 
	\end{equation*}
	which satisfies  $\lim_{\|(\eta,\dot{\eta},\psi)\|\rightarrow\infty} \Gamma_{1}(\eta,\dot{\eta},\psi) = \infty$. 
	It follows from Young's inequality and \eqref{Est:inequality} that 
	\begin{align*}
		\mathcal{L}_{\eta\psi}\Gamma_{1} 
		&\leq \delta_{1}\dot{\eta}^2 - 4\zeta_1\dot{\eta}^2 
		- \frac{\delta_{1}\kappa_{1}}{2}\eta^2 + \frac{2\delta_{1}\zeta_1^2}{\kappa_{1}} \dot{\eta}^2 
		+ 4\pi (2\zeta_1|\dot{\eta}|+\kappa_{1}|\eta|) + C(\zeta_1,\zeta_2,\kappa_{1},\kappa_{2},\sigma)  \\
		&\leq -\frac{\delta_{1}}{4} (\kappa_{1}\eta^2+\dot{\eta}^2+\psi^2) + \delta_{1} \pi^2 + C 
		\leq -\frac{\delta_{1}}{6} \Gamma_{1}(\eta,\dot{\eta},\psi) + C(\delta_{1},M). 
	\end{align*}
	This means that the Lyapunov condition holds for the process $\{(\eta(t),\dot{\eta}(t),\psi(t))\}_{t\geq0}$. 
	Moreover, the minorization condition can be verified (see, e.g., \cite[Section 8]{Baxendale24SIAM}). Consequently, in view of $|Q|\leq \Gamma_{1}$ for some properly chosen $M\geq1$, there exist constants $a\in(0,1)$ and $b\in(0,\infty)$ such that 
	\begin{equation}\label{ergodic lambda0}
		\left| \mathbb{E}\left[Q(\eta(t),\dot{\eta}(t),\psi(t))\right] - \lambda_{0} \right| \leq b a^{t} \, \Gamma_{1}(\eta(0),\dot{\eta}(0),\psi(0)). 
	\end{equation}

	\section{Random attractor for continuous RDS}\label{Sec:attractor} 
	In this section, we show that the SDE \eqref{system} generates a continuous RDS possessing a random attractor. To this end, we consider a two-sided Brownian motion $\{W(t)\}_{t\in\mathbb{R}}$ and adopt the canonical Wiener space as the underlying metric dynamical system. Specifically, $\Omega = C_{0}(\mathbb{R}; \mathbb{R})$ is the space of all continuous functions $\omega: \mathbb{R} \rightarrow \mathbb{R}$ satisfying $\omega(0)=0$ endowed with the compact-open topology, $\mathcal{F}$ is the Borel $\sigma$-algebra of $\Omega$, and $\mathbb{P}$ is the Wiener measure. Define the shift mappings on $\Omega$ as $\theta_t \omega(\cdot) := \omega(t+\cdot) - \omega(t)$. Then $(\Omega,\mathcal{F},\mathbb{P},\theta)$ is an ergodic metric dynamical system that drives the SDE \eqref{system} with $W(t,\omega)=\omega(t)$. See \cite[Appendix A]{arnoldRDS} for more details. 
	
	Next, we show that the SDE \eqref{system} induces a continuous RDS through its conjugation to a random differential equation via a suitable transformation. 
	To this end, for $\alpha>0$, we introduce a linear equation $\rmd z = -\alpha z \rmd t + \sigma\rmd\omega(t)$, which admits a unique stationary solution given by $z(\theta_{t}\omega) := \sigma\int_{-\infty}^{0}  e^{\alpha s} \rmd (\theta_{t}\omega)(s)$. 
	A direct calculation shows that 
	\begin{equation}\label{OU}
		z(\theta_t \omega) = z(\theta_s \omega) - \alpha \int_{s}^{t} z(\theta_r \omega) \rmd r + \sigma(\omega(t)-\omega(s)), \qquad \forall \, t>s\in\mathbb{R}. 
	\end{equation}
	By \cite[Lemma 2.1]{Duan2003}, there exists a $\{\theta_t\}_{t\in\mathbb{R}}$-invariant subset $\Omega_1 \subset \Omega$ of full measure such that $\lim_{t\rightarrow\pm\infty} |\omega(t)| / |t|=0$ for all $\omega\in\Omega_{1}$. Moreover, for all $\omega \in \Omega_1$, there exists a positive random variable $C(\omega)$ such that 
	\begin{equation}\label{growth OU}
		|z(\theta_t \omega)|^2  \leq C(\omega) (1+|t|), \qquad \forall \, t \in \mathbb{R}. 
	\end{equation}
	By restricting the probability space to $\Omega_1$, we obtain a metric dynamical system $(\Omega_1,\mathcal{F}_1,\mathbb{P}_1,\theta)$, where  $\mathcal{F}_{1} := \{\Omega_1\cap A, A\in \mathcal{F}\}$ and the restricted probability measure $\mathbb{P}_1:=\mathbb{P}|_{\mathcal{F}_1}$. 
	Henceforth, we work on this metric dynamical system and for simplicity of notation, denote it again by $(\Omega,\mathcal{F},\mathbb{P},\theta)$.

	In order to establish the existence of a random attractor, we introduce the process $(v,x,u,y)$ under the transformation $\mathcal{T}:\mathcal{N}\rightarrow\mathcal{N}$, which is defined by 
	\begin{equation}\label{mathcalT}
		(v,x,u,y)^{\top} = \mathcal{T}(v_1,v_2, u_1,u_2) := \left( v_1, \ v_2 - \gamma u_2 \sin u_1, \ u_1, \ u_2 - v_2 \sin u_1 \right)^{\top}, 
	\end{equation}
	where $\bm{v}=(v_1, v_2, u_1, u_2)^{\top}$ satisfies the SDE \eqref{system}.  
	Then applying It\^o's formula to the transformation $\mathcal{T}$ yields the following equivalent SDE with additive noise 
	\begin{equation}\label{system vxuy}
		\left\{	
		\begin{aligned} 
			\rmd v &= \frac{x + \gamma y \sin u}{1 - \gamma \sin^2 u}  \rmd t, \\
			\rmd x &= \left( - \kappa_{1} v -   \frac{2\zeta_1(x + \gamma y\sin u)}{1 - \gamma \sin^2 u}  \right) \rmd t + \sigma \rmd W(t),  \\
			\rmd u &= \frac{x\sin u + y}{1 - \gamma \sin^2 u} \rmd t,   \\
			\rmd y &= \left( -\kappa_{2} \sin u -  \frac{2\zeta_2(x\sin u + y)}{1 - \gamma \sin^2 u} - \frac{(x + \gamma y \sin u)(x\sin u + y) \cos u}{(1 - \gamma \sin^2 u)^2} \right) \rmd t.  
		\end{aligned}
		\right. 
	\end{equation} 
	Define $\tilde{x}(t) := x(t)-z(\theta_{t}\omega)$, which, from \eqref{system vxuy}, satisfies 
	\begin{equation*}
		\rmd \tilde{x} = \left(  - \kappa_{1} v - \frac{2\zeta_1(x + \gamma y \sin u)}{1 - \gamma \sin^2 u} + \alpha  z(\theta_{t}\omega) \right) \rmd t. 
	\end{equation*}
	Hence, this transformation removes the stochastic differential term and converts \eqref{system} into a random differential equation, which provides a convenient framework for the construction of a continuous RDS and the analysis of its long-time dynamics. 
	
	\begin{thm}\label{th:attractor}
		For all $\omega\in\Omega$ and $\bm{x}\in\mathbb{R}^{4}$, the SDE \eqref{system} admits a unique	solution. The associated solution mapping $\varphi$ generates a continuous RDS. Moreover, $\varphi$ possesses a random attractor. 
	\end{thm}
	\begin{proof}
		According to \cite[Chapter 2]{arnoldRDS}, there exists a local RDS generated by $(v, \tilde{x}, u, y)$, which is denoted by $\tilde{\varphi}(t,\omega,\bm{x})$. As will be shown in \eqref{global RDS}, $\tilde{\varphi}(t,\omega,\bm{x})$ exists for all $t\geq0$, thereby forming a global RDS. Let $\bm{z}(\omega) = (0,z(\omega),0,0)^{\top}$ and introduce 
		\begin{equation*}
			\varphi(t,\omega,\bm{x}) = \mathcal{T}^{-1} \big(\tilde{\varphi}\big(t,\omega,\mathcal{T}\bm{x}-\bm{z}(\omega)\big) + \bm{z}(\theta_{t}\omega)\big), 
		\end{equation*}
		where the transformation $\mathcal{T}:\mathcal{N}\rightarrow\mathcal{N}$ is given by \eqref{mathcalT}. 
		Since $\mathcal{T}$ is a diffeomorphism, $\varphi$ and $\tilde{\varphi}$ are conjugate, which means that  $\varphi(t,\omega,\bm{x})$ forms a continuous RDS generated by \eqref{system}. Furthermore, if $\tilde{\varphi}$ possesses a random attractor $\tilde{A}(\omega)$, then $A(\omega) := \mathcal{T}^{-1}(\tilde{A}(\omega)+\bm{z}(\omega))$ is a random attractor for $\varphi$. 
		Therefore, in the remainder of this proof, we aim to demonstrate the existence of a random attractor for $\tilde{\varphi}$ by deriving its \textit{a priori} estimate. 
		
		For a constant $\delta_{2}\in(0,\sqrt{\kappa_{1}})$ chosen later, we define a function $E: \mathbb{R}^4 \rightarrow \mathbb{R}$ as 
		\begin{equation*}
			E := \frac{1}{2(1 - \gamma \sin^2 u)} \left( \tilde{x}^2 + \gamma y^2 + 2\gamma \tilde{x} y \sin u \right) + \frac{1}{2} \kappa_{1} v^2 + \gamma \kappa_{2} (1 - \cos u) + \delta_{2} v\tilde{x}. 
		\end{equation*}
		One can verify that there exist constants $c_{1}, c_{2}>0$ such that 
		\begin{equation}\label{Est:E}
			c_{1} \|(v,\tilde{x},y)\|^2 \leq E(v,\tilde{x},u,y) \leq c_{2}(1+\|(v,\tilde{x},y)\|^2). 
		\end{equation}
		By denoting $s = 1 - \gamma \sin^2 u$, it follows from the chain rule, \eqref{system vxuy}, and Young's inequality that 
		\begin{align*}
			\frac{\rmd E}{\rmd t} 
			&= \frac{1}{s^2} \left( -2\zeta_1 (\tilde{x}+\gamma y\sin u)^2 - 2\zeta_{2}\gamma(\tilde{x}\sin u+y)^2 \right) 
			- \delta_{2} \kappa_{1} v^2 
			+  \frac{\delta_{2}}{s}(\tilde{x}+\gamma y\sin u)(\tilde{x}-2\zeta_{1}v)  \\
			&\quad + \frac{1}{s} \left( \kappa_{1}vz + \alpha  z(\tilde{x}+\gamma y\sin u) + \gamma\kappa_{2}z\sin^2u \right) 
			- \frac{\gamma z\cos u}{s^3} \left( (\tilde{x}+z)\sin u + y \right) \left( \tilde{x}\sin u + y \right)\\
			&\quad - \frac{z}{s^2} \left( 2\zeta_{2}\gamma\sin u (y+\tilde{x}\sin u) + 2\zeta_{1}(\tilde{x}+\gamma y\sin u)\right) 
			+ \alpha  \delta_{2} vz 
			+\frac{\delta_{2}}{s}\left( - 2\zeta_{1}vz + \tilde{x}z  \right) \\
			&\leq \frac{1}{s^2} \left( -2\zeta_1 (\tilde{x}+\gamma y\sin u)^2 - 2\zeta_{2}\gamma(\tilde{x}\sin u+y)^2 \right) 
			- \delta_{2} \kappa_{1} v^2 
			+  \frac{\delta_{2}}{s}(\tilde{x}+\gamma y\sin u)(\tilde{x}-2\zeta_{1}v)  \\
			&\quad + \epsilon (\tilde{x}^2+v^2+y^2) + C(\epsilon,\gamma,\chi,\zeta_{1},\zeta_{2},\delta_{2},\alpha ) (z^4+1)
			+ \frac{2\gamma|z|}{(1-\gamma)^3}(\tilde{x}^2+y^2) \\
			&= (\epsilon- \delta_{2} \kappa_{1}) v^2  
			+ \left( \epsilon + \frac{\delta_{2}}{s} - \frac{2\zeta_1+2\gamma\zeta_2 \sin^2 u}{s^2} \right) \tilde{x}^2 
			+ \left( \epsilon -\frac{2\zeta_1 \gamma^2 \sin^2 u + 2\gamma\zeta_2}{s^2} \right) y^2  \\
			&\quad + \left( -\frac{4\zeta_1 \gamma \sin u + 4\gamma\zeta_2 \sin u}{s^2} + \frac{\delta_{2} \gamma \sin u}{s} \right) \tilde{x} y  
			+ \left( -\frac{2\zeta_1 \delta_{2}}{s} \right) v \tilde{x} + \left( -\frac{2\delta_{2}\zeta_1 \gamma \sin u}{s} \right) v y \\
			&\quad +  C(\epsilon,\gamma,\chi,\zeta_{1},\zeta_{2},\delta_{2},\alpha ) (z^4+1)
			+ \frac{2\gamma|z|}{(1-\gamma)^3}(\tilde{x}^2+y^2) \\
			&=: -b(v,\tilde{x},y) +C(\epsilon,\gamma,\chi,\zeta_{1},\zeta_{2},\delta_{2},\alpha ) (z^4+1)
			+ \frac{2\gamma|z|}{(1-\gamma)^3}(\tilde{x}^2+y^2),
		\end{align*}
		where $b(v,\tilde{x},y)$ represents the quadratic form with respect to $(v,\tilde{x},y)$ and is strictly positive definite for sufficiently small $\delta_{2}$ and $\epsilon$. Thus there exist constants $c_{3}, \hat{c}_{3}>0$ such that  
		\begin{align*}
			\frac{\rmd E}{\rmd t} &\leq \left(-c_{3}+ \frac{2\gamma|z|}{(1-\gamma)^3}\right) \|(v,\tilde{x},y)\|^2 + C(|z|^4+1) 
			\leq \left(-\frac{c_{3}}{c_{2}} + \frac{2\gamma|z|}{c_{1}(1-\gamma)^3}\right) E + \hat{c}_{3} (|z|^4+1), 
		\end{align*}
		where we have used \eqref{Est:E}. 
		By using Gronwall's inequality and denoting $\tilde{c}_{3} = \frac{c_{3}}{c_{2}}$, $\bar{c}_3 = \frac{2\gamma}{c_{1}(1-\gamma)^3}$, we obtain that for any $t>s\in\mathbb{R}$, 
		\begin{equation}\label{global RDS}
			\begin{aligned}
				E\left(v(t),\tilde{x}(t),y(t)\right)
				&\leq E\left(v(s),\tilde{x}(s),y(s)\right) e^{-\tilde{c}_{3}(t-s) + \bar{c}_{3} \int_{s}^{t}|z(\theta_{r}\omega)|\rmd r} \\
				&\quad + \hat{c}_{3}\int_{s}^{t}  \left(|z(\theta_{r}\omega)|^4+1\right) e^{-\tilde{c}_{3}(t-r) + \bar{c}_{3} \int_{r}^{t}|z(\theta_{\varsigma}\omega)|\rmd \varsigma} \rmd r. 
			\end{aligned}
		\end{equation}
		Recall that $\tilde{\varphi}(t,\omega,\bm{x})$ represents the solution $(v,\tilde{x},u,y)$ starting from $\bm{x}$. By using $|u|\leq2\pi$ and \eqref{Est:E}, it follows that $\forall\,t\geq0$, 
		\begin{align*} 
			&c_{1}\|\tilde{\varphi}(t,\theta_{-t}\omega,\bm{x})\|^2 \\ 
			\leq \, & c_{2}(1+\|\bm{x}\|^2 )e^{-\tilde{c}_{3}t +  \bar{c}_{3} \int_{-t}^{0}|z(\theta_{r}\omega)|\rmd r}
			+ \hat{c}_{3}\int_{-t}^{0}  \left(|z(\theta_{r}\omega)|^4+1\right) e^{\tilde{c}_{3}r + \bar{c}_{3} \int_{r}^{0}|z(\theta_{\varsigma}\omega)|\rmd \varsigma} \rmd r + 4\pi^2 c_{1}. 
		\end{align*}
		Since $z(\theta_{t}\omega)$ is ergodic, i.e., $\lim_{t\rightarrow \pm\infty}\frac{1}{t}\int_{0}^{t}|z(\theta_{r}\omega)|\rmd r 
		= \mathbb{E}[|z|]$ (see, e.g., \cite[Lemma 2.1]{Duan2003}), for any $\varepsilon>0$, there exists a finite random variable $T_0 = T_0(\omega,\varepsilon)>0$ such that 
		\begin{equation}\label{ergodic z}
			\left| \frac{1}{t}\int_{0}^{t}|z(\theta_{r}\omega)|\rmd r 
			- \mathbb{E}[|z|] \right| < \varepsilon, \qquad \forall \, |t|>T_0. 
		\end{equation}
		It follows from $\mathbb{E}[|z|]\leq (\mathbb{E}[|z|^2])^{\frac{1}{2}} =  \frac{\sigma}{\sqrt{2\alpha}}$ that when $\varepsilon\leq\frac{\sigma}{\sqrt{2\alpha}}$ and $\alpha \geq \frac{8\sigma^2\bar{c}_{3}^2}{\tilde{c}_{3}^2}$, 
		\begin{align*}
			e^{-\tilde{c}_{3}t + \bar{c}_{3} \int_{-t}^{0}|z(\theta_{r}\omega)|\rmd r} \leq e^{-\tilde{c}_{3}t + \bar{c}_3(\mathbb{E}[|z|]+\varepsilon)t} < e^{-\frac{\tilde{c}_{3}}{2}t}, \qquad \forall \, t>T_0.  
		\end{align*}
		Thus for any tempered set $D\in\mathcal{D}$ that satisfies $D(\omega)\subset B_{\rho_D(\omega)}(0)$ with a tempered random variable $\rho_D$, we know that when $t>T_0$, 
		\begin{equation*}
			\sup_{\bm x\in D(\theta_{-t}\omega)} \frac{c_{2}}{c_1}(1+\|\bm{x}\|^2 )e^{-\tilde{c}_{3}t +  \bar{c}_{3} \int_{-t}^{0}|z(\theta_{r}\omega)|\rmd r} 
			\leq \frac{c_{2}}{c_1}(1+|\rho_D(\theta_{-t}\omega)|^2) e^{-\frac{\tilde{c}_{3}}{2}t}, \qquad \mathbb{P}\text{-a.s.}
		\end{equation*} 
		Since $\rho_D$ is tempered, there exists a $T_1 = T_1(\omega,D)>0$ such that $\log^+\rho_D(\theta_{-t}\omega) \leq \frac{\tilde{c}_3}{8}t$ for all $t>T_1$. 
		It follows that when $t>\max\{T_0,T_1\}$, 
		\begin{equation}\label{absorb time}
			\sup_{\bm x\in D(\theta_{-t}\omega)}
			\frac{c_{2}}{c_1}(1+\|\bm{x}\|^2 )e^{-\tilde{c}_{3}t +  \bar{c}_{3} \int_{-t}^{0}|z(\theta_{r}\omega)|\rmd r}
			\leq 
			\frac{c_2}{c_1}\big(1+e^{\frac{\tilde{c}_{3}}{4}t}\big) e^{-\frac{\tilde{c}_{3}}{2}t}
			\leq \frac{2c_2}{c_1}e^{-\frac{\tilde{c}_{3}}{4}t}. 
		\end{equation}
		Define 
		\begin{equation*}
			K(\omega) := \left\{ \bm{x}\in\mathbb{R}^4: \ \|\bm{x}\|^2 \leq  
			1+ R(\omega)+4\pi^2  \right\}, 
		\end{equation*}
		where $R$ is a finite random variable given by 
		\begin{equation*}
			R(\omega) = \frac{\hat{c}_{3}}{c_1}\int_{-\infty}^{0}  \left(|z(\theta_{r}\omega)|^4+1\right) e^{\tilde{c}_{3}r + \bar{c}_{3} \int_{r}^{0}|z(\theta_{\varsigma}\omega)|\rmd \varsigma} \rmd r. 
		\end{equation*}
		From \eqref{absorb time} one can derive that for a.s. $\omega\in\Omega$ and for any $D\in\mathcal{D}$, there exists a $T_D(\omega)=\max\{T_0,T_1,(4\log^+(2c_2/c_1))/\tilde{c}_3\}$ such that
		\begin{equation*}
			\tilde\varphi\big(t,\theta_{-t}\omega, D(\theta_{-t}\omega)\big)\subset K(\omega),
			\qquad \forall\, t>T_D(\omega). 
		\end{equation*}
		Therefore, it remains to show that $K(\omega)$ is tempered, namely, $\lim_{t\rightarrow\infty}\frac{1}{t}\log^+K(\theta_{-t}\omega)=0$. 
		
		In fact, by the change of variables $s=-t+r$ and $\xi=-t+\varsigma$, 
		\begin{align*}
			R(\theta_{-t}\omega) 
			&= \frac{\hat{c}_{3}}{c_1}\int_{-\infty}^{-t}  \left(|z(\theta_{s}\omega)|^4+1\right) e^{\tilde{c}_{3}(s+t) + \bar{c}_{3} \int_{s}^{-t}|z(\theta_{\xi}\omega)|\rmd \xi} \rmd s. 
		\end{align*}
		Using \eqref{ergodic z}, when $\varepsilon\leq\frac{\sigma}{\sqrt{2\alpha}}$ and $\alpha \geq \frac{8\sigma^2\bar{c}_{3}^2}{\tilde{c}_{3}^2}$, we obtain that for any $t>T_0$ and $s<-t$, 
		\begin{align*}
			\int_{s}^{-t}|z(\theta_{\xi}\omega)|\rmd \xi 
			= \left(\int_{s}^{0} - \int_{-t}^{0}\right) |z(\theta_{\xi}\omega)|\rmd \xi 
			\leq  -(\mathbb{E}[|z|]+\varepsilon)(s+t) + 2\varepsilon t 
			\leq -\frac{\tilde{c}_3}{2\bar{c}_3}(s+t) 
			+ 2\varepsilon t. 
		\end{align*} 
		In view of \eqref{growth OU}, for any $t>T_0$, we have 
		\begin{equation*}
			R(\theta_{-t}\omega) 
			\leq \frac{\hat{c}_{3}}{c_1} \int_{-\infty}^{-t}   C(\omega) (|s|^2+1) e^{\frac{\tilde{c}_{3}}{2}(s+t)+2\varepsilon \bar{c}_3 t} \rmd s 
			\leq C(\omega) e^{2\varepsilon \bar{c}_3 t} (1+t^2), 
		\end{equation*}
		which means that $\lim_{t\to\infty}\frac{1}{t}\log^+ R(\theta_{-t}\omega)\leq 2\varepsilon\bar{c}_3$ holds   $\forall \, \varepsilon\in(0,\frac{\sigma}{\sqrt{2\alpha}})$.   
		Since $\varepsilon$ is arbitrary, we deduce that $R$ is a tempered random variable, and thus $K(\omega)$ is an absorbing set.  
		Consequently, by \cite[Theorem B.2]{Doan2018}, there exists a random attractor $\tilde{A}(\omega)$ for $\tilde{\varphi}$, and accordingly, $A(\omega) = \mathcal{T}^{-1}(\tilde{A}(\omega)+\bm{z}(\omega))$ is a random attractor for the continuous RDS $\varphi$. 
	\end{proof}

	\section{Numerical method and random attractor for discrete RDS}\label{Sec:numerical attractor}
	In this section we construct the numerical discretization of the SDE \eqref{system} that can preserve two key dynamical features, namely the random attractor and the almost-sure stability of the single mode solution,  of the continuous system. 
	
	A key step in the analysis is to establish uniform moment bounds for the numerical solutions. However, the Lyapunov structure of \eqref{system} is difficult to transfer directly to the discrete setting due to the absence of an It\^o's formula at the discrete level. To overcome this difficulty, we introduce the transformed variables $(\bar{v}_1, \bar{v}_2, \bar{u}_1, \bar{u}_2)$ given by 
	\begin{equation*}
		\bar{v}_1=v_1, \quad \bar{v}_2=v_2-\gamma u_2 \sin u_1, \quad \bar{u}_1=u_1, \quad \bar{u}_2 = u_2 \sqrt{1-\gamma\sin^2 u_1}. 
	\end{equation*} 
	Applying It\^o's formula to \eqref{system} and denoting $\bar{s}=(1-\gamma\sin^2 \bar{u}_1)^{-\frac{1}{2}}$, we obtain   
	\begin{equation}\label{system barvxuy}
		\left\{\begin{aligned}
			\rmd \bar{v}_1 &= \left( \bar{v}_2 + \gamma\bar{s} \, \bar{u}_2\sin \bar{u}_1 \right) \rmd t, \\
			\rmd \bar{v}_2 &= -\left( \kappa_{1} \bar{v}_1 + 2\zeta_1 \bar{v}_2 + 2\zeta_1 \gamma \bar{s} \, \bar{u}_2\sin \bar{u}_1 \right) \rmd t + \sigma \rmd W(t), \\
			\rmd \bar{u}_1 &= \bar{s} \, \bar{u}_2 \rmd t, \\
			\rmd \bar{u}_2 &= -2\bar{s}^2 (\zeta_2 + \zeta_1\gamma \sin^2 \bar{u}_1)  \bar{u}_2 \rmd t
			-  \bar{s}(\kappa_{2} + 2\zeta_1 \bar{v}_2 +\kappa_{1} \bar{v}_1) \sin \bar{u}_1 \rmd t
			+ \sigma \bar{s} \, \sin \bar{u}_1 \rmd W(t). 
		\end{aligned}\right.
	\end{equation}  
	
	\subsection{Numerical method}
	Motivated by the discrete Lyapunov structure of \eqref{system barvxuy}, we construct a numerical approximation for this transformed system.
	For the fixed step size $\tau>0$ and grid points $t_{k}=k\tau$, $k\geq0$, we propose the following numerical method: $(\bar{V}_{0}, \bar{\mathcal{V}}_{0}, \bar{U}_{0}, \bar{\mathcal{U}}_{0}) = ( \bar{v}_1(0),  \bar{v}_2(0), \bar{u}_1(0), \bar{u}_2(0) )$, and for $k\geq0$, 
	\begin{equation}\label{BEM} 
		\begin{aligned}
			\bar{V}_{k+1} &= \bar{V}_{k} + \tau \bar{\mathcal{V}}_{k+1} + \tau \gamma\bar{S}_{k}\bar{\mathcal{U}}_{k+1} \sin \bar{U}_{k}, \\
			\bar{\mathcal{V}}_{k+1} &= \bar{\mathcal{V}}_{k} - \tau\kappa_{1} \bar{V}_{k+1} 
			- 2\tau\zeta_1 \bar{\mathcal{V}}_{k+1} - 2\tau\zeta_1 \gamma \bar{S}_{k} \bar{\mathcal{U}}_{k+1} \sin \bar{U}_{k} + \sigma \Delta W_{k}, \\
			\bar{U}_{k+1} &= \bar{U}_{k} + \tau \bar{S}_{k} \bar{\mathcal{U}}_{k+1}, \\
			\bar{\mathcal{U}}_{k+1} &= \bar{\mathcal{U}}_{k} - 2\tau\bar{S}_{k}^2 \left(\zeta_2+\zeta_1 \gamma \sin^2 \bar{U}_{k}\right)\bar{\mathcal{U}}_{k+1}
			- \tau \bar{S}_{k} \left(\kappa_{2} + 2\zeta_1 \bar{\mathcal{V}}_{k+1} + \kappa_{1} \bar{V}_{k+1}\right) \sin \bar{U}_{k} \\
			&\qquad  + \sigma \bar{S}_{k} \sin\bar{U}_{k}\Delta W_{k}, 
		\end{aligned}
	\end{equation}
	where $\bar{S}_{k} = \left(1-\gamma\sin^2 \bar{U}_{k}\right)^{-\frac{1}{2}}$ and $\Delta W_{k} = W(t_{k+1})-W(t_{k})$ for $k\geq0$ represent the Brownian increments. 
	We denote $\bar{\bm{V}}_k = (\bar{V}_{k}, \bar{\mathcal{V}}_{k}, \bar{U}_{k}, \bar{\mathcal{U}}_{k})^{\top}$ and define $\bm{V}_k$ by 
	\begin{equation}\label{Vk}
		\bm{V}_{k} := (V_{k}, \mathcal{V}_{k}, U_{k}, \mathcal{U}_{k})^{\top} := \mathcal{J}(\bar{\bm{V}}_k), 
	\end{equation}
	where the transformation $\mathcal{J}:\mathbb{R}^4\rightarrow\mathbb{R}^4$ is given by  
	\begin{equation}\label{mathcalT tau}
		\mathcal{J}(x_1,x_2,x_3,x_4) = \left(x_1, \  x_2+\gamma x_4\sin x_3(1-\gamma\sin^2x_3)^{-\frac{1}{2}}, \ x_3, \  x_4(1-\gamma\sin^2x_3)^{-\frac{1}{2}} \right)^{\top}. 
	\end{equation} 
	Since $\mathcal{J}$ is well defined and invertible, one can verify that $\bm{V}_k$ satisfies 
	\begin{equation}\label{BEM2}
		\bm{V}_{k+1} = \bm{V}_k + \tau \bm{f}\left(V_{k+1},\mathcal{V}_{k+1},U_{k},\mathcal{U}_{k+1}\right) + \bm{g}(\bm{V}_k) \Delta W_{k} + \bm{R}_{k}, 
	\end{equation} 
	where $\bm{f}$ and $\bm{g}$ are defined by \eqref{nonlinear fg}, and $\|\bm{R}_{k}\|  \leq C\tau (\tau+|\Delta W_{k}|) (\|\bm{V}_k\|^2+1)$, $\forall\,k\geq0$. 
	Therefore, \eqref{BEM2} provides a consistent approximation of the SDE \eqref{system} with the remainder term $ \bm{R}_{k}$ being of higher order. 
	
	\begin{remark}\label{rmk1}
		It can be verified that $\left(\eta_k,\bar{\eta}_k,0,0\right)^{\top}$ is a single mode solution of both \eqref{BEM2} and \eqref{BEM}. 
		Moreover, since the Jacobian matrix  $D\mathcal{J}(\eta_k,\bar{\eta}_k,0,0)=I_{4\times4}$, 
		the linearizations of \eqref{BEM2} and \eqref{BEM} coincide at the single mode solution. Hence, in the subsequent analysis of the numerical Lyapunov exponent, we focus on \eqref{BEM} and its linearization rather than on \eqref{BEM2}. 
		
		In practical computations, the numerical solution $\bar{\bm{V}}_k$ is first obtained from \eqref{BEM}, and the approximation $\bm{V}_k$ of the original SDE \eqref{system} is then recovered through the transformation $\mathcal{J}$ defined in \eqref{Vk}. 
	\end{remark}
	
	We next show that the numerical method \eqref{BEM} is uniquely solvable.  
	For $x\in\mathbb{R}$, by denoting $s=(1-\gamma\sin^2 x)^{-1/2}$ and introducing the  matrix 
	\begin{align*}
		\bm{J}(x) = \begin{pmatrix}
			0 & 1 & 0  & \gamma s \sin x \\[1mm]
			-\kappa_{1} & -2\zeta_1 &  0 & -2\zeta_1\gamma s \sin x \\[1mm]
			0 & 0 & 0 & s \\[1mm]
			-\kappa_{1}s\sin x &  -2\zeta_1 s\sin x & 0 &  -2s^2 (\zeta_2 + \zeta_1 \gamma \sin^2 x) 
		\end{pmatrix}, 
	\end{align*}
	the implicit equation \eqref{BEM} can be written as  
	\begin{equation}\label{compact BEM}
		(I_{4\times4}-\tau \bm{J}(\bar{U}_{k}))\bar{\bm{V}}_{k+1}
		= \bar{\bm{V}}_{k} + \tau \bm{J}_1(\bar{\bm{V}}_{k}) + \bm{J}_2(\bar{\bm{V}}_{k}) \Delta W_{k}, 
	\end{equation}
	where $\bm{J}_1(\bar{\bm{V}}_{k}) = (0,0,0,-\kappa_{2}\bar{S}_{k}\sin \bar{U}_{k} )^{\top}$ and $\bm{J}_2(\bar{\bm{V}}_{k}) = (0,\sigma, 0, \sigma \bar{S}_{k}\sin \bar{U}_{k} )^{\top}$. 
	For any $k\geq0$, the determinant of $I_{4\times4}-\tau \bm{J}(\bar{U}_{k})$ is 
	\begin{equation*}
		\textsf{det}(I_{4\times4}-\tau \bm{J}(\bar{U}_{k}))  = (1+2\tau\zeta_{2}\bar{S}_{k}^2) (1+2\tau\zeta_{1}+\tau^2\kappa_{1})
		+ \gamma\bar{S}_{k}^2\sin^2\bar{U}_k  (2\tau\zeta_1+\tau^2\kappa_{1}) > 0, 
	\end{equation*}
	which yields the unique solvability of the numerical method \eqref{BEM}.

	In the remainder of this subsection, we present the moment boundedness of numerical solutions. 
	For $\delta_{3}\in(0,\sqrt{\kappa_{1}})$, denote 
	\begin{align*}
		&\tilde{\mathbb{V}}_{k} = \frac{1}{2} \left( \bar{\mathcal{V}}_{k} + \sigma \Delta W_{k} \right)^2 + \frac{\gamma}{2} \left( \bar{\mathcal{U}}_{k} + \sigma \bar{S}_{k}\sin \bar{U}_{k} \Delta W_{k} \right)^2 + \frac{\kappa_{1}}{2}  \bar{V}_k^{2} 
		+ \delta_{3} \bar{V}_{k} \left( \bar{\mathcal{V}}_{k} + \sigma \Delta W_{k} \right), \\
		&\mathbb{V}_{k+1} =  \left(\frac{1}{2} \bar{\mathcal{V}}_{k+1}^2 + \frac{\gamma}{2} \bar{\mathcal{U}}_{k+1}^2 + \frac{\kappa_{1}}{2}  \bar{V}_{k+1}^{2} +\delta_{3} \bar{V}_{k+1} \bar{\mathcal{V}}_{k+1}  \right). 
	\end{align*} 
	By substituting \eqref{BEM} into $\mathbb{V}_{k+1}$ and using Taylor's expansions, we obtain
	\begin{align*}
		& \tilde{\mathbb{V}}_{k} - \mathbb{V}_{k+1} \\ 
		\geq \, & (2\tau \zeta_1-\tau\delta_{3}) \bar{\mathcal{V}}_{k+1}^2 + 2\tau\bar{S}_{k}^2(\zeta_2\gamma+\zeta_1\gamma^2\sin^2 \bar{U}_{k}) \bar{\mathcal{U}}_{k+1}^2
		+ \tau\gamma (4\zeta_1-\delta_{3})\bar{S}_{k}\sin \bar{U}_{k} \bar{\mathcal{V}}_{k+1} \bar{\mathcal{U}}_{k+1} 
		\\
		& + \tau \delta_{3}\kappa_{1} \bar{V}_{k+1}^2  + 2\tau\delta_{3}\zeta_1\bar{V}_{k+1}\bar{\mathcal{V}}_{k+1} 
		+ 2\tau\delta_{3}\zeta_1 \gamma\bar{S}_{k}\sin\bar{U}_{k}\bar{V}_{k+1}\bar{\mathcal{U}}_{k+1} + \tau \gamma \kappa_{2}\bar{S}_{k}\sin \bar{U}_{k} \bar{\mathcal{U}}_{k+1} \\
		=: \, & b^\tau(\bar{\mathcal{V}}_{k+1}, \bar{\mathcal{U}}_{k+1}, \bar{V}_{k+1}) + \tau \gamma \kappa_{2}\bar{S}_{k}\sin \bar{U}_{k} \bar{\mathcal{U}}_{k+1}. 
	\end{align*} 
	Since $|\sin \bar{U}_k|\leq 1$ and $1\leq \bar{S}_k\leq (1-\gamma)^{-\frac{1}{2}}$, one can choose $\delta_3$ sufficiently small such that $b^{\tau}$ is a strictly positive definite quadratic form.  
	Thus there exists a constant $c_{4}>0$ such that 
	\begin{equation}\label{mathbbV}
		\tilde{\mathbb{V}}_{k} - \mathbb{V}_{k+1}
		\geq c_{4} \tau  
		\left(\kappa_{1}\bar{V}_{k+1}^2 + \bar{\mathcal{V}}_{k+1}^2 + \gamma\bar{\mathcal{U}}_{k+1}^2 \right) - C\tau
		\geq \frac{2c_{4} \tau}{3} 
		\mathbb{V}_{k+1} - C\tau, 
	\end{equation} 
	where \eqref{Est:inequality} is used. By \eqref{mathbbV}, $|\sin \bar{U}_k|\leq 1$, and $1\leq \bar{S}_k\leq (1-\gamma)^{-\frac{1}{2}}$, we obtain  
	\begin{align*} 
		\left(1+\frac{2c_{4} \tau}{3}\right) \mathbb{E}[\mathbb{V}_{k+1}] 
		&\leq \mathbb{E}[\tilde{\mathbb{V}}_{k}] + C\tau 
		= \mathbb{E}[\mathbb{V}_{k}] + \mathbb{E}\left[ \frac{1}{2}\sigma^2\Delta W_k^2 + \frac{\gamma}{2} \sigma^2\bar{S}_k^2\sin^2\bar{U}_k\Delta W_k^2 \right] + C\tau \\
		&= \mathbb{E}[\mathbb{V}_{k}] 
		+ \frac{1}{2}\sigma^2\tau + \frac{\gamma}{2} \sigma^2\bar{S}_k^2\sin^2\bar{U}_k \, \tau + C\tau
		\leq \mathbb{E}[\mathbb{V}_{k}] + C\tau. 
	\end{align*} 
	It follows that $\sup_{k\geq0}\mathbb{E}[\mathbb{V}_{k}] \leq C$, which, together with $|\bar{U}_{k}|\leq2\pi$ and \eqref{Est:inequality}, leads to $\sup_{k\geq0}\mathbb{E}\left[\|\bar{\bm{V}}_{k}\|^{2}\right] \leq C$.  
	Furthermore, by \eqref{Vk}, we obtain 
	\begin{equation*} 
		\sup\limits_{k\geq0} \mathbb{E}[\|\bm{V}_{k}\|^{2}] = \sup\limits_{k\geq0}\mathbb{E}[\|\mathcal{J}(\bar{\bm{V}}_{k})\|^{2}] 
		\leq C(\gamma) \sup\limits_{k\geq0}\mathbb{E}[\|\bar{\bm{V}}_{k}\|^{2}] 
		\leq C. 
	\end{equation*}

	\subsection{Random attractor for discrete RDS}
	Similar to the continuous setting, we consider a two-sided Brownian motion in order to construct the discrete RDS associated with the numerical method. Recall that the numerical method \eqref{BEM} can be written as 
	\begin{equation*}
		\bar{\bm{V}}_{k+1}
		= \bar{\bm{V}}_k 
		+ \tau \bm{J}(\bar{U}_{k})\bar{\bm{V}}_{k+1}
		+ \tau \bm{J}_1(\bar{\bm{V}}_{k}) + \bm{J}_2(\bar{\bm{V}}_{k}) \Delta W_{k}. 
	\end{equation*}
	For any $\bm{y}=(y_1,y_2,y_3,y_4)^{\top}\in\mathbb{R}^4$, define the mapping $\bm{\Psi}:\Omega\times \mathbb{R}^4\rightarrow \mathbb{R}^4$ by letting $\bm{\Psi}(\omega,\bm{y})$ be the unique solution of the following equation 
	\begin{equation*}
		\bm{\Psi}(\omega,\bm{y}) = \bm{y} + \tau \bm{J}(y_3) \bm{\Psi}(\omega,\bm{y}) 
		+ \tau \bm{J}_{1}(\bm{y})
		+ \bm{J}_2(\bm{y}) \left(\omega(\tau)-\omega(0)\right), \quad \forall \, \omega \in \Omega. 
	\end{equation*}
	Define the mappings $\bar{\varphi}^{\tau}: \mathbb{Z}_0^+\times\Omega\times\mathbb{R}^4\rightarrow\mathbb{R}^4$ as 
	\begin{align*}
		\bar{\varphi}^{\tau}(k,\omega,\bm{y}) := \bm{\Psi}\left(\theta_{(k-1)\tau}\omega,\bar{\varphi}^{\tau}(k-1,\omega,\bm{y})\right), \quad k\geq2
	\end{align*}
	with 
	\begin{align*}
		\bar{\varphi}^{\tau}(0,\omega,\bm{y}) = \bm{y}, \quad
		\bar{\varphi}^{\tau}(1,\omega,\bm{y}) = \bm{\Psi}(\omega,\bm{y}). 
	\end{align*}
	By construction, $\bar{\varphi}^{\tau}(k,\omega,\bm{y})$ coincides with the numerical solution
	$\bar{\bm{V}}_k$ starting from $\bar{\bm{V}}_0=\bm{y}$.
	Moreover, a direct verification shows that $\bar{\varphi}^\tau$ satisfies the cocycle property over
	$(\Omega,\mathcal{F},\mathbb{P},\theta)$. 
	For any $\bm{x}\in\mathbb{R}^4$, we define 
	\begin{equation}\label{conjugate1}
		\varphi^{\tau}(k,\omega,\bm{x}) := \mathcal{J}\left(\bar{\varphi}^{\tau}(k,\omega,\mathcal{J}^{-1}(\bm{x}))\right), \quad k\geq2
	\end{equation}
	with 
	\begin{align*}
		\varphi^{\tau}(0,\omega,\bm{x}) = \bm{x}, \quad
		\varphi^{\tau}(1,\omega,\bm{x}) = \mathcal{J}\left(\bar{\varphi}^{\tau}(1,\omega,\mathcal{J}^{-1}(\bm{x}))\right), 
	\end{align*}
	where $\mathcal{J}$ is given by \eqref{mathcalT tau}. 
	Since $\mathcal{J}$ is a diffeomorphism,  $\varphi^{\tau}$ and $\bar{\varphi}^{\tau}$ are conjugate.  
	Consequently, $\varphi^\tau$ is a discrete RDS on $\mathbb{R}^4$ over $(\Omega,\mathcal F,\mathbb P,\theta)$. 
	
	\begin{remark}
		Although the driving metric dynamical system $\{\theta_t\}_{t\in\mathbb{R}}$ is two-sided, the discrete RDS generated by the numerical approximation is generally defined only for $k\in\mathbb{Z}_0^+$. 
		The reason is that the one-step map $\bm{\Psi}(\omega,\cdot)$ may not be invertible for all $\omega$. 
		Indeed, extending the cocycle to negative times would require the one-step map $\bm{\Psi}(\omega,\cdot)$ to be invertible for all $\omega$ in order to define the backward iterates, which fails in general. 
		
		This absence of a backward-time discrete RDS does not affect the asymptotic analysis in the present work: the random attractors and Lyapunov exponents considered here are determined by the forward-time behavior as $k\rightarrow+\infty$.  The two-sidedness of  $\{\theta_t\}_{t\in\mathbb{R}}$ remains essential for formulating the random attractor in the pullback sense (see the definition of random attractor in Section \ref{Sec:preliminary}). 
		Therefore, while the underlying noise model is two-sided, only a forward discrete RDS is needed in the present analysis.
	\end{remark}
	
	Before proving the existence of the random attractor for the discrete RDS $\varphi^\tau$, we present the following lemma regarding the estimates of $z(\theta_t\omega)$. 
	
	\begin{lemma}\label{le:tempered}
		For any $\delta\in(0,\frac{1}{2})$, there exist two positive random variables $L_z$ and $L_\delta$ such that 
		\begin{equation*}
			|z(\theta_t \omega)| \leq L_z(\omega) (1+|t|^{\frac{1}{2}}) \quad \text{and} \quad 
			|z(\theta_{t+s} \omega)-z(\theta_t \omega)| \leq L_\delta(\theta_t \omega) |s|^{\delta}, \qquad \forall \, t \in \mathbb{R}, s\in[0,1]. 
		\end{equation*}
		In addition, for any $\varepsilon>0$, 
		\begin{equation*}
			\lim\limits_{t\rightarrow\pm\infty} e^{-\varepsilon |t|} \left(L_z(\theta_t \omega)  + L_\delta(\theta_t \omega)\right) = 0.  
		\end{equation*} 
	\end{lemma}

	\begin{proof}
		Define 
		\begin{equation*}
			L_z(\omega) := \sup_{s\in\mathbb{R}} \frac{|z(\theta_s\omega)|}{\ 1+|s|^{\frac{1}{2}} \ }. 
		\end{equation*}
		By \eqref{growth OU}, $L_z(\omega)<\infty$, $\mathbb{P}$-a.s., and hence $|z(\theta_t\omega)| \leq L_z(\omega) (1+|t|^{\frac{1}{2}})$. Moreover, using the inequality $1+|t+s|^{1/2}\leq(1+|t|^{1/2})(1+|s|^{1/2})$, we obtain that for all $\varepsilon>0$, 
		\begin{equation*}
			e^{-\varepsilon |t|} L_z(\theta_t \omega) 
			\leq \sup_{s\in\mathbb{R}} \frac{C(\omega)e^{-\varepsilon |t|} (1+|t+s|^{\frac{1}{2}})}{1+|s|^{\frac{1}{2}}} 
			\leq C(\omega)e^{-\varepsilon |t|} (1+|t|^{\frac{1}{2}}) 
			\rightarrow 0, \qquad \text{as } \ |t|\rightarrow\infty. 
		\end{equation*}
		
		For any $\delta\in(0,\tfrac12)$, we then define 
		\begin{equation*} 
			L_\delta(\omega):=\sup_{0\leq u<v \leq 1}\frac{|z(\theta_u\omega)-z(\theta_v\omega)|}{|u-v|^\delta}, 
		\end{equation*}
		which leads to  $|z(\theta_{t+s} \omega)-z(\theta_t \omega)| \leq L_\delta(\theta_t \omega) |s|^{\delta}$ for $s\in[0,1]$ where 
		\begin{equation*} 
			L_\delta(\theta_t \omega)=\sup_{0\leq u<v \leq 1}\frac{|z(\theta_{u+t}\omega)-z(\theta_{v+t}\omega)|}{|u-v|^\delta}, \qquad \forall \, t\in\mathbb R. 
		\end{equation*}   
		Introduce $\hat{z}_{u,v} = \frac{z(\theta_u\omega)-z(\theta_v\omega)}{|u-v|^\delta}$ for $u\neq v\in[0,1]$ and $\hat{z}_{u,v}=0$ for $u=v\in[0,1]$. 
		It follows from \eqref{OU} and the following property due to L\'evy (see \cite[Theorem 1.1.1]{Csorgo1981}) \begin{equation}\label{Levy}
			\lim\limits_{u\rightarrow0} \frac{\sup_{0\leq s \leq 1-u} \sup_{0<t\leq u} |\omega(s+t)-\omega(s)|}{\sqrt{2u \log(1/u)}} = 1, \qquad \mathbb{P}\text{-a.s.},
		\end{equation}
		that for any $\delta\in(0,\frac{1}{2})$ and $u_0\in[0,1]$,  
		$\lim_{(u,v)\rightarrow(u_0,u_0)} \hat{z}_{u,v} =0= \hat{z}_{u_0,u_0}$. This, combined with the H\"older continuity of $z$, leads to $\hat{z}_\cdot \in C([0,1]^2)$, $\mathbb{P}$-a.s. Thus $\hat{z}_{u,v}$ is a centered Gaussian element in the separable Banach space $C([0,1]^2)$. 
		Applying Fernique's theorem (see \cite[Theorem 2.7]{Prato2014}), there exists an $a_\delta>0$ such that $\mathbb{E}[\exp(a_\delta L_\delta^2)] = \mathbb{E}[\exp(a_\delta \|\hat{z}\|_{C([0,1]^2)}^2)] < \infty$.  
		By the invariance of $\theta$ with respect to $\mathbb{P}$ and Markov's inequality,  for all $n\in\mathbb{Z}$ and $r>0$, 
		\begin{equation*} 
			\mathbb{P} \big(L_\delta(\theta_n\omega)>r\big)=\mathbb{P}(L_\delta(\omega)>r)\leq  e^{- a_\delta r^2} \mathbb{E}[\exp(a_\delta \|\hat{z}\|_{C([0,1]^2)}^2)] 
			\leq C e^{- a_\delta r^2}, 
		\end{equation*}
		which means that for all $\varepsilon>0$, 
		\begin{equation*} 
			\sum_{n=1}^\infty \mathbb{P}\Big(L_\delta(\theta_n\omega)>e^{\varepsilon n}\Big)
			\leq   \sum_{n=1}^\infty C \exp\left(- a_\delta e^{2\varepsilon n}\right)<\infty. 
		\end{equation*}
		By Borel--Cantelli's lemma, there exists an a.s. finite random integer $N(\omega)$ such that
		$L_\delta(\theta_n\omega)\leq e^{\varepsilon n}$ for all $n\ge N(\omega)$.
		Let $C_{\varepsilon}(\omega)=\max\left\{1,\ \max_{0\le n\le N(\omega)} e^{-\varepsilon n}L_\delta(\theta_n\omega)\right\}$. Then  $L_\delta(\theta_n\omega)\le C_{\varepsilon}(\omega)\,e^{\varepsilon n}$ for all $n\geq 0$. Applying the same argument to $\{\theta_{-n}\}_{n\geq0}$ yields
		\begin{equation}\label{tempered n}
			L_\delta(\theta_n\omega)\le C_{\varepsilon}(\omega)\,e^{\varepsilon |n|},\qquad \forall \, n\in\mathbb Z. 
		\end{equation}
		
		We next extend the argument \eqref{tempered n} from $n \in \mathbb Z$ to $n \in \mathbb R$. 
		We claim that for all $t\in[n,n+1]$, 
		\begin{equation}\label{tilde L}
			L_\delta(\theta_t\omega)\leq L_\delta(\theta_n\omega) + L_\delta(\theta_{n+1}\omega), 
		\end{equation} 
		which yields the temperedness of $L_\delta$, namely, for any $\tilde{\varepsilon}>0$, 
		\begin{equation*} 
			L_\delta(\theta_t\omega)\leq   C_{\tilde{\varepsilon}}(\omega)\,\big(e^{\tilde{\varepsilon} |n|}+e^{\tilde{\varepsilon} |n+1|}\big)
			\leq  2 C_{\tilde{\varepsilon}}(\omega)  e^{\tilde{\varepsilon} |t+1|} 
			\leq (2C_{\tilde{\varepsilon}}(\omega)  e^{\tilde{\varepsilon}}) e^{\tilde{\varepsilon} |t|}, \qquad \forall \, t\in\mathbb{R}. 
		\end{equation*}
		Therefore, it suffices to prove the claim \eqref{tilde L} to obtain the desired conclusion. 
		In fact, for $t\in[n,n+1]$ and $u<v\in[0,1]$, we consider the following three cases: 
		\begin{itemize}
			\item if $u+t\in[n+1,n+2]$, then $v+t\in[n+1,n+2]$, and thus 
			\begin{align*}
				\frac{|z(\theta_{u+t}\omega)-z(\theta_{v+t}\omega)|}{|u-v|^\delta} 
				\leq L_\delta(\theta_{n+1}\omega); 
			\end{align*}
			
			\item if $u+t\in[n,n+1]$ and $v+t\in[n,n+1]$, we have 
			\begin{align*}
				\frac{|z(\theta_{u+t}\omega)-z(\theta_{v+t}\omega)|}{|u-v|^\delta} 
				\leq L_\delta(\theta_{n}\omega); 
			\end{align*}
			
			\item if $u+t\in[n,n+1]$ and $v+t\in[n+1,n+2]$, by the triangle inequality, 
			\begin{align*}
				&|z(\theta_{u+t}\omega)-z(\theta_{v+t}\omega)| 
				\leq |z(\theta_{u+t}\omega) - z(\theta_{n+1}\omega)| + |z(\theta_{n+1}\omega)-z(\theta_{v+t}\omega)| \\
				&\leq L_\delta(\theta_{n}\omega) |n+1-(u+t)|^{\delta} 
				+ L_\delta(\theta_{n+1}\omega) |v+t-(n+1)|^{\delta}. 
			\end{align*}
			Since $v-u=v+t-(n+1) + n+1-(u+t)$, in view of $n+1-(u+t)\in[0,1]$ and $v+t-(n+1)\in[0,1]$, we obtain 
			\begin{equation*}
				n+1-(u+t) \in [0,v-u], \qquad v+t-(n+1) \in[0,v-u], 
			\end{equation*}
			which means that 
			\begin{align*}
				\frac{|z(\theta_{u+t}\omega)-z(\theta_{v+t}\omega)| }{|u-v|^{\delta}} \leq L_\delta(\theta_{n}\omega) 
				+ L_\delta(\theta_{n+1}\omega). 
			\end{align*}
		\end{itemize}
		Combining these three cases yields the claim \eqref{tilde L}, which completes the proof. 
	\end{proof}
	
	We next demonstrate that the discrete RDS $\varphi^{\tau}$ possesses a random attractor. 
	\begin{thm}\label{thm:numerical attractor}
		For sufficiently small step size $\tau>0$, the discrete RDS $\varphi^\tau$ possesses a random attractor $A^\tau(\omega)$. 
	\end{thm}
	\begin{proof}
		For any $k\in\mathbb{Z}$, we denote $Z_{k} = z(\theta_{t_{k}} \omega)$ that satisfies 
		\begin{equation}\label{OU3}
			Z_{k+1} - Z_{k} = -\alpha \int_{t_{k}}^{t_{k+1}} z(\theta_t \omega) \rmd t + \sigma\Delta W_{k} 
			=: - \alpha \tilde{z}_{k+1} + \sigma\Delta W_{k}, 
		\end{equation}
		where $z(\theta_{t} \omega)$ is given by \eqref{OU}. 
		Let $\bm{Z}_{k} = (0,Z_{k},0,0)^{\top}$ and consider the following process 
		\begin{equation}\label{bar V}
			(\tilde{V}_{k}, \tilde{\mathcal{V}}_{k},\tilde{U}_{k}, \tilde{\mathcal{U}}_{k})^{\top} = \tilde{\mathcal{J}}(\bar{\bm{V}}_{k}) - \bm{Z}_{k} = \Big(\bar{V}_{k}, \ \bar{\mathcal{V}}_{k} , \ \bar{U}_{k}, \  \bar{\mathcal{U}}_{k}\sqrt{1-\gamma \sin^2 \bar{U}_{k}} - \bar{\mathcal{V}}_{k}\sin \bar{U}_{k}\Big)^{\top} - \bm{Z}_{k}, 
		\end{equation} 
		where $\bar{\bm{V}}_k$ is the numerical solution from \eqref{BEM}.  
		Since the transformation $\tilde{\mathcal{J}}:\mathbb{R}^4 \rightarrow \mathbb{R}^4$ is a diffeomorphism, the mapping \begin{equation}\label{conjugate2}
			\tilde{\varphi}^{\tau}(k,\omega,\bm{x}) = \tilde{\mathcal{J}} \big(\bar{\varphi}^{\tau}(k,\omega,\tilde{\mathcal{J}}^{-1}(\bm{x}+\bm{Z}_{0}(\omega))) \big) - \bm{Z}_{k}(\omega) 
		\end{equation}
		defines a cocycle conjugate to $\bar{\varphi}^{\tau}$ up to the random shift $\bm{Z}_{k}$, and hence generates a discrete RDS. 
			Furthermore, if $\tilde{\varphi}^{\tau}$ possesses a random attractor, denoted by $\tilde{A}^{\tau}(\omega)$, then $A^{\tau}(\omega) = \mathcal{J} \circ \tilde{\mathcal{J}}^{-1}\big(\tilde{A}^{\tau}(\omega)+\bm{Z}_0(\omega)\big)$ is a random attractor for the discrete RDS $\varphi^{\tau}$ in view of the conjugate relation \eqref{conjugate1} and \eqref{conjugate2}. 
			Hence, in the following, it suffices to prove the existence of a random attractor for the discrete RDS $\tilde{\varphi}^{\tau}$, which can be done by constructing a random absorbing set. 
			
			We first derive the \textit{a priori} estimate for $\tilde{\varphi}^{\tau}$. 
			For a constant $\delta_{3}>0$ to be chosen later, define
			\begin{equation}\label{Ektau}
				E^{\tau}_{k} := \frac{ \tilde{\mathcal{V}}_{k}^2 + \gamma\tilde{\mathcal{U}}_{k}^{2} + 2\gamma\tilde{\mathcal{V}}_{k}\tilde{\mathcal{U}}_{k} \sin \tilde{U}_{k} }{2(1-\gamma \sin^2 \tilde{U}_{k})}  + \frac{\kappa_{1}}{2} \tilde{V}_{k}^2 + \delta_{3} \tilde{V}_{k}\tilde{\mathcal{V}}_{k},
			\end{equation} 
			where $\bar{V}_k$, $\bar{\mathcal{V}}_k$, $\bar{U}_k$, and $\bar{\mathcal{U}}_k$ are determined by \eqref{BEM}. 
			By Appendix \ref{appendix Ektau}, we have  \begin{equation*}  
				E^{\tau}_{k+1} \leq e^{-c_{5}\tau+\bar{c}_{5} \int_{t_{k}}^{t_{k+1}} |z(\theta_s \omega)| \rmd s} E^{\tau}_{k} + \mathcal{G}_{k+1},  
			\end{equation*}
			where $\mathcal{G}_{k+1}$ is defined in \eqref{mathcalG}. 
			It follows from Gronwall's inequality that 
			\begin{equation*}
				E^{\tau}_{k} \leq e^{-c_{5}t_k +   \bar{c}_{5} \int_{0}^{t_{k}} |z(\theta_s \omega)| \rmd s} E^{\tau}_{0} 
				+ \sum_{i=1}^{k}  e^{-c_{5}(t_k-t_i)+ \bar{c}_{5} \int_{t_{i}}^{t_{k}} |z(\theta_s \omega)| \rmd s} \mathcal{G}_{i}. 
			\end{equation*}
			Similar to \eqref{Est:E}, we have $c_{1} \big(\tilde{V}_{k}^2 + \tilde{\mathcal{V}}_{k}^2 + \tilde{\mathcal{U}}_{k}^2\big) \leq E^{\tau}_{k} \leq c_{2}\big(1+\tilde{V}_{k}^2 + \tilde{\mathcal{V}}_{k}^2 + \tilde{\mathcal{U}}_{k}^2\big)$, 
			which, together with $\tilde{U}_k^2\leq4\pi^2$, means that 
			\begin{align*}
				&c_{1}\|\tilde{\varphi}^{\tau}(k,\omega,\bm{x})\|^2 
				\leq E_k^\tau + 4\pi^2 c_1 \\
				&\leq c_2e^{-c_{5}t_k +   \bar{c}_{5} \int_{0}^{t_{k}} |z(\theta_s \omega)| \rmd s} (1+\|\bm{x}\|^2)
				+ \sum_{i=1}^{k}  e^{-c_{5}(t_k-t_i) + \bar{c}_{5} \int_{t_{i}}^{t_{k}} |z(\theta_s \omega)| \rmd s} \mathcal{G}_{i} + 4\pi^2 c_1. 
			\end{align*}
			Recall that $\tilde{\varphi}^{\tau}(k,\theta_{-t_k}\omega,\bm{x})$ represents the process $\big(\tilde{V}_{0}, \tilde{\mathcal{V}}_{0},  \tilde{U}_{0}, \tilde{\mathcal{U}}_{0}\big)^{\top}$ obtained by evolving the system from $-k$ to $0$ with the initial condition $\bm{x} \in \mathbb{R}^4$. It follows that 
			\begin{equation*}
				\|\tilde{\varphi}^{\tau}(k,\theta_{-t_k}\omega,\bm{x})\|^2 
				\leq  \frac{c_2}{c_{1}} e^{-c_{5}t_k + \bar{c}_{5} \int_{-t_{k}}^{0} |z(\theta_s \omega)| \rmd s} (1+\|\bm{x}\|^2)
				+ \frac{1}{c_{1}} \sum_{i=-k}^{0}  e^{c_{5}t_i + \bar{c}_{5} \int_{t_i}^{0} |z(\theta_s \omega)| \rmd s} \mathcal{G}_{i} + 4\pi^2. 
			\end{equation*}  
			
			In order to estimate $\mathcal{G}_{i}$ defined in \eqref{mathcalG}, by Lemma \ref{le:tempered} and \eqref{OU3}, namely, $\sigma\Delta W_{i-1} = z(\theta_{t_{i}} \omega) - z(\theta_{t_{i-1}} \omega) + \alpha \int_{t_{i-1}}^{t_{i}} z(\theta_s \omega) \rmd s$, we obtain 
			\begin{align*}
				\mathcal{G}_{i} 
				&\leq C \tau \left(1+|z(\theta_{t_{i}} \omega)|^4 + |z(\theta_{t_{i-1}} \omega)|^4 \right) + C\int_{t_{i-1}}^{t_{i}} |z(\theta_s \omega)|^4 \rmd s 
				+ \frac{C}{\tau^2} \int_{t_{i-1}}^{t_{i}} |z(\theta_s \omega)-z(\theta_{t_{i}} \omega)|^8 \rmd s 
				\\
				& \leq CL^4_z(\omega) \, \tau(1+|t_i|^2) 
				+ \frac{C}{\tau^2} \int_{t_{i-1}}^{t_{i}} |z(\theta_s \omega)-z(\theta_{t_{i}} \omega)|^8 \rmd s, 
			\end{align*}
			which, together with Lemma \ref{le:tempered}, means that  when $\tau<1$, 
			\begin{align*}
				&\frac{1}{c_1}\sum_{i=-k}^{0}  e^{c_{5}t_i +   \bar{c}_{5} \int_{t_i}^{0} |z(\theta_s \omega)| \rmd s} \mathcal{G}_{i} \\
				\leq \,& \frac{\tau}{c_1} \sum_{i=-k}^{0}  e^{c_{5}t_i+   \bar{c}_{5} \int_{t_i}^{0} |z(\theta_s \omega)| \rmd s} \left( CL^4_z(\omega) (1+|t_i|^2) + \frac{C}{\tau^3} \int_{t_{i-1}}^{t_{i}} |z(\theta_s \omega)-z(\theta_{t_{i}} \omega)|^8 \rmd s \right) \\ 
				= \,& \frac{1}{c_1}\sum_{i=-k}^{0} \int_{t_{i-2}}^{t_{i-1}}  e^{c_{5}t_i + \bar{c}_{5} \int_{t_i}^{0} |z(\theta_s \omega)| \rmd s} \left( CL^4_z(\omega) (1+|t_i|^2) +  \frac{C}{\tau^3} \int_{t_{i-1}}^{t_{i}} |z(\theta_s \omega)-z(\theta_{t_{i}} \omega)|^8 \rmd s \right) \rmd r \\ 
				\leq \,& C\sum_{i=-k}^{0} \int_{t_{i-2}}^{t_{i-1}}  e^{c_{5}(r+2\tau) + \bar{c}_{5} \int_{r}^{0} |z(\theta_s \omega)| \rmd s} \Bigg( L_z^4(\omega)(1+|r+2\tau|^2) \\
				&\qquad \qquad \qquad +  \frac{C}{\tau^3} \int_{t_{i-1}}^{t_{i}} |z(\theta_s \omega)-z(\theta_{r} \omega)|^8 \rmd s 
				+  \frac{C}{\tau^3} \int_{t_{i-1}}^{t_{i}} |z(\theta_r \omega)-z(\theta_{t_{i}} \omega)|^8 \rmd s \Bigg) \rmd r \\ 
				\leq \,& Ce^{2c_5\tau} \sum_{i=-k}^{0} \int_{t_{i-2}}^{t_{i-1}}  e^{c_{5}r + \bar{c}_{5} \int_{r}^{0} |z(\theta_s \omega)| \rmd s} \left( 2L_z^4(\omega)(1+r^2+4\tau^2) + 2C L_{\delta}^8(\theta_r \omega) \tau^{8\delta-2} \right)  \rmd r \\
				\leq \,& C\int_{-\infty}^{0} e^{c_{5}r + \bar{c}_{5} \int_{r}^{0} |z(\theta_s \omega)| \rmd s} \left( L_z^4(\omega)(1+r^2) + L_\delta^8(\theta_r\omega) \right) \rmd r, 
			\end{align*}
			where in the last step, we choose $\delta\in(\frac{1}{4},\frac{1}{2})$.
			Define 
			\begin{equation*}
				\tilde{R}(\omega) := C\int_{-\infty}^{0} e^{c_{5}r + \bar{c}_{5} \int_{r}^{0} |z(\theta_s \omega)| \rmd s} \left( L_z^4(\omega)(1+r^2) + L_\delta^8(\theta_r\omega) \right) \rmd r, 
			\end{equation*}
			which is independent of $\tau$. One can show that $\tilde{R}$ is a finite and tempered random variable.  In fact, by using \eqref{ergodic z} and Lemma \ref{le:tempered}, for $t>T_0$ and sufficiently small $\varepsilon$, when $\alpha \geq \frac{8\sigma^2\bar{c}_{5}^2}{c_{5}^2}$, 
			\begin{align*}
				&\tilde{R}(\theta_{-t}\omega) 
				= C \int_{-\infty}^{-t} e^{c_{5}(t+s) + \bar{c}_{5} \left(\int_{s}^{0}-\int_{-t}^{0}\right) |z(\theta_\xi \omega)| \rmd \xi} \left( L_z^4(\theta_{-t}\omega)(1+|t+s|^2) + L_\delta^8(\theta_s\omega) \right) \rmd s \\
				&\leq C \int_{-\infty}^{-t} e^{\frac{c_{5}}{2}(t+s) + 2\bar{c}_{5}\varepsilon t} \left( L_z^4(\theta_{-t}\omega)(1+|t+s|^2) + L_\delta^8(\theta_s\omega) \right) \rmd s \\
				&\leq C(\omega) \int_{0}^{\infty} e^{-\frac{c_{5}}{2}r + 2\bar{c}_{5}\varepsilon t} \left( e^{4\varepsilon t}(1+r^2) + e^{8\varepsilon|t+r|} \right)  \rmd r \\
				&\leq C(\omega) e^{2\bar{c}_{5}\varepsilon t} \left( e^{4\varepsilon t} +  e^{8\varepsilon t} \right) 
				\leq 2C(\omega) e^{2\bar{c}_{5}\varepsilon t}   e^{8\varepsilon t}, 
			\end{align*}  
			which leads to 
			\begin{align*}
				\lim_{t\to\infty}\frac{1}{t}\log^+ \tilde{R}(\theta_{-t}\omega) 
				\leq \lim_{t\to\infty}\frac{1}{t} \left(\log 2C(\omega) + \varepsilon t(2\bar{c}_{5}+8) \right) 
				= \varepsilon (2\bar{c}_{5}+8). 
			\end{align*}
			Since $\varepsilon$ is arbitrary, it follows that $\tilde{R}$ is a tempered random variable. 
			In addition, for $T_0(\omega)$ determined by \eqref{ergodic z}, 
			\begin{align*}
				&\tilde{R}(\omega) =C \left(\int_{-\infty}^{-T_0} + \int_{-T_0}^{0} \right) e^{c_{5}s + \bar{c}_{5} \int_{s}^{0} |z(\theta_\xi \omega)| \rmd \xi} \left( L_z^4(\omega)(1+|s|^2) +L_\delta^8(\theta_s \omega)  \right) \rmd s \\  
				&\leq C(\omega) \int_{-\infty}^{-T_0} e^{\frac{c_{5}}{2}s} (1+|s|^2+e^{-8\varepsilon s}) \rmd s 
				+ C(\omega)\int_{-T_0}^{0} e^{\bar{c}_{5} \int_{s}^{0} |z(\theta_\xi \omega)| \rmd \xi}  (1+|s|^2+e^{8\varepsilon|s|}) \rmd s \\
				&\leq C(\omega,T_0(\omega)), 
			\end{align*}  
			which means that $\tilde{R}$ is a finite random variable. 
			
			Finally, we construct the absorbing set for $\tilde{\varphi}^{\tau}$. 
			Similar to the proof of Theorem \ref{th:attractor}, 
			for any tempered set $D\in\mathcal{D}$, when $\alpha \geq \frac{8\sigma^2\bar{c}_{5}^2}{c_{5}^2}$ and $t_k>\max\{T_0,T_1\}$, 
			\begin{equation}\label{absorb time2}
				\sup_{\bm x\in D(\theta_{-t_k}\omega)} \frac{c_2}{c_1} e^{-c_{5}t_k +  \bar{c}_{5} \int_{-t_{k}}^{0} |z(\theta_s \omega)| \rmd s} (1+\|\bm{x}\|^2)
				\leq \frac{2c_2}{c_1}e^{-\frac{c_{5}}{4}t_k}. 
			\end{equation}
			Define 
			\begin{equation}\label{mathbbK}
				\mathbb{K}(\omega) := \left\{ \bm{x}\in\mathbb{R}^4: \ \|\bm{x}\|^2 \leq  
				1+ \tilde{R}(\omega)+4\pi^2 \right\}. 
			\end{equation} 
			Then by \eqref{absorb time2}, for any $D\in\mathcal{D}$ and a.s. $\omega\in\Omega$, there exists a $\tilde{T}_D(\omega)=\max\{T_0,T_1, (4\log^+(2c_2/c_1))/c_5\}$ such that
			\begin{equation*}
				\tilde{\varphi}^{\tau}(k,\theta_{-t_k}\omega,D(\theta_{-t_k}\omega))
				\subset \mathbb{K}(\omega),
				\qquad \forall\, t_k > \tilde{T}_D(\omega). 
			\end{equation*}
			It follows that $\mathbb{K}(\omega)$ is an absorbing set. Consequently, by \cite[Theorem B.2]{Doan2018}, there exists a random attractor $\tilde{A}^{\tau}(\omega)$ for $\tilde{\varphi}^{\tau}$. 
			Accordingly, $A^{\tau}(\omega) = \mathcal{J} \circ \tilde{\mathcal{J}}^{-1}\big(\tilde{A}^{\tau}(\omega)+\bm{Z}_0\big)$ is a random attractor for the discrete RDS $\varphi^{\tau}$. 
		\end{proof}

		In the following proposition, we present the convergence of $A^{\tau}(\omega)$ under the Hausdorff semi-distance, whose proof follows from a contradiction argument (cf. \cite[Theorem 5.1]{Han2017}, \cite[Proposition 3.4]{SRB2025}).
		\begin{prop}\label{uppersemicont}
			The numerical random attractor $A^{\tau}(\omega)$ converges to the random attractor $A(\omega)$ under the Hausdorff	semi-distance, i.e., 
			\begin{equation}\label{upper sc}
				\lim\limits_{\tau\rightarrow0} d(A^{\tau}(\omega) , A(\omega)) = 0
			\end{equation}
			for $\mathbb{P}$-a.s. $\omega\in\Omega$. 
		\end{prop}

		\begin{proof} 
			Suppose, contrary to the assertion \eqref{upper sc}, that there exist an $\epsilon_0>0$, an $\omega\in\Omega$, a subsequence $\tau_n\rightarrow0$ as $n\rightarrow\infty$, and points $a_n\in A^{\tau_n}(\omega)$ such that
			\begin{equation}\label{contradict}
				d(a_n, A(\omega)) > \epsilon_0. 
			\end{equation} 
			
			From the construction of random attractor, we know that $\tilde{A}^{\tau}(\omega) \subset \mathbb{K}(\omega)$, where $\mathbb{K}(\omega)$ is defined by \eqref{mathbbK}.  
			In view of $A^{\tau}(\omega)=\mathcal{J} \circ \tilde{\mathcal{J}}^{-1}\big(\tilde{A}^{\tau}(\omega)+\bm{Z}_0\big)$ and $\|\mathcal{J} \circ \tilde{\mathcal{J}}^{-1}(\bm{x}+\bm{Z}_0)\|^2 \leq C(\gamma) \|\bm{x}\|^2 + 2\|z(\omega)\|^2$ for all $\bm{x}\in\mathbb{R}^4$, one can verify that $A^{\tau}(\omega)$ is contained in a ball $B_{R^*(\omega)}(0)$ in $\mathbb{R}^4$ with center $0$ and radius $R^*(\omega)=C(\gamma)  
			(1+ \tilde{R}(\omega)+4\pi^2) + 2\|z(\omega)\|^2$ that is tempered and is independent of $\tau$. Thus the ball $B_{R^*(\omega)}(0)\in\mathcal{D}$. By the definition of the random attractor, for any $\omega\in\Omega$, there exists a $T_2 = T_2(\omega,R^*)$ independent of $\tau$ such that
			\begin{equation}\label{attract}
				d\left(\varphi\left(t,\theta_{-t}\omega,B_{R^*(\theta_{-t}\omega)}(0) \right) , A(\omega)\right) < \frac{1}{4} \epsilon_0, \qquad \forall \, t \geq T_2. 
			\end{equation}

			Following a standard derivation,  for sufficiently small $\tau$, the global discretisation error on an interval of length $T_2>0$ for the driving system starting at $\bm{x}\in B_{R^*(\theta_{-T_2}\omega)}(0)$ is  
			\begin{equation*} 
				\left\| \varphi(t_k,\theta_{-t_k}\omega,\bm{x}) - \varphi_{\tau}(k,\theta_{-t_k}\omega,\bm{x}) \right\| 
				\leq C(\omega,T_2) \tau^{\delta}, \qquad 0\leq k\leq \lfloor T_2 / \tau \rfloor, 
			\end{equation*}
			where $\delta<\frac{1}{2}$ is determined by the H\"older continuity of the underlying Brownian motions.
			This means that for any $\bm{x}\in B_{R^*(\theta_{-T_2}\omega)}(0)$, when $\tau<(\epsilon_0/4C(\omega,T_2))^{1/\delta}$, 
			\begin{equation}\label{path convergence} 
				\|\varphi(t_k,\theta_{-t_k}\omega,\bm{x}) - \varphi_{\tau}(k,\theta_{-t_k}\omega,\bm{x})\| 
				\leq \frac{\epsilon_0}{4}, \qquad 0\leq k\leq \lfloor T_2 / \tau \rfloor. 
			\end{equation} 
			
			Consequently, we pick and fix $\tau_n<(\epsilon_0/4C(\omega,T_2))^{1/\delta}$, and suppose for convenience that $N_n \tau_n = T_2$. By the invariance property of the numerical random attractor $A^{\tau_n}(\omega)$, for any $a_n\in A^{\tau_n}(\omega)$, there exists an $a'_n \in A^{\tau_n}(\theta_{-T_2}\omega) \subset B_{R^*(\theta_{-T_2}\omega)}(0)$ such that $\varphi_{\tau_n}(N_n,\theta_{-T_2}\omega,a'_n)=a_n$. Thus for any $a_n\in A^{\tau_n}(\omega)$, by \eqref{attract} and \eqref{path convergence}, 
			\begin{align*}
				&d(a_n, A(\omega)) 
				= d\left( \varphi_{\tau_n}(N_n,\theta_{-T_2}\omega,a'_n) , A(\omega) \right) \\
				\leq \,& \left\| \varphi(T_2,\theta_{-T_2}\omega,a'_n) - \varphi_{\tau_n}(N_n,\theta_{-T_2}\omega,a'_n) \right\| 
				+ d\left( \varphi(T_2,\theta_{-T_2}\omega,a'_n) , A(\omega) \right) 
				< \frac{\epsilon_0}{2}, 
			\end{align*}
			which contradicts \eqref{contradict}. 
			This completes the proof. 
		\end{proof}

		\section{Numerical Lyapunov exponent for the single mode solution}\label{Sec:numerical lambda}
		In this section, we investigate the numerical Lyapunov exponent $\lambda^{\tau}$ associated with the single mode solution $(\eta_{k}, \bar{\eta}_{k}, 0, 0)$ of \eqref{BEM}. 
		We show that the proposed numerical method \eqref{BEM} can preserve the sign of $\lambda_{0}$, and thus preserve the almost-sure stability of the single mode solution. 
		
		\subsection{Definition of $\lambda^{\tau}$} 
		
		Linearizing \eqref{BEM} along the single mode solution $(\eta_{k}, \bar{\eta}_{k}, 0, 0)$ gives the process $\{(\alpha_k, \bar{\alpha}_k,\beta_k, \bar{\beta}_k)\}_{k\geq0}$ satisfying 
		\begin{equation}\label{linearization1}
			\left\{\begin{aligned}
				\alpha_{k+1} &= \alpha_{k} + \tau \bar{\alpha}_{k+1},  \\
				\bar{\alpha}_{k+1} &= \bar{\alpha}_{k} -\tau \kappa_{1} \alpha_{k+1} - 2\tau\zeta_1 \bar{\alpha}_{k+1}, \\
				\beta_{k+1} &= \beta_{k} + \tau \bar{\beta}_{k+1}, \\
				\bar{\beta}_{k+1} &= \bar{\beta}_{k} - \tau(\kappa_{2} + 2\zeta_1 \bar{\eta}_{k+1} + \kappa_{1} \eta_{k+1}) \beta_{k} 
				- 2\tau\zeta_2 \bar{\beta}_{k+1} + \sigma\beta_{k}\Delta W_{k}. 
			\end{aligned}\right. 
		\end{equation}
		The first two equations in \eqref{linearization1} correspond to an unforced damped oscillator. Since $\kappa_1>0$ and $\zeta_1>0$, its implict Euler discretization is exponentially stable, and hence $\|(\alpha_k, \bar{\alpha}_k)\| \rightarrow 0$ as $k\rightarrow\infty$. 
		The stability of the single mode solution is determined by the long-time behavior of the linearized process $\{(\alpha_k, \bar{\alpha}_k,\beta_k, \bar{\beta}_k)\}_{k\geq0}$.  
		Therefore, the transverse stability is governed by the parametric excitation of $(\beta_k,\bar\beta_k)$ induced by the single mode vibration $(\eta_k,\bar\eta_k)$.
		More precisely, in the sequel, we consider the long-time growth or decay rate of the process $(\beta_k, \bar{\beta}_k)$ satisfying 
		\begin{equation}\label{VXbeta}
			\left\{\begin{aligned}
				\beta_{k+1} &= \beta_{k} + \tau \bar{\beta}_{k+1}, \\
				\bar{\beta}_{k+1} &= \bar{\beta}_{k} - \tau(\kappa_{2} + 2\zeta_1 \bar{\eta}_{k+1} + \kappa_{1} \eta_{k+1}) \beta_{k} 
				- 2\tau\zeta_2 \bar{\beta}_{k+1} + \sigma\beta_{k}\Delta W_{k}, \\
				\eta_{k+1} &= \eta_{k} + \tau \bar{\eta}_{k+1}, \\
				\bar{\eta}_{k+1} &= \bar{\eta}_{k} - \tau \left( \kappa_{1} \eta_{k+1} + 2\zeta_1 \bar{\eta}_{k+1} \right) + \sigma \Delta W_{k}.  
			\end{aligned}\right. 
		\end{equation}
		The implicit system \eqref{VXbeta} is uniquely solvable at each time step: one can first determine $(\eta_{k+1},\bar{\eta}_{k+1})$ from the last two equations, and then substitute them into the first two equations to obtain a linear implicit system with a unique solution for $(\beta_{k+1},\bar\beta_{k+1})$.  
		The stability of the single mode solution is determined by the almost-sure exponential growth rate of \eqref{VXbeta}. 
		For any $(\beta_0,\bar\beta_0)\ne (0,0)$, the numerical Lyapunov exponent is defined by 
		\begin{equation}\label{numerical Lyapunov exponent}
			\lambda^{\tau} := \lim\limits_{k\rightarrow\infty} \frac{1}{k\tau} \log \left\| (\beta_{k}, \bar{\beta}_{k}) \right\|.
		\end{equation}
		From \eqref{compact BEM}, we obtain the variational equation 
			\begin{equation*} 
				\delta\bar{\bm{V}}_{k+1}
				= \delta\bar{\bm{V}}_{k} + \tau \bm{J}(\bar{U}_{k}) \, \delta\bar{\bm{V}}_{k+1} 
				+ \left( \tau D_{\bar{\bm{V}}_{k}} (\bm{J}(\bar{U}_{k})\bar{\bm{V}}_{k+1}) + \tau D\bm{J}_1(\bar{\bm{V}}_{k}) + D \bm{J}_2(\bar{\bm{V}}_{k}) \Delta W_{k} \right) \delta\bar{\bm{V}}_{k}. 
		\end{equation*}
	In view of the invertibility of the matrix $I_{4\times4}-\tau \bm{J}(\bar{U}_{k})$, we have 
	 \begin{align*} 
		&\left\| (I_{4\times4}-\tau \bm{J}(\bar{U}_{k}))^{-1}  \left( I_{4\times4} + \tau D_{\bar{\bm{V}}_{k}} (\bm{J}(\bar{U}_{k})\bar{\bm{V}}_{k+1}) + \tau D\bm{J}_1(\bar{\bm{V}}_{k}) + D \bm{J}_2(\bar{\bm{V}}_{k}) \Delta W_{k} \right) \right\| \\
		&\leq C\left( \|\bar{\bm{V}}_{k+1}\| + |\Delta W_k| + 1\right). 
	\end{align*}
	It follows from $\mathbb{E}[\log^+\left(\|\bar{\bm{V}}_{k+1}\| + |\Delta W_k| + 1\right)] < \infty$ and the multiplicative ergodic theorem (see \cite[Theorem 3.4.1]{arnoldRDS}) that $\lambda^{\tau}$ exists.

		We note that the process $(\beta_{k}, \bar{\beta}_{k})$ is excited by a combination of the Brownian increment $\Delta W_{k}$ as well as the colored noise $(\eta_{k}, \bar{\eta}_{k})$, which makes the evaluation of $\lambda^{\tau}$ nontrivial. 
		Rather than analyzing $\lambda^\tau$ directly, we relate it to the continuous Lyapunov exponent $\lambda_0$ through finite-time discretization errors and the ergodic properties of the exact and numerical variational processes.

		\subsection{Estimate of $\lambda^{\tau}$}
		To estimate $\lambda^{\tau}$, we set $B_{k} = (\beta_{k}, \bar{\beta}_{k})^{\top}$ for all $k\geq 0$ and introduce the  polar coordinate 
		\begin{equation*}
			B_k = \|B_{k}\|(\cos\psi_{k},\sin\psi_{k})^{\top}. 
		\end{equation*} 
		For a nonzero initial condition $B_0\neq (0,0)^\top$, the corresponding linear update is invertible for sufficiently small $\tau$, and hence $B_k\neq (0,0)^\top$ for all $k\geq 0$.  
		For brevity of notation, we denote 
		\begin{equation*}
			\mathcal{Z}_{k}=\kappa_{2} + 2\zeta_1 \bar{\eta}_{k} + \kappa_{1} \eta_{k}. 
		\end{equation*}
		The following proposition presents the ergodicity of the Markov chain associated with the variational dynamics, which plays a critical role in establishing the connection between $\lambda^{\tau}$ and $\lambda_{0}$. The proof of this proposition is postponed to Section \ref{Sec:proof}. 
		\begin{prop}\label{prop:ergodicity eta psi}
			There exists a unique stationary measure $\mu^{\tau}$ (resp. $\nu^{\tau}$) for  $\{(\eta_{k},\bar{\eta}_{k},\psi_{k})\}_{k\geq0}$ (resp.  $\{(\eta_{k},\bar{\eta}_{k})\}_{k\geq0}$). Furthermore, there exist constants $\bar{a}\in(0,1)$ and $\bar{b}\in(0,\infty)$ such that  
			\begin{equation}\label{ergodicity for eta}
				\left| \mathbb{E}\left[ Q(\eta_k,\bar{\eta}_k,\psi_k) \right] - \mu^{\tau}(Q) \right| \leq \bar{b} \bar{a}^{k\tau}  \Gamma_{2}(\eta_0,\bar{\eta}_0,\psi_0), 
			\end{equation}
			where the functions $Q$ and $\Gamma_{2}$ are given by \eqref{Qh} and \eqref{Gamma2}, respectively. 
		\end{prop}
		The following lemma involves the estimates of $B_{k}$, which is proved in Appendix \ref{appendix:A}. 
		\begin{lemma}\label{lemma:Bk}
			Assume that $B_0\neq(0,0)^\top$. Then for any $k\geq1$, it holds that 
			\begin{align}
				\frac{\|B_{k}\|^2}{\|B_{k+1}\|^2} 
				&= 1+ \mathbb{B}_{k+1}
				\leq 1 + C\tau \left(|\mathcal{Z}_{k+1}|^2+1\right) + C|\Delta W_{k}| 
				+ C\Delta W_{k}^2, \label{Est:Bk1} \\
				\frac{\|B_{k}\|^2}{\|B_{k+1}\|^2} 
				&= 1-2\sigma\cos\psi_{k}\sin\psi_{k}\Delta W_{k} + \tilde{\mathbb{B}}_{k+1}, \label{Est:Bk3}
			\end{align}
			where $\mathbb{B}_{k+1}$ is given by \eqref{Est:Bk2} and $|\tilde{\mathbb{B}}_{k+1}| \leq C\left(\Delta W_{k}^2 + \Delta W_{k}^4\right) 
			+ C\tau\left(|\mathcal{Z}_{k+1}|^4+1\right)$. 
			Furthermore,  
			\begin{equation}\label{Est:Bk+1/k}
				\frac{\|B_{k+1}\|^2}{\|B_{k}\|^2} = 1+\mathbb{D}_{k+1} 
				\leq 1 + C\tau \left(|\mathcal{Z}_{k+1}|^2+1\right) + C|\Delta W_{k}| 
				+ C\Delta W_{k}^2, 
			\end{equation} 
			where $\mathbb{D}_{k+1}$ is given by \eqref{mathbbD}. 
		\end{lemma} 
		Based on these results, we present a critical estimate for $\lambda^{\tau}$, which is used to establish its connection to $\lambda_0$. 
		\begin{thm}\label{thm:lambda tau}
			Assume that $B_0\neq(0,0)^\top$. Then for sufficiently small $\tau$, it holds that 
			\begin{equation}\label{positivity of lambdatau}
				\mu^{\tau}(Q) 
				\leq \lambda^{\tau} 
				+ C \tau^{\frac{1}{2}}, 
			\end{equation}
			where $\mu^{\tau}$ is the stationary measure for $(\eta_{k},\bar{\eta}_{k},\psi_{k})$ and the function $Q$ is given by \eqref{Qh}. 
		\end{thm}
		\begin{proof}
			In order to derive the expression of $\lambda^{\tau}$ defined by \eqref{numerical Lyapunov exponent}, we first investigate the estimate of $\|B_{k}\|$. 
			Denote $\bm{I} = (0, 1)^{\top}$. Applying \eqref{VXbeta} and $\beta_{k}=\beta_{k+1} - \tau \bar{\beta}_{k+1}$, we obtain 
			\begin{align*}
				B_{k+1} = B_{k} + \begin{pmatrix}
					0 & \tau \\ 
					- \tau \mathcal{Z}_{k+1} &  
					- 2\tau\zeta_2  
				\end{pmatrix}
				B_{k+1} +
				\left( \sigma\beta_{k}\Delta W_{k}  + \tau^2 \mathcal{Z}_{k+1}\bar{\beta}_{k+1}\right) \bm{I}. 
			\end{align*}   
			On the one hand, by Taylor's expansion, we have 
			\begin{align*}
				&\|B_{k}+\sigma\beta_{k}\Delta W_{k} \bm{I}\|^2 \\
				=& \|B_{k+1}\|^2 + 2 \left\langle B_{k+1} , \,  B_{k}+\sigma\beta_{k}\Delta W_{k} \bm{I}-B_{k+1} \right\rangle + \|B_{k}+\sigma\beta_{k} \Delta W_{k} \bm{I}-B_{k+1}\|^2 \\
				=& \|B_{k+1}\|^2 \left(1-2 \tau \cos\psi_{k+1} \sin\psi_{k+1} \left(1- \mathcal{Z}_{k+1}\right) + 4\tau\zeta_2 \sin^2\psi_{k+1}
				+ \tau^2\mathcal{H}_{k+1} \right), 
			\end{align*}
			where $\mathcal{H}_{k} = \sin^2\psi_{k} (1 - 2\mathcal{Z}_{k}) + \left( \mathcal{Z}_{k} \cos\psi_{k} + (2\zeta_2-\tau\mathcal{Z}_{k}) \sin\psi_{k} \right)^2$ and satisfies 
			\begin{equation}\label{mathcal H}
				|\mathcal{H}_{k}| \leq  |1 - 2\mathcal{Z}_{k}| 
				+ ( |\mathcal{Z}_{k}| + 2\zeta_2+\tau|\mathcal{Z}_{k}|)^2  
				\leq C(|\mathcal{Z}_{k}|^2+1). 
			\end{equation}
			On the other hand, 
			\begin{align*}
				\|B_{k}+\sigma\beta_{k}\Delta W_{k} \bm{I}\|^2&=\|B_{k}\|^2+\sigma^2\beta_{k}^2\Delta W_{k}^2 + 2\sigma\beta_{k}\bar{\beta}_{k} \Delta W_{k} \\
				&=\|B_{k}\|^2 + \|B_{k+1}\|^2 \frac{\|B_{k}\|^2}{\|B_{k+1}\|^2}\left( \sigma^2\cos^2\psi_{k}\Delta W_{k}^2 + 2\sigma\cos\psi_{k}\sin\psi_{k}\Delta W_{k} \right). 
			\end{align*}
			It follows that 
			\begin{align*}
				\|B_{k}\|^2 &= \|B_{k+1}\|^2 \bigg( 1-2 \tau \cos\psi_{k+1} \sin\psi_{k+1} \left(1- \mathcal{Z}_{k+1}\right) + 4\tau\zeta_2 \sin^2\psi_{k+1} + \tau^2\mathcal{H}_{k+1} \\
				&\qquad \qquad \qquad - \frac{\|B_{k}\|^2}{\|B_{k+1}\|^2}\left( \sigma^2\cos^2\psi_{k}\Delta W_{k}^2 + 2\sigma\cos\psi_{k}\sin\psi_{k}\Delta W_{k} \right)
				\bigg). 
			\end{align*}
			
			We next derive the estimate of $\log\|B_{k}\|$. 
			By using the inequality $\log(1+x)\leq x-\frac{x^2}{2}+\frac{x^3}{3}$ for $x>-1$, we have 
			\begin{align*}
				& \log\|B_{k}\|^2 - \log \|B_{k+1}\|^2 \\
				= \,& \log\bigg( 1 \underbrace{-2 \tau \cos\psi_{k+1} \sin\psi_{k+1} \left(1- \mathcal{Z}_{k+1}\right) + 4\tau\zeta_2 \sin^2\psi_{k+1}+ \tau^2\mathcal{H}_{k+1}
					- \frac{\sigma^2\cos^2\psi_{k}\|B_{k}\|^2}{\|B_{k+1}\|^2} \Delta W_{k}^2}_{:=\mathbb{X}_1}  \\
				&\qquad - \underbrace{\frac{\|B_{k}\|^2}{\|B_{k+1}\|^2} 2\sigma\cos\psi_{k}\sin\psi_{k}\Delta W_{k}}_{:=\mathbb{X}_2}  
				\bigg) \\
				\leq \,&-2 \tau \cos\psi_{k+1} \sin\psi_{k+1} \left(1-\mathcal{Z}_{k+1}\right) + 4\tau\zeta_2 \sin^2\psi_{k+1} + \tau^2\mathcal{H}_{k+1} 
				- \frac{\|B_{k}\|^2}{\|B_{k+1}\|^2} \sigma^2\cos^2\psi_{k}\Delta W_{k}^2 \\
				& - \frac{\|B_{k}\|^2}{\|B_{k+1}\|^2}  2\sigma\cos\psi_{k}\sin\psi_{k}\Delta W_{k} 
				- \frac{1}{2} \frac{\|B_{k}\|^4}{\|B_{k+1}\|^4} 4\sigma^2 \cos^2\psi_{k}\sin^2\psi_{k} \Delta W_{k}^2 + \mathcal{R}_{1,k+1}, 
			\end{align*}
			where $\mathcal{R}_{1,k+1}= -\frac{1}{2} \mathbb{X}_1^2 + \mathbb{X}_1 \mathbb{X}_2 + \frac{1}{3} (\mathbb{X}_1-\mathbb{X}_2)^3$. 
			It follows from \eqref{mathcal H} and \eqref{Est:Bk1} that 
			\begin{equation}\label{Est:mathcalR1} 
				|\mathcal{R}_{1,k+1}|
				\leq C (\tau|\Delta W_{k}|+\tau^2) \left(  |\mathcal{Z}_{k+1}|^{12} +1 \right)
				+ |\Delta W_{k}|^3 + |\Delta W_{k}|^{12}.  
			\end{equation}
			In addition, by \eqref{Est:Bk1} and \eqref{Est:Bk3}, we have 
			\begin{align*}
				&- \frac{\|B_{k}\|^2}{\|B_{k+1}\|^2}\left( \sigma^2\cos^2\psi_{k}\Delta W_{k}^2 + 2\sigma\cos\psi_{k}\sin\psi_{k}\Delta W_{k} \right) 
				- \frac{2\sigma^2\|B_{k}\|^4}{\|B_{k+1}\|^4} \cos^2\psi_{k}\sin^2\psi_{k}\Delta W_{k}^2 \\ 
				=\,& \sigma^2\cos^2\psi_{k}\Delta W_{k}^2 \left(-1-\mathbb{B}_{k+1}\right) 
				- 2\sigma\cos\psi_{k}\sin\psi_{k}\Delta W_{k} \left( 1-2\sigma\cos\psi_{k}\sin\psi_{k}\Delta W_{k} 
				+ \tilde{\mathbb{B}}_{k+1} \right) \\
				& - 2\sigma^2 \cos^2\psi_{k}\sin^2\psi_{k} \Delta W_{k}^2 
				\left( 1+ \mathbb{B}_{k+1} \right)^2 \\ 
				\leq\,& -\sigma^2\cos^2\psi_{k}\Delta W_{k}^2 
				- 2\sigma\cos\psi_{k}\sin\psi_{k}\Delta W_{k} + 2\sigma^2 \cos^2\psi_{k}\sin^2\psi_{k} \Delta W_{k}^2 \\
				& + \sigma^2\Delta W_{k}^2|\mathbb{B}_{k+1}|
				+ 2\sigma|\Delta W_{k}| |\tilde{\mathbb{B}}_{k+1}| 
				+ 2\sigma^2 \Delta W_{k}^2 
				\left( \mathbb{B}_{k+1}^2 + 2 |\mathbb{B}_{k+1}| \right)\\ 
				=:& -\sigma^2\cos^2\psi_{k}\Delta W_{k}^2 
				- 2\sigma\cos\psi_{k}\sin\psi_{k}\Delta W_{k} + 2\sigma^2 \cos^2\psi_{k}\sin^2\psi_{k} \Delta W_{k}^2 
				+ \mathcal{R}_{2,k+1}, 
			\end{align*}
			where, in view of \eqref{Est:Bk1} and \eqref{Est:mathbbB2}, 
			\begin{equation}\label{Est:mathcalR2}
				|\mathcal{R}_{2,k+1}| 
				\leq C \left( \tau|\Delta W_{k}| + \tau \Delta W_{k}^2 \right) \left(|\mathcal{Z}_{k+1}|^4+1\right)
				+ |\Delta W_{k}|^6 + |\Delta W_{k}|^3. 
			\end{equation}
			Combining the above estimates together, we obtain 
			\begin{equation}\label{logBk-logBk+1}
				\begin{aligned}
					& \log\|B_{k}\|^2 - \log \|B_{k+1}\|^2 \\
					\leq\,& -2 \tau \cos\psi_{k+1} \sin\psi_{k+1} \left(1-\mathcal{Z}_{k+1}\right) + 4\tau\zeta_2 \sin^2\psi_{k+1} -\sigma^2\cos^2\psi_{k}\Delta W_{k}^2 \\
					& + 2\sigma^2 \cos^2\psi_{k}\sin^2\psi_{k} \Delta W_{k}^2
					- 2\sigma\cos\psi_{k}\sin\psi_{k}\Delta W_{k} + \tau^2\mathcal{H}_{k+1} 
					+ \mathcal{R}_{2,k+1} + \mathcal{R}_{1,k+1} \\
					=\,& -2 \tau Q\left( \eta_{k+1},\bar{\eta}_{k+1},\psi_{k+1} \right) 
					+\sigma^2\left(\tau-\Delta W_{k}^2\right)  \cos^2\psi_{k}\cos(2\psi_{k})
					- 2\sigma\cos\psi_{k}\sin\psi_{k}\Delta W_{k}  \\
					&- \tau \sigma^2 \left( \cos^2\psi_{k}\cos(2\psi_{k}) - \cos^2\psi_{k+1}\cos(2\psi_{k+1}) \right) + \tau^2\mathcal{H}_{k+1} 
					+ \mathcal{R}_{2,k+1} + \mathcal{R}_{1,k+1}, 
				\end{aligned}
			\end{equation} 
			where $Q$ is given by \eqref{Qh}.  Summing both sides of \eqref{logBk-logBk+1} over $k$ from $0$ to $N-1$ yields that  
			\begin{equation}\label{1/N logB}
				\begin{aligned}
					& \frac{1}{N} \sum_{k=0}^{N-1} Q\left( \eta_{k+1},\bar{\eta}_{k+1},\psi_{k+1} \right) 
					+ \frac{\sigma^2}{2N} \left( \cos^2\psi_{0}\cos(2\psi_{0}) - \cos^2\psi_{N}\cos(2\psi_{N}) \right) \\
					\leq\,& \frac{1}{N\tau} \log \|B_{N}\| -  \frac{1}{N\tau} \log\|B_{0}\| 
					+ \frac{1}{2N\tau} \sum_{k=0}^{N-1} \left( \tau^2\mathcal{H}_{k+1} 
					+ \mathcal{R}_{2,k+1} + \mathcal{R}_{1,k+1} \right)
					\\
					& +\frac{1}{2N\tau} \sum_{k=0}^{N-1} \left( \left(\tau-\Delta W_{k}^2\right)  \sigma^2\cos^2\psi_{k}\cos(2\psi_{k})
					- 2\sigma\cos\psi_{k}\sin\psi_{k}\Delta W_{k} \right). 
				\end{aligned}
			\end{equation}
			This gives the estimate of $\log\|B_{k}\|$. By the definition of $\lambda^{\tau}$, taking the limit as $N\rightarrow\infty$ in \eqref{1/N logB} yields the lower bound for $\lambda^{\tau}$. 
			
			We next demonstrate that the last summation in \eqref{1/N logB} vanishes as $N\rightarrow\infty$. In fact, one can be verify that 
			\begin{equation*}
				\sum_{k=1}^{n}  \left(\tau-\Delta W_{k-1}^2\right)  \sigma^2\cos^2\psi_{k-1}\cos(2\psi_{k-1}) =: \sum_{k=1}^{n} x_{k}
			\end{equation*}
			is a martingale and  satisfies that $\sum_{k=1}^{\infty} k^{-2} \mathbb{E}[x_k^2] < \infty$. 
			It follows from the martingale limit theory (see, e.g., \cite[Theorem 2.18]{Hall1980}) that  
			\begin{equation}\label{martingale1}
				\lim\limits_{N\rightarrow\infty} \frac{1}{N} \sum_{k=0}^{N-1}  \left(\tau-\Delta W_{k}^2\right)  \sigma^2\cos^2\psi_{k}\cos(2\psi_{k})	= 0, \qquad \mathbb{P}\text{-a.s.} 
			\end{equation} 
			In an analogue way, one can show that 
			\begin{equation}\label{martingale2}
				\lim\limits_{N\rightarrow\infty} \frac{1}{N} \sum_{k=0}^{N-1}  2\sigma\cos\psi_{k}\sin\psi_{k}\Delta W_{k} = 0, \qquad \mathbb{P}\text{-a.s.} 
			\end{equation}
			Now by taking the limit as $N\rightarrow\infty$ in \eqref{1/N logB} and using Birkhoff's ergodic theorem, \eqref{martingale1}, and \eqref{martingale2}, we obtain 
			\begin{equation}\label{Est:lambda tau}
				\mu^{\tau}(Q) 
				\leq \lambda^{\tau} 
				+ \lim\limits_{N\rightarrow\infty}  \frac{1}{2N\tau} \sum_{k=1}^{N} \left( \tau^2\mathcal{H}_{k} 
				+ \mathcal{R}_{2,k} + \mathcal{R}_{1,k} \right), 
			\end{equation}
			where $\mu^{\tau}$ is the unique stationary measure for the ergodic process $\{(\eta_{k},\bar{\eta}_{k},\psi_{k})\}_{k\geq0}$.   
			
			Finally, we derive \eqref{positivity of lambdatau} by estimating the limitation in \eqref{Est:lambda tau}. 
			For any $p\geq1$, since $\Delta W_{k}$ are independent and identically distributed random variables, by the strong law of large numbers, we have 
			\begin{equation}\label{SLLN1}
				\lim\limits_{N\rightarrow\infty}\frac{1}{N} \sum_{k=1}^{N} |\Delta W_{k}|^{p} 
				= \mathbb{E}\left[ |\Delta W_{1}|^{p} \right]
				= C(p) \tau^{\frac{p}{2}}, \qquad \mathbb{P}\text{-a.s.}
			\end{equation}
			For any $p_1, p_2 \geq 1$, it follows from Cauchy--Schwarz's inequality and \eqref{SLLN1} that 
			\begin{equation}\label{SLLN2}
				\begin{aligned}
					\lim\limits_{N\rightarrow\infty}\frac{1}{N} \sum_{k=1}^{N} |\Delta W_{k}|^{p_1} |\mathcal{Z}_{k}|^{p_2} 
					&\leq \lim\limits_{N\rightarrow\infty} \left(\frac{1}{N} \sum_{k=1}^{N} |\Delta W_{k}|^{2p_1} \right)^\frac{1}{2} 
					\left(\frac{1}{N} \sum_{k=1}^{N} |\mathcal{Z}_{k}|^{2p_2} \right)^{\frac{1}{2}} \\
					&\leq C(p_1)\tau^{\frac{p_1}{2}} 
					\lim\limits_{N\rightarrow\infty}
					\left(\frac{1}{N} \sum_{k=1}^{N} |\mathcal{Z}_{k}|^{2p_2} \right)^{\frac{1}{2}}. 
				\end{aligned} 
			\end{equation}
			Consequently, combining \eqref{mathcal H}--\eqref{Est:mathcalR2} and utilizing \eqref{SLLN1}, \eqref{SLLN2}, and Proposition \ref{prop:ergodicity eta psi}, we obtain 
			\begin{align*}
				& \lim\limits_{N\rightarrow\infty}  \frac{1}{2N\tau} \sum_{k=1}^{N} \left( \tau^2\mathcal{H}_{k} 
				+ \mathcal{R}_{2,k} + \mathcal{R}_{1,k} \right) \\
				\leq \,& \lim\limits_{N\rightarrow\infty}  \frac{C}{2N\tau} \sum_{k=1}^{N} \left( \tau|\Delta W_{k-1}| + \tau\Delta W_{k-1}^2 + \tau^2 \right) \left(|\mathcal{Z}_{k}|^{12} + 1 \right) 
				+ \lim\limits_{N\rightarrow\infty}  \frac{C}{2N\tau} \sum_{k=0}^{N-1}\left( |\Delta W_{k}|^3 + \Delta W_{k}^{12} \right) \\
				\leq \,& C \tau^{\frac{1}{2}} \left( \lim\limits_{N\rightarrow\infty}  \frac{1}{N} \sum_{k=1}^{N} |\mathcal{Z}_{k}|^{24} + 1\right) 
				+ \frac{C}{2\tau} (\tau^{\frac{3}{2}} + \tau^{4}) 
				\leq C \tau^{\frac{1}{2}} \left( 1+\nu^{\tau}(\|\cdot\|^{24}) \right), 
			\end{align*}
			where $\nu^{\tau}$ is the unique stationary measure for $\{(\eta_{k},\bar{\eta}_{k})\}_{k\geq0}$, and $\nu^{\tau}(\|\cdot\|^{p}):=\int_{\mathbb{R}^2} \|\bm{x}\|^{p} \nu^{\tau}(\rmd \bm{x})$. 
			In view of the exponential moment boundedness of $\eta_k$ and $\bar{\eta}_k$, i.e., \eqref{exp int X}, one can verify that $\nu^{\tau}(\|\cdot\|^{p})\leq C(p)$ for any $p\geq1$. 
			Substituting the above estimates into \eqref{Est:lambda tau} completes the proof. 
		\end{proof}

		\begin{thm}\label{thm:lambda tau2}
			For sufficiently small $\tau$, the numerical Lyapunov exponent $\lambda^{\tau}$ satisfies that 
			\begin{equation}\label{negative of lambdatau}
				\lambda^{\tau} 
				\leq \mu^{\tau}(Q) 
				+ C \tau^{\frac{1}{2}}. 
			\end{equation} 
		\end{thm}
		\begin{proof}
			Reversing the comparison used in the proof of Theorem \ref{thm:lambda tau}, we obtain
			\begin{align*}
				\|B_{k+1}\|^2 &= \|B_{k}\|^2 \bigg( 1 + \sigma^2\cos^2\psi_{k}\Delta W_{k}^2 + 2\sigma\cos\psi_{k}\sin\psi_{k}\Delta W_{k} \\
				&\qquad \quad + \frac{\|B_{k+1}\|^2}{\|B_{k}\|^2}\left( 2 \tau \cos\psi_{k+1} \sin\psi_{k+1} \left(1- \mathcal{Z}_{k+1}\right) - 4\tau\zeta_2 \sin^2\psi_{k+1} - \tau^2\mathcal{H}_{k+1} \right)
				\bigg). 
			\end{align*}
			By using the inequality $\log(1+x)\leq x-\frac{x^2}{2}+\frac{x^3}{3}$ for $x>-1$ and \eqref{Est:Bk+1/k}, we have 
			\begin{align*}
				& \log\|B_{k+1}\|^2 - \log \|B_{k}\|^2 \\
				= \,& \log\bigg( 1 + \underbrace{2\sigma\cos\psi_{k}\sin\psi_{k}\Delta W_{k}}_{=:\mathbb{Y}_1} + \underbrace{\sigma^2\cos^2\psi_{k}\Delta W_{k}^2}_{=:\mathbb{Y}_2} \\
				&\qquad \quad + \underbrace{\frac{\|B_{k+1}\|^2}{\|B_{k}\|^2}\left( 2 \tau \cos\psi_{k+1} \sin\psi_{k+1} \left(1- \mathcal{Z}_{k+1}\right) - 4\tau\zeta_2 \sin^2\psi_{k+1} - \tau^2\mathcal{H}_{k+1} \right)}_{=:\mathbb{Y}_3}
				\bigg) \\
				\leq \,&2\sigma\cos\psi_{k}\sin\psi_{k}\Delta W_{k} + \sigma^2\cos^2\psi_{k}\Delta W_{k}^2 
				+ 2\tau (1+\mathbb{D}_{k+1})  \cos\psi_{k+1} \sin\psi_{k+1} \left(1- \mathcal{Z}_{k+1}\right)   \\
				& - (1+\mathbb{D}_{k+1}) \left( 4\tau\zeta_2 \sin^2\psi_{k+1} + \tau^2\mathcal{H}_{k+1} \right)
				- \frac{1}{2} 4\sigma^2 \cos^2\psi_{k}\sin^2\psi_{k} \Delta W_{k}^2 \\
				&-\frac{1}{2} (\mathbb{Y}_2+\mathbb{Y}_3)^2 - \mathbb{Y}_1(\mathbb{Y}_2+\mathbb{Y}_3) + \frac{1}{3} (\mathbb{Y}_1+\mathbb{Y}_2+\mathbb{Y}_3)^3 \\ 
				=\,&2\sigma\cos\psi_{k}\sin\psi_{k}\Delta W_{k} + \tau\sigma^2\cos^2\psi_{k} \cos(2\psi_{k}) 
				+ 2\tau \cos\psi_{k+1} \sin\psi_{k+1} \left(1- \mathcal{Z}_{k+1}\right)   \\
				& - 4\tau\zeta_2 \sin^2\psi_{k+1} 
				+ \sigma^2(\Delta W_{k}^2-\tau)\cos^2\psi_{k} \cos(2\psi_{k})  +  \tilde{\mathcal{R}}_{1,k+1}, 
			\end{align*}
			where $|\tilde{\mathcal{R}}_{1,k+1}|
			\leq C (\tau|\Delta W_{k}|+\tau^2) \left(  |\mathcal{Z}_{k+1}|^{8} +1 \right)
			+ |\Delta W_{k}|^3 + \Delta W_{k}^8$.  
			It follows that 
			\begin{align*}
				& \log\|B_{k+1}\|^2 - \log \|B_{k}\|^2 \\
				\leq\,&2\sigma\cos\psi_{k}\sin\psi_{k}\Delta W_{k} 
				+ 2\tau Q\left( \eta_{k+1},\bar{\eta}_{k+1},\psi_{k+1} \right) + \sigma^2(\Delta W_{k}^2-\tau) \cos^2\psi_{k} \cos(2\psi_{k}) 
				\\
				& + \tau\sigma^2\left(\cos^2\psi_{k} \cos(2\psi_{k})-\cos^2\psi_{k+1} \cos(2\psi_{k+1}) \right)  +  \tilde{\mathcal{R}}_{1,k+1}. 
			\end{align*}
			Finally, by using the similar estimates as \eqref{1/N logB}--\eqref{SLLN2}, we obtain 
			\begin{align*}
				\lambda^\tau &\leq \lim\limits_{N\rightarrow\infty}\frac{1}{2N\tau} \sum_{k=0}^{N-1} \left(2\sigma\cos\psi_{k}\sin\psi_{k}\Delta W_{k} 
				+ \sigma^2(\Delta W_{k}^2-\tau) \cos^2\psi_{k} \cos(2\psi_{k}) \right)
				\\
				&\quad + \lim\limits_{N\rightarrow\infty} \frac{1}{N} \sum_{k=1}^{N} Q(\eta_{k},\bar{\eta}_{k},\psi_{k}) 
				+ \lim\limits_{N\rightarrow\infty}\frac{1}{2N\tau} \sum_{k=0}^{N-1}  \tilde{\mathcal{R}}_{1,k+1} \\
				&\leq \mu^\tau(Q) + C\tau^{\frac{1}{2}}, 
			\end{align*}
			which leads to the desired conclusion. 
		\end{proof}
		
		Let  
		\begin{equation}\label{C and c}
			\bm{\mathcal{C}}_{k} = \left(\eta_{k}, \bar{\eta}_{k}, \cos\psi_{k}, \sin\psi_{k}\right)^{\top} \quad \text{and} \quad  \bm{c}(t) = \left(\eta(t), \dot{\eta}(t), \cos \psi(t), \sin \psi(t)\right)^{\top}. 
		\end{equation}
		The following proposition establishes the strong convergence between $\bm{\mathcal{C}}_{k}$ and $\bm{c}(t_k)$ in a finite time interval $[0,T]$ for a given $T=N\tau$ with $N\in\mathbb{N}$, which will be proved in Section \ref{Sec:proof}. 
		\begin{prop}\label{prop:convergence}
			For any $p\in(0,1)$, it holds that  
			\begin{equation*}
				\mathbb{E}\Big[ \sup_{0\leq k \leq N} \|\bm{\mathcal{C}}_{k}-\bm{c}(t_{k}) \|^{2p} \Big] 
				\leq C(\bm{c}(0),T) \tau^p. 
			\end{equation*}
		\end{prop}

		A key advantage of our method \eqref{BEM} lies in its ability to preserve the sign of the Lyapunov exponent $\lambda_{0}$ for the single mode solution of the original system, as established in the following corollary. This implies that the numerical single mode solution $(\eta_{k}, \bar{\eta}_{k})$ retains the same almost-sure stability (or instability) as the single mode solution $(\eta(t) , \dot{\eta}(t))$. 
		\begin{coro}
			There exist constants $a,\bar{a}\in(0,1)$ and $b,\bar{b}\in(0,\infty)$ such that for sufficiently small $\tau$ and for all $k\geq0$, \begin{equation}\label{lambda0 and lambdatau}
				\begin{aligned}
					|\lambda_{0} - \lambda^{\tau}|
					&\leq C\tau^{\frac{1}{2}} + b a^{t_{k}}\Gamma_{1}(\eta(0),\dot{\eta}(0),\psi(0))
					+ \bar{b} \bar{a}^{k\tau} \Gamma_{2}(\eta_{0},\bar{\eta}_{0},\psi_{0}) \\
					&\quad + \left| \mathbb{E} \left[Q(\eta(t_k),\dot{\eta}(t_k),\psi(t_k))\right] - \mathbb{E} \left[Q\left( \eta_{k},\bar{\eta}_{k},\psi_{k} \right) \right] \right|. 
				\end{aligned}
			\end{equation}
			In particular, there exist sufficiently large $K>0$ and sufficiently small $\tau_0>0$ such that  $\lambda^{\tau}$ is sign-consistent with the continuous exponent $\lambda_{0}$ whenever $k>K$ and $\tau<\tau_0$. 
		\end{coro}
		\begin{proof}
			Firstly, \eqref{lambda0 and lambdatau} can be derived by combining Theorems \ref{thm:lambda tau}, \ref{thm:lambda tau2}, and the estimates \eqref{ergodic lambda0}, \eqref{ergodicity for eta}. We next use the strong convergence result to derive the desired conclusion. 
			In fact, if $\lambda_{0}<0$, there exists a constant $T>0$ such that $b a^{t_{k}}\Gamma_{1}(\eta(0),\dot{\eta}(0),\psi(0))
			+ \bar{b} \bar{a}^{k\tau} \Gamma_{2}(\eta_{0},\bar{\eta}_{0},\psi_{0})<-\frac{\lambda_{0}}{4}$ for all $t_k>T$. For such fixed $T$, let $N=\lfloor T/\tau \rfloor$. Then by using Proposition \ref{prop:convergence} with $p>\frac{1}{2}$ and H\"older's inequality, there exists a constant $\tilde{\tau}_{0}>0$ such that when $\tau<\tilde{\tau}_0$, 
			\begin{align*}
				&\sup_{0\leq k \leq N}\left| \mathbb{E} \left[Q(\eta(t_k),\dot{\eta}(t_k),\psi(t_k))\right] - \mathbb{E} \left[Q\left( \eta_{k},\bar{\eta}_{k},\psi_{k} \right) \right] \right| \\
				&\leq \sup_{0\leq k \leq N} C \, \mathbb{E} \left[ \|\bm{\mathcal{C}}_{k}-\bm{c}(t_{k})\| (\|\bm{\mathcal{C}}_{k}\|+\|\bm{c}(t_{k})\|) \right] \\
				&\leq C \sup_{0\leq k \leq N}\left(\mathbb{E}\big[ \|\bm{\mathcal{C}}_{k}-\bm{c}(t_{k}) \|^{2p} \big]\right)^{\frac{1}{2p}} \left(\mathbb{E}\big[ (\|\bm{\mathcal{C}}_{k}\|+\|\bm{c}(t_{k})\|)^{\frac{2p}{2p-1}} \big]\right)^{\frac{2p-1}{2p}} 
				\leq C\tau^{\frac{1}{2}}
				<-\frac{\lambda_{0}}{4}. 
			\end{align*} 
			Finally, take $\tau_0<\tilde{\tau}_0$ sufficiently small such that $C\tau^{\frac{1}{2}}<-\frac{\lambda_{0}}{4}$ when $\tau<\tau_0$. In summary, we obtain that $|\lambda_{0} - \lambda^{\tau}|<-\frac{3\lambda_{0}}{4}$ when $\tau<\tau_0$, which implies $\lambda^{\tau}<0$. Similarly, the case $\lambda_{0}>0$ leads to $\lambda^{\tau}>0$. 
		\end{proof}

		\section{Proofs of propositions in Section \ref{Sec:numerical lambda}}\label{Sec:proof}
		In this section, we present the proofs of Propositions \ref{prop:ergodicity eta psi} and \ref{prop:convergence} in Section \ref{Sec:numerical lambda}.
		\subsection{Proof of Proposition \ref{prop:ergodicity eta psi}} 
		The proof relies on the verification of the Lyapunov condition and the minorization condition.
		\begin{proof}[Proof of Proposition \ref{prop:ergodicity eta psi}]
			We first prove the Lyapunov condition for $\{(\eta_{k}, \bar{\eta}_k, \psi_k)\}_{k\geq0}$. 
			Define the Lyapunov function $\Gamma_{2}:\tilde{\mathcal{N}}\rightarrow[1,\infty)$ as 
			\begin{equation}\label{Gamma2}
				\Gamma_{2}(v,x,\psi) := \kappa_{1} v^2 + x^2 + \delta_{4} vx + M, \qquad \text{for } \  \delta_{4}\in(0,\sqrt{\kappa_{1}}), \ M\geq1. 
			\end{equation}
			It can be seen that $\lim_{(v,x,\psi)\rightarrow\infty}\Gamma_{2}(v,x,\psi)=\infty$.  
			By using \eqref{VXbeta}, Taylor's expansion, Young's inequality, \eqref{Est:inequality}, and taking $0 < \delta_{4} < \min\{ \sqrt{\kappa_{1}} , \, 4\zeta_1\kappa_{1}(\kappa_{1}+2\zeta_1^2)^{-1} \}$, there exists a constant $c_7\in(0,1)$ such that   
			\begin{equation}\label{tilde Gamma}
				\begin{aligned}
					\Gamma_{2}(\eta_{k},\bar{\eta}_{k}+\sigma\Delta W_{k}) - \Gamma_{2}(\eta_{k+1},\bar{\eta}_{k+1}) 
					&\geq \left(4\tau\zeta_1-\tau\delta_{4}-\frac{2\tau\delta_{4}\zeta_1^2}{\kappa_{1}}\right)\bar{\eta}_{k+1}^2 + \frac{\tau\delta_{4}\kappa_{1}}{2} \eta_{k+1}^2 \\  
					&\geq c_7 \tau \Gamma_{2}\left( \eta_{k+1}, \bar{\eta}_{k+1}, \psi_{k+1} \right) 
					- c_7\tau M. 
				\end{aligned}
			\end{equation}
			It follows that 
			\begin{align*}
				&\quad \mathbb{E}\left[ \Gamma_{2}\left(\eta_{k+1},\bar{\eta}_{k+1},\psi_{k+1}\right) \big| \mathcal{F}_{t_k} \right]
				\leq \frac{1}{1+c_{7}\tau} \Gamma_{2}\left(\eta_{k},\bar{\eta}_{k},\psi_{k}\right) + \frac{\sigma^2\tau + c_7\tau M}{(1+c_{7}\tau)}. 
			\end{align*} 
			This leads to the Lyapunov condition for the process $\{(\eta_{k},\bar{\eta}_{k},\psi_{k})\}_{k\geq0}$.

			We next aim to show that, for any $(\eta_{0}, \bar{\eta}_{0}, \psi_{0})^{\top}\in\tilde{\mathcal{N}}$ and $\delta\in(0,1)$, one can find $\Delta W_{i}$, $i=0,1,2$ such that   $(\eta_{3}, \bar{\eta}_{3}, \psi_{3})^{\top} \in \mathcal{B}_{\delta}(0,0,\arctan\frac{1}{\tau})$, where $\mathcal{B}_{\delta}(\bm{x})\subset\mathbb{R}^3$ represents the ball of radius $\delta$ centered at $\bm{x}$. 
			
			Let $\varpi=1+2\tau\zeta_1+\tau^2\kappa_{1}$. Then we can solve from \eqref{VXbeta} that 
			\begin{equation}\label{eta k}
				\begin{pmatrix}
					\eta_{k+1} \\ \bar{\eta}_{k+1}
				\end{pmatrix}
				= \frac{1}{\varpi}\begin{pmatrix}
					1+2\tau\zeta_1 & \tau \\
					-\tau\kappa_{1} & 1
				\end{pmatrix}
				\begin{pmatrix}
					\eta_{k} \\ \bar{\eta}_{k}
				\end{pmatrix}
				+ \frac{1}{\varpi} \begin{pmatrix}
					\sigma\tau \\ \sigma 
				\end{pmatrix}
				\Delta W_{k}
				=:
				\Lambda_{1} \begin{pmatrix}
					\eta_{k} \\ \bar{\eta}_{k}
				\end{pmatrix}
				+ \Lambda_{2} \Delta W_{k}. 
			\end{equation}
			Let the matrix $\Lambda_{3} = (\Lambda_{1}\Lambda_{2} \   \Lambda_{2})\in\mathbb{R}^{2\times2}$. 
			By recursion, we obtain 
			\begin{equation}\label{V3}
				\begin{pmatrix}
					\eta_{3} \\ \bar{\eta}_{3}
				\end{pmatrix}
				=
				\Lambda_{1}^3 \begin{pmatrix}
					\eta_{0} \\ \bar{\eta}_{0}
				\end{pmatrix}
				+ \Lambda_{1}^2 \Lambda_{2} \Delta W_{0}
				+ \Lambda_{3} 
				\begin{pmatrix}
					\Delta W_{1} \\ \Delta W_{2}
				\end{pmatrix}. 
			\end{equation}
			Since the matrix $\Lambda_{3}$ is invertible, for any given $\eta_3^*,\bar{\eta}_3^*, \Delta W_{0}\in\mathbb{R}$, we can choose the following $\Delta W_{1}$ and $\Delta W_{2}$ such that \eqref{V3} holds:  
			\begin{equation}\label{V3*}
				\begin{pmatrix}
					\Delta W_{1} \\ \Delta W_{2}
				\end{pmatrix}
				= \Lambda_{3}^{-1}\begin{pmatrix}
					\eta_{3}^{*} \\ \bar{\eta}_{3}^{*}
				\end{pmatrix}
				-
				\Lambda_{3}^{-1}\Lambda_{1}^3 \begin{pmatrix}
					\eta_{0} \\ \bar{\eta}_{0}
				\end{pmatrix}
				- \Lambda_{3}^{-1}\Lambda_{1}^2 \Lambda_{2} \Delta W_{0}. 
			\end{equation} 
			It follows from \eqref{V3*} that $\Delta W_{1}$ and $\Delta W_{2}$ are linear functions of $\Delta W_{0}$, namely, 
			\begin{equation*}
				\Delta W_{1} = d_1 + d_2 \Delta W_{0} \quad \text{and} \quad \Delta W_{2} = d_3 + d_4 \Delta W_{0}, \qquad \text{with } d_2 = -\frac{2(1+\tau\zeta_1)}{\varpi}, \quad  d_4 = \frac{1}{\varpi}. 
			\end{equation*}
			
			For $\bar{\beta}_{k}$ and $\beta_{k}$ determined by \eqref{VXbeta}, we define 
			\begin{equation*}
				\rho_{k} := \frac{\bar{\beta}_{k}}{\beta_{k}} = \tan \psi_{k}. 
			\end{equation*}
			When $\beta_{k}\neq0$, namely, $|\rho_{k}|<\infty$, it follows from \eqref{VXbeta} that 
			\begin{equation*} 
				\frac{(1+2\tau\zeta_2)\rho_{k+1}}{1-\tau\rho_{k+1}} = \rho_{k}-\tau\left(\kappa_{2}+2\zeta_1\bar{\eta}_{k+1}+\kappa_{1}\eta_{k+1}\right) +\sigma\Delta W_{k}, 
			\end{equation*}
			which, by taking $k=2$ and recalling $\mathcal{Z}_{k}=\kappa_{2} + 2\zeta_1 \bar{\eta}_{k} + \kappa_{1} \eta_{k}$, implies that 
			\begin{equation*}
				\rho_{3} \left( 1+2\tau\zeta_2+\tau\left(\rho_{2}-\tau\mathcal{Z}_{3}+\sigma\Delta W_{2}\right) \right)=  \rho_{2}-\tau\mathcal{Z}_{3}+\sigma\Delta W_{2}. 
			\end{equation*}
			For any $\eta_3^*, \bar{\eta}_3^*, \rho_{3}^{*}\in\mathbb{R}$, we denote 
			\begin{equation*}
				\mathcal{Z}_{k}^{*}=\kappa_{2} + 2\zeta_1 \bar{\eta}_{k}^{*} + \kappa_{1} \eta_{k}^{*} \quad \text{and} \quad \varrho = \frac{\left(1+2\tau\zeta_2\right) \rho_{3}^{*} }{1-\tau\rho_{3}^{*}} + \tau\mathcal{Z}_{3}^{*}. 
			\end{equation*}
			The aim is now reduced to verifying whether there exists a $\Delta W_{0}\in\mathbb{R}$ such that $\rho_{2}+\sigma\Delta W_{2}=\varrho$ for some properly chosen $\varrho$. In the following we address this issue.

			It can be verified that both $\rho_{2}$ and $\Delta W_{2}$ are functions of $\Delta W_{0}$. More precisely, 
			\begin{equation*}
				\rho_{2} = \rho_{2}(\Delta W_{0}) \quad \text{and} \quad \Delta W_{2} = d_3 + d_4 \Delta W_{0}, \quad \text{with } d_4=\varpi^{-1}>0. 
			\end{equation*}
			The idea is to investigate the asymptotic behavior of $\rho_{2}(\Delta W_{0})$ as $\Delta W_{0}\rightarrow\infty$. 
			Denote 
			\begin{equation*}
				e_0 = \tau\rho_0-\tfrac{\tau^2\kappa_{1}}{\varpi}\eta_0+\tau(\tfrac{1}{\varpi}-1)\bar{\eta}_0-\tau^2\kappa_{2} \quad \text{and} \quad \tilde{e}_0 = \tfrac{\tau\kappa_{1}(\tau^2\kappa_{1}-1)}{\varpi^2} \eta_0 - \tfrac{2\tau\zeta_1+2\tau^2\kappa_{1}}{\varpi^2}\bar{\eta}_0 + \tfrac{\sigma}{\varpi} d_1 - \tau\kappa_{2}. 
			\end{equation*}
			It follows that 
			\begin{align*}
				\rho_{2} 
				&= \frac{\tfrac{1}{\tau}e_0+\tfrac{\sigma}{\varpi}\Delta W_{0}+\left(\tilde{e}_0-\tfrac{2\sigma}{\varpi}\Delta W_{0}\right)\left(1+2\tau\zeta_2+e_0+\tfrac{\sigma\tau}{\varpi}\Delta W_{0}\right)}{e_0+\tfrac{\sigma\tau}{\varpi}\Delta W_{0}+\left(1+2\tau\zeta_2+\tau\tilde{e}_0-\tfrac{2\sigma\tau}{\varpi}\Delta W_{0}\right)(1+2\tau\zeta_2+e_0+\tfrac{\sigma\tau}{\varpi}\Delta W_{0})}. 
			\end{align*}
			The above equality holds when $\Delta W_{0} \neq w_1$ and $\Delta W_{0} \neq w_2$, where $w_1, w_2$ are roots of the denominator of the above equality.
			Therefore, 
			\begin{align*}
				\lim\limits_{\Delta W_{0}\rightarrow\pm\infty} \rho_{2} =\frac{1}{\tau}, 
			\end{align*}
			which, together with $\Delta W_{2} = d_3 + d_4 \Delta W_{0}$ and $d_4>0$, leads to    
			\begin{equation*}
				\lim\limits_{\Delta W_{0}\rightarrow\pm\infty} \left(\rho_{2}+\sigma\Delta W_{2}\right) = \pm\infty. 
			\end{equation*}
			This means that for any $\delta\in(0,1)$, $\rho_{3}^*\in\left(\tfrac{1}{\tau}-\delta , \tfrac{1}{\tau}\right)\cup\left(\tfrac{1}{\tau} , \tfrac{1}{\tau}+\delta\right)$, and $\eta_3^*, \bar{\eta}_3^* \in (-\delta,\delta)$, we can choose sufficiently large (or small) $\Delta W_{0}$ such that  
			\begin{equation*}
				\rho_{2}+\sigma\Delta W_{2} = \varrho, \qquad \text{with } \ \varrho = \frac{\left(1+2\tau\zeta_2\right) \rho_{3}^{*} }{1-\tau\rho_{3}^{*}} + \tau\mathcal{Z}_{3}^{*}. 
			\end{equation*}
			With these $\eta_3^*, \bar{\eta}_3^* \in (-\delta,\delta)$ and $\Delta W_{0}$, we can further select $\Delta W_{1}$ and $\Delta W_{2}$ so that \eqref{V3*} is satisfied. 
			In addition, we note that the above analysis excludes the case $\beta_0=0$. In fact, if $\beta_0=0$, by $\beta_1=\beta_0+\tau\bar{\beta}_1$, we obtain $\rho_1=\frac{1}{\tau}$, which returns to the case $\beta_0\neq0$. 
			
			In summary, for any $\delta\in(0,1)$ and $\eta_0, \bar{\eta}_0, \rho_0\in\mathbb{R}$, we can find an open set $\mathcal{B}\subset\mathbb{R}^3$ such that when $\left(\Delta W_{0}, \Delta W_{1}, \Delta W_{2}\right)^{\top}\in\mathcal{B}$ with $\Delta W_{0}\neq w_1, w_2$,  
			it holds that  
			\begin{equation*}
				\left( \eta_3, \bar{\eta}_3 , \rho_3 \right) \in \left(-\delta,\delta\right)^2 \times \left(\tfrac{1}{\tau}-\delta , \tfrac{1}{\tau}+\delta\right). 
			\end{equation*}
			It follows from $\psi_{k}=\arctan \rho_{k}\mod \pi$ that for all $(\eta_0, \bar{\eta}_0, \psi_0)^{\top} \in \tilde{\mathcal{N}}$, 
			\begin{align*}
				&\mathbb{P}\left( \left( \eta_3, \bar{\eta}_3 , \psi_3 \right) \in \mathcal{B}_{\delta}(0,0,\arctan\tfrac{1}{\tau}) \right) 
				\geq
				\mathbb{P}\left( \left( \eta_3, \bar{\eta}_3 , \rho_3 \right) \in \mathcal{B}_{\delta}(0,0,\tfrac{1}{\tau}) \right) \\ 
				&\geq \mathbb{P}\big( \{\left(\Delta W_{0}, \Delta W_{1}, \Delta W_{2}\right)^{\top}\in\mathcal{B}\} \cap \{\Delta W_{0}\neq w_1, w_2\} \big) > 0, 
			\end{align*}
			where we have used the fact that the Wiener measure of any interval is positive and $\mathbb{P}(\{\Delta W_{0}\neq w_1, w_2\})=1$.  
			Furthermore, we write $(\eta_3, \bar{\eta}_3 , \psi_3)^{\top} = \Phi_3\big(\eta_0, \bar{\eta}_0 , \psi_0, \{\Delta W_i\}_{i=0,1,2}\big)$. 
			Since $\Phi_3$ is a rational polynomial with respect to $\{\Delta W_{i}\}_{i=0,1,2}$ and $\Delta W_{i}$ has $C^{\infty}$ density, we can find a compact set $C$ on which the transition probability for the process $(\eta_3, \bar{\eta}_3 , \psi_3)^\top$ has a continuous density when $\tau$ is sufficiently small. 
			This verifies the Assumption 2.1 in \cite{Mattingly2002}, which, by \cite[Lemma 2.3]{Mattingly2002}, implies the minorization condition for the process $\{(\eta_{k}, \bar{\eta}_{k}, \psi_{k})\}_{k\geq0}$.  
			
			Finally, one can verify that the function $Q$ defined by \eqref{Qh} satisfies $|Q|\leq \Gamma_{2}$ for some properly chosen $M$. 
			It follows from \cite[Theorem 2.5]{Mattingly2002} that the process $(\eta_{k},\bar{\eta}_{k},\psi_{k})$ exists a unique stationary measure $\mu^{\tau}$ such that \eqref{ergodicity for eta} holds. 
			The existence of the stationary measure for $(\eta_{k},\bar{\eta}_{k})$ is analogous, and we omit the details for brevity.
		\end{proof}

		\subsection{Proof of Proposition \ref{prop:convergence}}
		In this subsection we prove the strong convergence between $\bm{\mathcal{C}}_{k}=(\eta_{k}, \bar{\eta}_{k}, \cos\psi_{k}, \sin\psi_{k})^{\top}$ and $\bm{c}(t_{k})=(\eta(t_k), \dot{\eta}(t_k), \cos \psi(t_k), \sin \psi(t_k))^{\top}$ in a finite time interval $[0,T]$. Recall that $t_{k}=k\tau$, $k=0,1,2,...,N$ are gridpoints and $\tau$ is the step size.  
		According to \eqref{xi}, It\^o's formula, and Lemma \ref{lemma:A3}, we have $\bm{\mathcal{C}}_{0}=\bm{c}(0)$ and for $k\geq0$, 
		\begin{align*}
			&\bm{c}(t_{k+1}) = \bm{c}(t_{k}) + \int_{t_{k}}^{t_{k+1}} \bm{\varPhi}(\bm{c}(s)) \rmd s + \int_{t_{k}}^{t_{k+1}} \bm{\varPsi}(\bm{c}(s)) \rmd W(s), \\
			&\bm{\mathcal{C}}_{k+1} = \bm{\mathcal{C}}_{k} + \tau \bm{\varPhi}(\bm{\mathcal{C}}_{k+1}) + \bm{\varPsi}(\bm{\mathcal{C}}_{k}) \Delta W_{k} + \varTheta(\bm{\mathcal{C}}_{k+1}) 
			+ (\tau-\Delta W_{k}^2)  \tilde{\varTheta}(\bm{\mathcal{C}}_{k}), 
		\end{align*}
		where $\bm{\varPhi}$, $\bm{\varPsi}$, $\tilde{\varTheta}$ are given by Lemma \ref{lemma:A3}, and $\sup_{0\leq k \leq N}\mathbb{E}\big[\|\varTheta(\bm{\mathcal{C}}_{k})\|^{p}\big] \leq C(p) \tau^{\frac{3p}{2}}$.

		We first present the exponential integrability properties of $\bm{c}(t)$ and $\bm{\mathcal{C}}_{k}$. 
		\begin{lemma}
			For any $T>0$ and sufficiently small $\tau$, there exists a constant $c_8>0$ such that 
			\begin{equation}\label{exp int X}
				\mathbb{E}\left[ e^{c_8 \tau \sum_{k=1}^{N} \|\bm{c}(t_{k})\|^2}\right]
				+ \mathbb{E}\left[ e^{c_8 \tau \sum_{k=1}^{N} \|\bm{\mathcal{C}}_{k}\|^2}\right] 
				\leq C(\bm{c}(0),T), 
			\end{equation}
			where the constant $C(\bm{c}(0),T)$ is independent of $\tau$. 
		\end{lemma}
		\begin{proof}
			We first derive the exponential integrability properties of $\bm{\eta}(t) := (\eta(t),\dot{\eta}(t))^{\top}$. From the first equation in \eqref{linearization}, we obtain $\rmd \bm{\eta}(t) = A\bm{\eta}(t) \rmd t + B \rmd W(t)$, where $A=\left(\genfrac{}{}{0pt}{}{0 \ \ \ 1}{-\kappa_1 \, -2\zeta_1}\right)$ and $B=(0,\sigma)^{\top}$. 
			Thus $\bm{\eta}(t)=e^{At} \bm{\eta}(0) + g(t)$ where $g(t)=\int_{0}^{t} e^{A(t-s)} B \rmd W(s)$. One can verify that $\mathbb{E}[\|g\|_{L^2([0,T])}^2] = \int_{0}^{T} \mathbb{E} \|g(t)\|^2 \rmd t = \int_{0}^{T} \int_{0}^{t} \|e^{A(t-s)}B\|^2 \rmd s \rmd t < \infty$, which means that $g$ is a centered Gaussian element in the separable Hilbert space $L^2([0,T];\mathbb{R}^2)$. By Fernique's theorem, there exists a constant $c_g>0$ such that $\mathbb{E}[\exp(2c_g \|g\|^2_{L^2})] < \infty$. It follows that 
			\begin{equation*}
				\mathbb{E}\left[ e^{c_g\int_{0}^{T} \|\bm{\eta}(t)\|^2 \rmd t} \right] 
				\leq \mathbb{E}\left[ e^{2c_g\int_{0}^{T}\|e^{At}\bm{\eta}(0)\|^2 \rmd t + 2c_g \|g\|^2_{L^2}} \right] 
				\leq C(\bm{\eta}(0), T)< \infty.
			\end{equation*}   
			In addition, by $\bm{\eta}(t_k) = \bm{\eta}(t) + \int_{t}^{t_k}A\bm{\eta}(r) \rmd r + B (W(t_k)-W(t))$ with $t\in[t_{k-1},t_k]$, we obtain 
			\begin{align*}
				&\tau \sum_{k=1}^{N} \|\bm{\eta}(t_k)\|^2 
				= \sum_{k=1}^{N}\int_{t_{k-1}}^{t_{k}}  \|\bm{\eta}(t_k)\|^2 \rmd t \\
				&\leq \sum_{k=1}^{N}\int_{t_{k-1}}^{t_{k}}  \left(3\|\bm{\eta}(t)\|^2 + 3\tau \int_{t_{k-1}}^{t_{k}} \|A\bm{\eta}(r)\|^2 \rmd r + 6\sigma^2  (|W(t)|^2+|W(t_k)|^2)\right) \rmd t \\
				&\leq C \int_{0}^{T}  \|\bm{\eta}(t)\|^2 \rmd t 
				+ 12\sigma^2T \sup_{t\in[0,T]}|W(t)|^2. 
			\end{align*}
			By Fernique's theorem, there exists a constant $c_w>0$ such that $\mathbb{E}[\exp(c_w \|W\|^2_{C([0,T])})] < \infty$. In summary, there exists a constant $\tilde{c}_8\leq\min\{ c_g/(2C), c_w/(24\sigma^2T) \}$ such that 
			\begin{equation*}
				\mathbb{E}\big[ e^{\tilde{c}_8 \tau \sum_{k=1}^{N} \|\bm{\eta}(t_k)\|^2} \big]  
				\leq \left(\mathbb{E}\big[ e^{2C \tilde{c}_8 \int_{0}^{T}  \|\bm{\eta}(t)\|^2 \rmd t} \big] \right)^{\frac{1}{2}} 
				\left(\mathbb{E}\big[ e^{24\tilde{c}_8 \sigma^2T \|W\|^2_{C([0,T])}} \big] \right)^{\frac{1}{2}} 
				\leq C(\bm{\eta}(0), T)< \infty.
			\end{equation*}

			We next consider the exponential integrability properties of $(\eta_k, \bar{\eta}_k)^{\top}$. 
			Let $\Gamma_{3}(x_1,x_2) := \kappa_{1} x_1^2 + x_2^2 + \delta_{4} x_1x_2$ for $\delta_{4}\in(0,\sqrt{\kappa_1})$. 	
			By Taylor's expansion and Young's inequality, and by taking $\delta_{4}$ sufficiently small, we obtain  
			\begin{align*}
				\Gamma_{3}(\eta_k, \bar{\eta}_k+\sigma\Delta W_k) 
				&\geq  \Gamma_{3}(\eta_{k+1}, \bar{\eta}_{k+1}) 
				+\left( 4\tau\zeta_1 -\tau\delta_{4}(1+\tfrac{2\zeta_1^2}{\kappa_1})\right) \bar{\eta}_{k+1}^2 + \frac{\tau\delta_{4}\kappa_1}{2}\eta_{k+1}^2 \\
				&\geq  \Gamma_{3}(\eta_{k+1}, \bar{\eta}_{k+1}) + C\tau (\eta_{k+1}^2 + \bar{\eta}_{k+1}^2),
			\end{align*}
			which, upon summing from $k=0$ to $N-1$, leads to  
			\begin{equation}\label{exp int 2}
				\epsilon C\tau \sum_{k=1}^{N} (\eta_{k}^2+\bar{\eta}_{k}^2) 
				\leq \epsilon (\kappa_1\eta_{0}^2 +\bar{\eta}_{0}^2 +  \delta_{4}\eta_{0}\bar{\eta}_{0})  
				+ \epsilon  \sum_{k=0}^{N-1} \left(\sigma^2\Delta W_{k}^2 + \delta_{4}\sigma \eta_{k} \Delta W_{k}+  2\sigma \bar{\eta}_{k} \Delta W_{k}  \right),  
			\end{equation}
			where $\epsilon>0$ is a small constant chosen later. 
			For any $\varsigma>0$, it can be verified that \begin{equation}\label{exp int 3}
				\mathbb{E} \big[e^{\varsigma\sum_{k=0}^{N-1} x_{k} \Delta W_{k} - \frac{\varsigma^2\tau}{2}\sum_{k=0}^{N-1} x_{k}^2}\big]=1 \qquad \text{for } \  x_{k}=\eta_{k} \ \text{and} \  \bar{\eta}_{k}. 
			\end{equation}
			By \eqref{exp int 2} and Young's inequality, we have  
			\begin{align*}
				\epsilon C\tau \sum_{k=1}^{N} \left( \eta_{k}^2+\bar{\eta}_{k}^2 \right) 
				&\leq C\epsilon (\eta_{0}^2+\bar{\eta}_{0}^2) 
				+ \epsilon\sigma^2 \sum_{k=0}^{N-1} \Delta W_{k}^2 
				+ \delta_{4}\epsilon\sigma  \sum_{k=0}^{N-1}  \eta_{k} \Delta W_{k} 
				- 2\delta_{4}^2\epsilon^2\sigma^2\tau  \sum_{k=0}^{N-1}  \eta_{k}^2\\
				&\quad + 2\epsilon\sigma  \sum_{k=0}^{N-1}  \bar{\eta}_{k} \Delta W_{k} 
				- 8\epsilon^2\sigma^2\tau  \sum_{k=0}^{N-1}  \bar{\eta}_{k}^2  
				+ 8\epsilon^2\sigma^2\tau  \sum_{k=0}^{N-1}  \bar{\eta}_{k}^2 
				+ 2\delta_{4}^2\epsilon^2\sigma^2\tau  \sum_{k=0}^{N-1}  \eta_{k}^2, 
			\end{align*}
			which, together with H\"older's inequality and \eqref{exp int 3}, implies that for sufficiently small $\epsilon$, 
			\begin{equation}\label{exp int 4}
				\begin{aligned}
					&\mathbb{E}\left[e^{\tau (C\epsilon-2\delta_{4}^2\epsilon^2\sigma^2) \sum_{k=1}^{N} \eta_{k}^2 + \tau (C\epsilon-8\epsilon^2\sigma^2) \sum_{k=1}^{N} \bar{\eta}_{k}^2}\right] \\
					\leq \, & 
					\left(\mathbb{E}\left[ e^{2C(\bar{\eta}_{0}^2 + \eta_{0}^2) 
						+ 2\epsilon\sigma^2 \sum_{k=0}^{N-1} \Delta W_{k}^2} \right]\right)^{\frac{1}{2}} \left(\mathbb{E}  \left[e^{4\delta_{4}\epsilon\sigma  \sum_{k=0}^{N-1}  \eta_{k} \Delta W_{k} 
						- 8\delta_{4}^2\epsilon^2\sigma^2\tau  \sum_{k=0}^{N-1}  \eta_{k}^2}\right]\right)^{\frac{1}{4}}  \\
					& \quad \left(\mathbb{E}  \left[e^{8\epsilon\sigma  \sum_{k=0}^{N-1}  \bar{\eta}_{k} \Delta W_{k} 
						- 32\epsilon^2\sigma^2\tau  \sum_{k=0}^{N-1}  \bar{\eta}_{k}^2}\right]\right)^{\frac{1}{4}} \\
					= \, &  
					\left(\mathbb{E}\left[ e^{2C(\bar{\eta}_{0}^2 + \eta_{0}^2)} \right]\right)^{\frac{1}{2}} \left(\mathbb{E}\left[ e^{2\epsilon\sigma^2 \sum_{k=0}^{N-1} \Delta W_{k}^2} \right]\right)^{\frac{1}{2}}. 
				\end{aligned} 
			\end{equation}
			In view of the independence of $\Delta W_{k}$, $\mathbb{E}[e^{c\Delta W_k^2}] = (1-2c\tau)^{-\frac{1}{2}}$ for $2c\tau<1$, and $(1+\bar{c}\tau)^{1/\tau} \leq e^{\bar{c}}$ for $\tau\in(0,1)$, 
			we obtain that when $8\epsilon\sigma^2\tau<1$, 
			\begin{align*}
				&\left(\mathbb{E} \left[e^{2\epsilon\sigma^2 \sum_{k=0}^{N-1} \Delta W_{k}^2}\right] \right)^{\frac{1}{2}} 
				= \left(\prod_{k=0}^{N-1} \mathbb{E} \left[e^{2\epsilon\sigma^2 \Delta W_{k}^2}\right] \right)^{\frac{1}{2}}
				= (1-4\epsilon\sigma^2\tau)^{-\frac{T}{4\tau}} 
				\leq (1+8\epsilon\sigma^2\tau)^{\frac{T}{4\tau}} 
				\leq e^{2\epsilon\sigma^2T}. 
			\end{align*} 
			Consequently, by taking a sufficiently small $\epsilon$ in \eqref{exp int 4}, there exists a constant $\bar{c}_8$ such that  
			\begin{equation*}
				\mathbb{E}\left[ e^{\bar{c}_8\tau \sum_{k=1}^{N} (\eta_k^2+\bar{\eta}_k^2)}\right] \leq C(\bm{\eta}(0),T). 
			\end{equation*}  
			In addition, in view of 
			\begin{equation*}
				\mathbb{E}\left[e^{\int_{0}^{T} (\cos^2\psi(s)+\sin^2\psi(s)) \rmd s} \right] 
				+ \mathbb{E}\left[ e^{\tau\sum_{k=1}^{N} (\cos^2\psi_k+ \sin^2\psi_k)}\right] 
				= 2e^{T}, 
			\end{equation*} 
			we can take $c_8 = \min\{ \tilde{c}_8, \bar{c}_8 \}$ to yield the desired conclusion. 
		\end{proof}

		With the exponential integrability properties in hand, we are now ready to prove Proposition \ref{prop:convergence}. 
		\begin{proof}[Proof of Proposition \ref{prop:convergence}]
			Let $e_{k} := \bm{c}(t_{k}) - \bm{\mathcal{C}}_{k}$. Then 
			\begin{align*}
				e_{k+1}
				&= e_{k} 
				+ \tau \left( \bm{\varPhi}(\bm{c}(t_{k})) - \bm{\varPhi}(\bm{\mathcal{C}}_{k}) \right) 
				+ \left( \bm{\varPsi}(\bm{c}(t_k)) - \bm{\varPsi}(\bm{\mathcal{C}}_k)\right) \Delta W_k \\
				&\quad + \underbrace{\int_{t_k}^{t_{k+1}} \left( \bm{\varPsi}(\bm{c}(s)) - \bm{\varPsi}(\bm{c}(t_k))\right) \rmd W(s)- (\tau-\Delta W_{k}^2) \tilde{\varTheta}(\bm{\mathcal{C}}_{k})}_{=:\mathcal{E}_k}  \\
				&\quad
				+ \underbrace{\int_{t_k}^{t_{k+1}}  \left(\bm{\varPhi}(\bm{c}(s))-\bm{\varPhi}(\bm{c}(t_{k}))\right) \rmd s 
					+ \tau \left( \bm{\varPhi}(\bm{\mathcal{C}}_{k}) - \bm{\varPhi}(\bm{\mathcal{C}}_{k+1}) \right) 
					- \varTheta(\bm{\mathcal{C}}_{k+1})}_{=:\mathcal{Q}_{k}}. 
			\end{align*}
			For the functions $\bm{\varPhi}$ and $\bm{\varPsi}$ defined by Lemma \ref{lemma:A3}, there exist constants $L_{\bm{\varPhi}}$ and $L_{\bm{\varPsi}}$ such that 
			\begin{align*}
				\|\bm{\varPhi}(\bm{x}) - \bm{\varPhi}(\bm{y})\| 
				&\leq L_{\bm{\varPhi}} \left(1+\|\bm{x}\|+\|\bm{y}\|\right) \|\bm{x} - \bm{y}\|, \qquad \forall \, \bm{x}, \bm{y} \in \mathbb{R}^4, \\
				\|\bm{\varPsi}(\bm{x}) - \bm{\varPsi}(\bm{y})\| 
				&\leq L_{\bm{\varPsi}} \|\bm{x} - \bm{y}\|,   \qquad \forall \, \bm{x}, \bm{y} \in \mathbb{R}^4,
			\end{align*}
			which, together with $(a+b+c+d)^2\leq4(a^2+b^2+c^2+d^2)$ and Young's inequality, leads to 
			\begin{align*}
				\|e_{k+1}\|^2  
				&\leq \left(1+2\tau L_{\bm{\varPhi}}  \left(1+\|\bm{\mathcal{C}}_{k}\|+\|\bm{c}(t_{k})\|\right) + 4L_{\bm{\varPsi}}^2 \Delta W_k^2 \right) \|e_{k}\|^2 
				+ \tau \|e_{k}\|^2 + \tau^{-1} \|\mathcal{Q}_{k}\|^2 
				+ 4\|\mathcal{E}_k\|^2 \\
				&\quad 
				+ 4 \|\mathcal{Q}_{k}\|^2  + 2\left\langle e_{k} , \mathcal{E}_k+\left( \bm{\varPsi}(\bm{c}(t_k)) - \bm{\varPsi}(\bm{\mathcal{C}}_k)\right) \Delta W_k\right\rangle
				+ 4\tau^2 \left\| \bm{\varPhi}(\bm{c}(t_{k})) - \bm{\varPhi}(\bm{\mathcal{C}}_{k}) \right\|^2 \\
				&\leq \left(1+C(\epsilon)\tau+\epsilon\tau \|\bm{\mathcal{C}}_{k}\|^2+\epsilon\tau\|\bm{c}(t_{k})\|^2+4L_{\bm{\varPsi}}^2 \tau\right) \|e_{k}\|^2 \\
				&\quad + 2\left\langle e_{k} , \mathcal{E}_k+\left( \bm{\varPsi}(\bm{c}(t_k)) - \bm{\varPsi}(\bm{\mathcal{C}}_k)\right) \Delta W_k\right\rangle 
				+ 4L_{\bm{\varPsi}}^2 (\Delta W_k^2-\tau) \|e_{k}\|^2 \\
				&\quad + 4\|\mathcal{E}_k\|^2 + (4+\tau^{-1}) \|\mathcal{Q}_{k}\|^2 
				+ C\tau^2\left(1+\|\bm{\mathcal{C}}_{k}\|^2+\|\bm{c}(t_{k})\|^2\right). 
			\end{align*} 
			
			We next temporarily denote $X_{n} = \|e_{n}\|^2$ and 
			\begin{align*}
				M_n &= \sum_{k=0}^{n-1} 2\left\langle e_{k} , \mathcal{E}_k+\left( \bm{\varPsi}(\bm{c}(t_k)) - \bm{\varPsi}(\bm{\mathcal{C}}_k)\right) \Delta W_k\right\rangle 
				+ 4L_{\bm{\varPsi}}^2 \sum_{k=0}^{n-1} (\Delta W_k^2-\tau) \|e_{k}\|^2, \\
				G_k &= C(\epsilon)\tau + \epsilon\tau \|\bm{\mathcal{C}}_{k}\|^2 + \epsilon \tau \|\bm{c}(t_k)\|^2 + 4L_{\bm{\varPsi}}^2 \tau, \\
				F_n &=4\sum_{k=0}^{n-1} \|\mathcal{E}_k\|^2
				+ (4+\tau^{-1}) \sum_{k=0}^{n-1} \|\mathcal{Q}_{k}\|^2 
				+ C\tau^2 \sum_{k=0}^{n-1}  \left(1+\|\bm{\mathcal{C}}_{k}\|^2+\|\bm{c}(t_{k})\|^2\right). 
			\end{align*}
			It follows that 
			\begin{equation*}
				X_n \leq \sum_{k=0}^{n-1} G_k X_k + M_n + F_n, 
			\end{equation*}
			where $M_n$ is a martingale, $X_n$, $G_n$, and $F_n$ are sequences of nonnegative and adapted processes with $\mathbb{E}[X_0]=0$. 
			By using the stochastic Gronwall inequality (see \cite[Theorem 1]{Kruse2018}), we obtain that for any $p\in(0,1)$, $\mu, \nu \in [1,\infty]$ with $\frac{1}{\mu}+\frac{1}{\nu}=1$, and $p\nu<1$
			, 	\begin{equation*}
				\mathbb{E}\Big[ \sup_{0\leq k \leq N} X_k^p \Big] 
				\leq \left(1+\frac{1}{1-\nu p}\right)^{\frac{1}{\nu}} \left(\mathbb{E}\left[ e^{p\mu\sum_{k=0}^{N-1} G_k} \right]\right)^{\frac{1}{\mu}} 
				\Big(\mathbb{E}\Big[ \sup_{0\leq n \leq N} F_n \Big] \Big)^p. 
			\end{equation*}
			We can take $\epsilon$ sufficiently small such that 
			\begin{equation*}
				p\mu\sum_{k=0}^{N-1} G_k 
				\leq 
				\frac{c_8\tau}{4}\sum_{k=1}^{N} \left( \|\bm{\mathcal{C}}_{k}\|^2 + \|\bm{c}(t_k)\|^2 \right) + C(\bm{c}(0),T). 
			\end{equation*}
			Thus by H\"older's inequality and \eqref{exp int X}, we have 
			\begin{align*}
				\mathbb{E}\big[ e^{p\mu\sum_{k=0}^{N-1} G_k} \big] 
				\leq \left(\mathbb{E}\big[e^{2C(\bm{c}(0),T)}\big]\right)^{\frac{1}{2}} \left(\mathbb{E}\big[ e^{c_8\tau\sum_{k=1}^{N}  \|\bm{\mathcal{C}}_{k}\|^2} \big]
				\, \mathbb{E}\big[ e^{c_8\tau\sum_{k=1}^{N}  \|\bm{c}(t_{k})\|^2} \big]\right)^{\frac{1}{4}} 
				\leq C(\bm{c}(0),T). 
			\end{align*}  
			Finally, by the moment boundedness of $\bm{c}(t_{k})$ and $\bm{\mathcal{C}}_{k}$, we obtain 
			\begin{equation*}
				\mathbb{E}\Big[ \sup_{0\leq n \leq N} F_n \Big]
				\leq   
				\sum_{k=0}^{N-1} \left(4\mathbb{E}[\|\mathcal{E}_k\|^2]
				+ (4+\tau^{-1})  \mathbb{E}[\|\mathcal{Q}_{k}\|^2] 
				+ C\tau^2 \mathbb{E}\left[ 1+\|\bm{\mathcal{C}}_{k}\|^2+\|\bm{c}(t_{k})\|^2\right]\right) 
				\leq C\tau, 
			\end{equation*}
			which completes the proof.  
		\end{proof}

		\section{Numerical experiments}\label{Sec:experiment}
		In this section we report several numerical experiments to validate the theoretical results. We begin by computing numerical Lyapunov exponents that quantify the stability of the single mode solution and, for comparison, of the four-dimensional system \eqref{system}.  
		The left panel of Figure~\ref{Fig:lyp} shows the numerical Lyapunov exponents $\lambda^{\tau}$ for the single mode solution as a function of parameters $\sigma$ and $\zeta_2$ (with the remaining parameters fixed as in the caption).
		The exponent increases as $\sigma$ grows, and regions with positive values are observed, indicating the loss of almost-sure stability of the single mode solution and suggesting possible complex dynamics.

		For comparison with the four-dimensional system, we fix $\zeta_2=0.1$ and plot the Lyapunov exponents of the single mode solution and the four-dimensional system as functions of $\sigma$ under the same parameter set; see the right panel of Figure~\ref{Fig:lyp}. 
		The two exponents exhibit qualitatively similar behavior: both increase with $\sigma$ and become positive beyond a threshold. While this does not constitute a quantitative equivalence, it supports the heuristic criterion used in \cite[Conjecture 2.1]{Baxendale24SIAM} that destabilization of the single mode solution may signal the onset of complex dynamics in the four-dimensional system.

		\begin{figure}[tbhp!] 
			\centering
			\includegraphics[scale = 0.35]{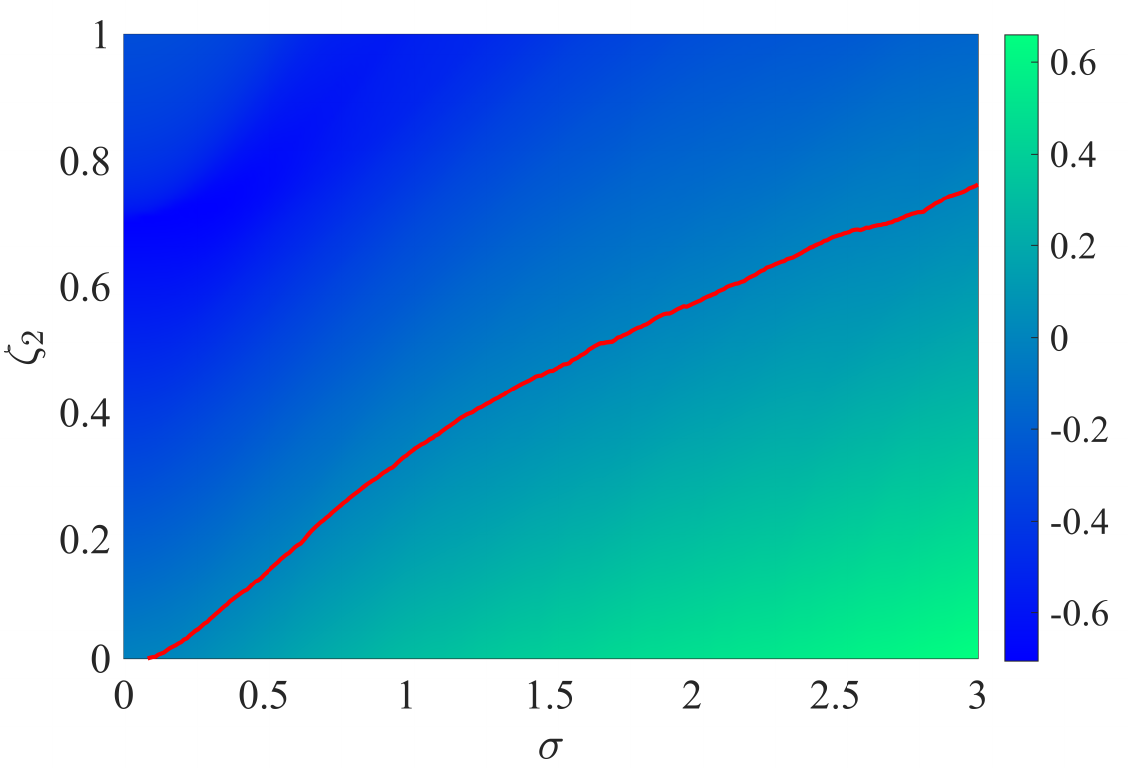}  
			\qquad \qquad 
			\includegraphics[scale = 0.35]{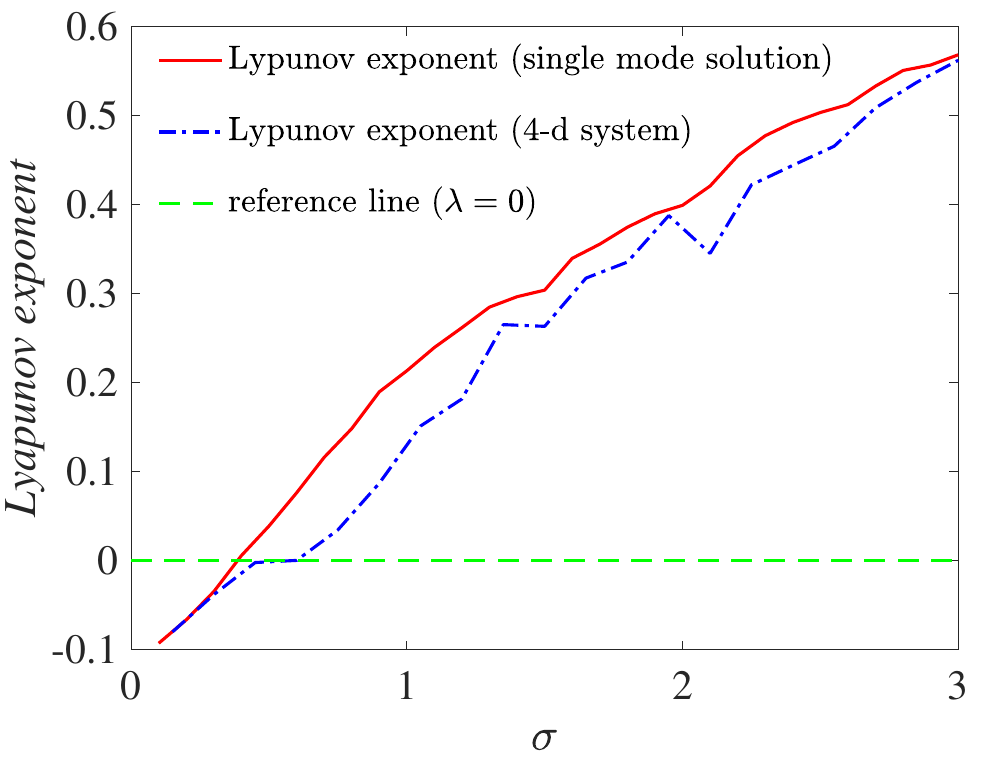}  
			\caption{Left: Numerical Lyapunov exponents for the single mode solution shown as a function of $(\sigma,\zeta_2)$. The red line indicates $\lambda^{\tau}=0$. Right: Numerical Lyapunov exponents for the single mode solution (solid) and the four-dimensional system (dashed) as functions of $\sigma$ with $\zeta_2=0.1$. 
				Parameters:  $T=1000$, $\tau=10^{-4}$,  $\gamma=0.8$, $\kappa_{1}=2$, $\kappa_{2}=0.5$, and $\zeta_{1}=8^{-\frac{1}{2}}$. }
			\label{Fig:lyp}
		\end{figure}

		Next, we fix a representative parameter set with $\sigma=2$ (for which $\lambda^\tau>0$) and examine the evolution of a large ensemble of initial conditions.
		Specifically, we sample $10^5$ initial points and simulate the system up to $T=50$ with step size $\tau=5\times10^{-4}$ under a fixed noise realization.
		Figure~\ref{Fig:attractor} displays (i) a projection of the ensemble onto the coordinates $(V_k,\mathcal V_k,\mathcal U_k)$ and (ii) the corresponding physical configuration represented as point clouds of the block and pendulum positions, illustrating the complex dynamics associated with the loss of stability of the single mode solution.
		
		To assess the numerical robustness of random attractors with respect to discretization, we recompute attractors using several step sizes $\tau=k\times10^{-3}$ with $k=0.5, 1, 2, 4$ under the same noise realization.
		Figure~\ref{Fig:convergence attractor} shows that the projected attractors onto $U_k$-$\mathcal{U}_k$ and $V_k$-$\mathcal{V}_k$ planes are nearly indistinguishable across different $\tau$, indicating the convergence of attractors as the step size is refined.

		Finally, we explore how the long-time dynamics depend on the noise intensity $\sigma$ by constructing bifurcation diagrams.
		For each value of $\sigma$, we compute the corresponding random attractor and project it onto the $\mathcal{U}_k$-axis and $\mathcal{V}_k$-axis, respectively; see Figure~\ref{Fig:bifurcation}.
		For small $\sigma$ (where $\lambda^\tau<0$) the projections collapse to the single point, corresponding to a stable state; for larger $\sigma$ (where $\lambda^\tau>0$) the projections broaden into intervals and then into wider bands, indicating a transition to chaotic dynamics as $\sigma$ increases.

		\begin{figure}[tbhp]
			\centering
			\includegraphics[scale = 0.4]{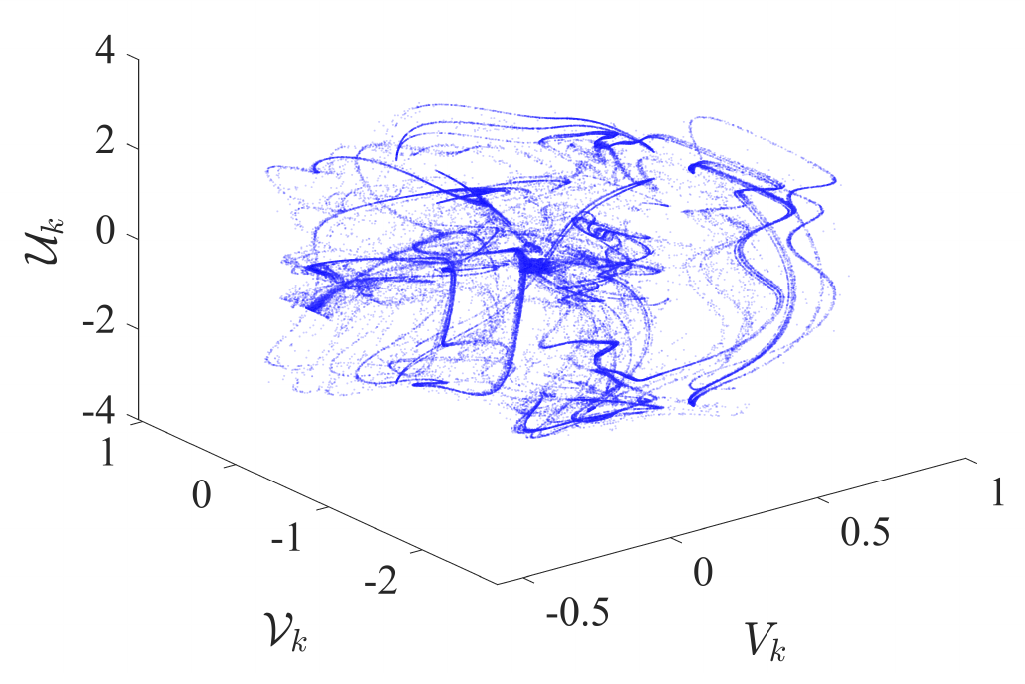}   
			\qquad\qquad 
			\includegraphics[scale = 0.4]{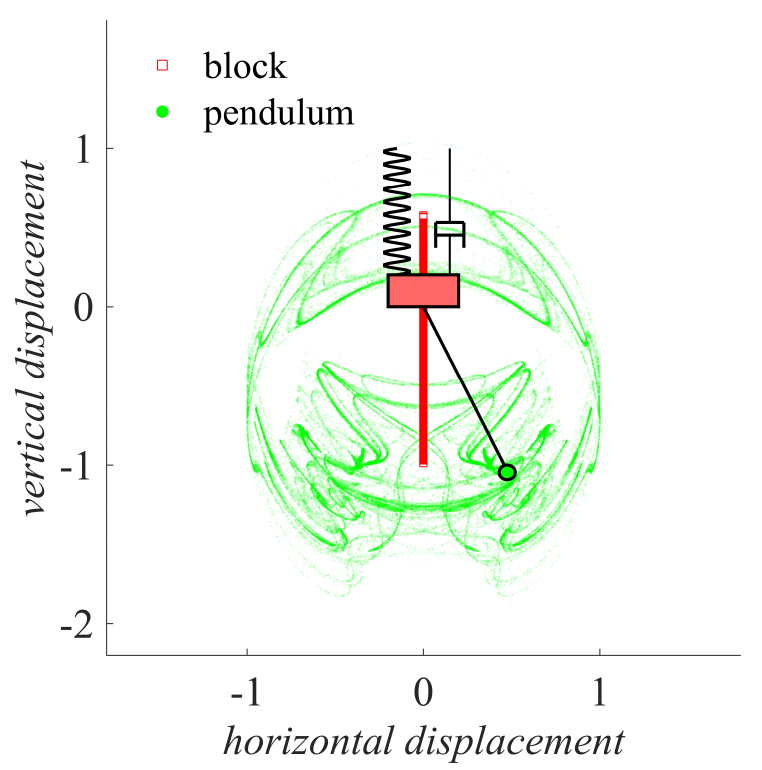} 
			\caption{Left: numerical random attractors projected onto $(V_k,\mathcal V_k,\mathcal U_k)$. 
				Right: point clouds of the block and pendulum positions.   Parameters: $\tau=5\times10^{-4}$, $\gamma=0.8$, $\kappa_{1}=2$, $\kappa_{2}=0.5$, $\zeta_{1}=8^{-\frac{1}{2}}$, $\zeta_{2}=0.1$, and $\sigma=2$.}
			\label{Fig:attractor}
		\end{figure}

		\begin{figure}[tbhp]
			\centering
			\includegraphics[scale=0.42]{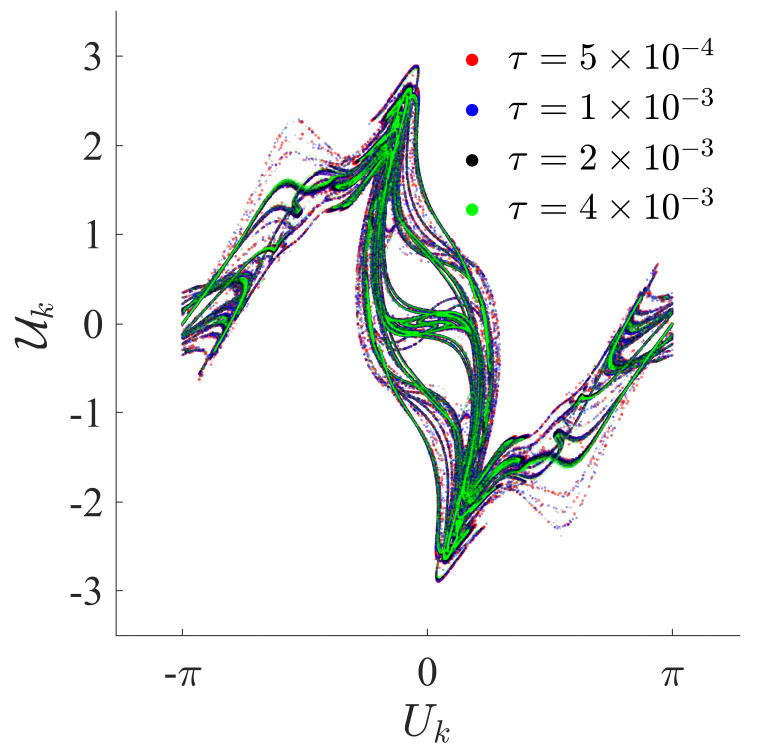}  
			\qquad \qquad 
			\includegraphics[scale=0.42]{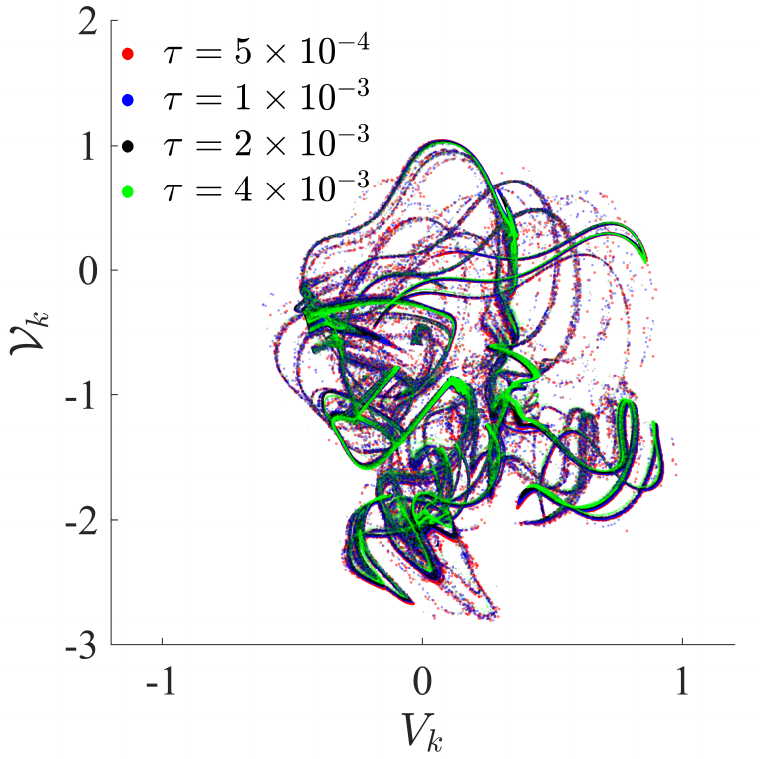}  
			\caption{Numerical random attractors projected onto $U_{k}$-$\mathcal{U}_{k}$ and $V_{k}$-$\mathcal{V}_{k}$ planes for different $\tau$. Other parameters are the same as in Figure \ref{Fig:attractor}. }
			\label{Fig:convergence attractor}
		\end{figure}

		\begin{figure}[tbhp] 
			\centering
			\includegraphics[scale = 0.35]{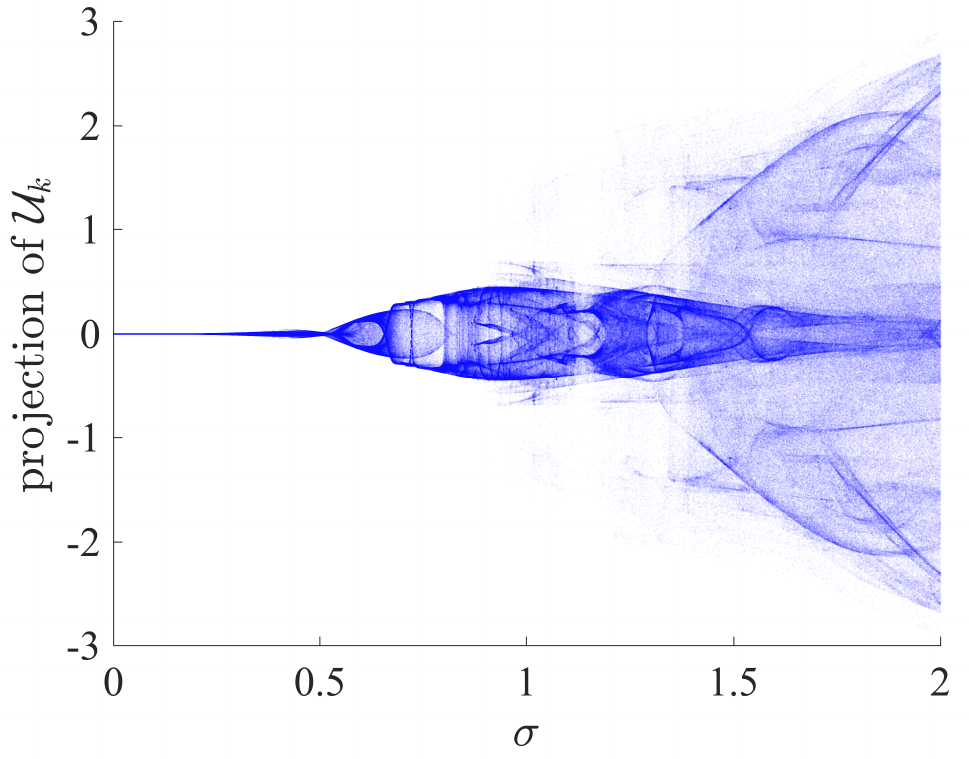} 
			\qquad \qquad 
			\includegraphics[scale = 0.35]{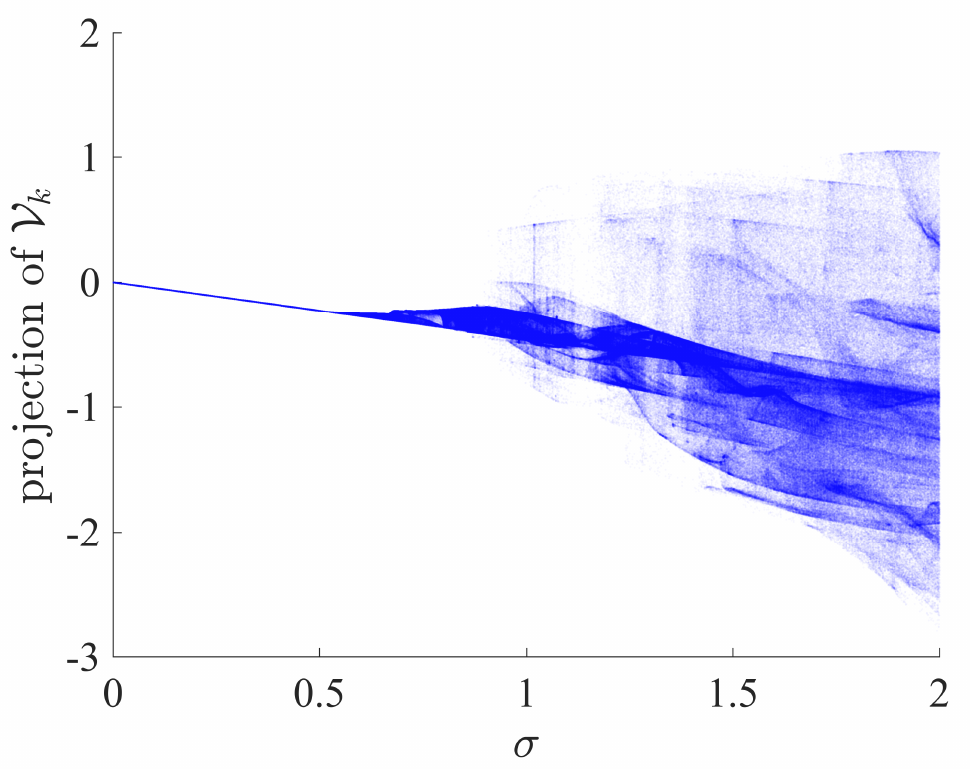}    
			\caption{Bifurcation diagrams for different $\sigma$ obtained from projections of the computed random attractors. 
				Parameters: $\gamma=0.8, \kappa_{1}=2, \kappa_{2}=0.5, \zeta_{1}=8^{-\frac{1}{2}}$, and  $\zeta_{2}=0.1$.}
			\label{Fig:bifurcation}
		\end{figure}

		\appendix
		\section{Estimate of $E_k$ in the proof of Theorem \ref{thm:numerical attractor}}\label{appendix Ektau}
		Recall that $\bar{S}_{k} = \left(1-\gamma\sin^2 \bar{U}_{k}\right)^{-\frac{1}{2}}$. By using \eqref{bar V}, we can rewrite  $E^{\tau}_{k}$ defined by \eqref{Ektau} as
		\begin{equation*}
			E^{\tau}_{k}=\frac{\bar{\mathcal{V}}_{k}^2}{2}  + \frac{\gamma\bar{\mathcal{U}}_{k}^2}{2}  - \bar{\mathcal{V}}_{k}Z_{k} + \frac{\bar{S}_{k}^2Z_{k}^2}{2}  - \gamma\bar{S}_{k}\bar{\mathcal{U}}_{k}Z_{k}\sin \bar{U}_{k}
			+ \frac{\kappa_{1}\bar{V}_{k}^2}{2}  + \delta_{3} \bar{V}_{k}(\bar{\mathcal{V}}_{k}-Z_{k}). 
		\end{equation*}
		It follows from Taylor's expansion and Young's inequality that for sufficiently small $\delta_{3}$, 
		\begin{align*}
			&\quad E^{\tau}_{k} - E^{\tau}_{k+1} \\
			&= 2\tau \zeta_1 \bar{\mathcal{V}}_{k+1}^2 + 2\tau\gamma(\zeta_2+\zeta_1\gamma\sin^2 \bar{U}_{k})\bar{S}_{k}^2 \bar{\mathcal{U}}_{k+1}^2
			+ 4\tau \zeta_1\gamma\bar{S}_{k}\sin \bar{U}_{k}  \bar{\mathcal{V}}_{k+1} \bar{\mathcal{U}}_{k+1} 
			+ \tau \gamma \kappa_{2}\bar{S}_{k}\sin \bar{U}_{k} \bar{\mathcal{U}}_{k+1} \\
			&\quad + \tau \delta_{3} \left(\kappa_{1} \bar{V}_{k+1}^2 - \bar{\mathcal{V}}_{k+1}^2 + 2\zeta_1\bar{V}_{k+1}\bar{\mathcal{V}}_{k+1} -  \gamma \bar{S}_{k}\sin \bar{U}_{k}  \bar{\mathcal{V}}_{k+1}\bar{\mathcal{U}}_{k+1}
			+ 2\zeta_{1}\gamma\bar{S}_{k}\sin \bar{U}_{k}  \bar{V}_{k+1}\bar{\mathcal{U}}_{k+1} \right)
			\\
			&\quad + \tau Z_{k+1} \left( -\kappa_{1} \bar{V}_{k+1} - 2\zeta_1\bar{\mathcal{V}}_{k+1} - 2\zeta_{1}\gamma\bar{S}_{k}\sin \bar{U}_{k}  \bar{\mathcal{U}}_{k+1} 
			- 2\gamma\bar{S}_{k}^3\sin \bar{U}_{k}  (\zeta_{2}+\zeta_{1}\gamma\sin^2 \bar{U}_{k})\bar{\mathcal{U}}_{k+1}\right)   \\
			&\quad + \tau Z_{k+1} \left( -\gamma\bar{S}_{k}^2\sin^2 \bar{U}_{k}  (\kappa_{2}+2\zeta_{1}\bar{\mathcal{V}}_{k+1}+\kappa_{1}\bar{V}_{k+1}) 
			+ \delta_{3}\bar{\mathcal{V}}_{k+1} + \delta_{3}\gamma \bar{S}_{k}\sin \bar{U}_{k} \bar{\mathcal{U}}_{k+1} \right)  \\ 
			&\quad + \alpha\tilde{z}_{k+1}  \left( -\bar{\mathcal{V}}_{k+1} + \bar{S}_{k}^2 Z_{k+1} - \gamma\bar{S}_{k}\sin \bar{U}_{k} \bar{\mathcal{U}}_{k+1} - \delta_{3}\bar{V}_{k+1} \right)  \\ 
			&\quad + \frac{Z_{k+1}^2(\bar{S}_{k}^2-\bar{S}_{k+1}^2)}{2}
			- \bar{\mathcal{U}}_{k+1}Z_{k+1} \left( \gamma\bar{S}_{k}\sin \bar{U}_{k} - \gamma\bar{S}_{k+1}\sin \bar{U}_{k+1} \right) 
			- (\bar{\mathcal{V}}_{k}-\bar{\mathcal{V}}_{k+1})(Z_{k}-Z_{k+1}) \\
			&\quad + \frac{(\bar{\mathcal{V}}_{k}-\bar{\mathcal{V}}_{k+1})^2}{2} + \frac{\gamma(\bar{\mathcal{U}}_{k}-\bar{\mathcal{U}}_{k+1})^2}{2} 
			+ \frac{\kappa_{1}(\bar{V}_{k}-\bar{V}_{k+1})^2}{2} 
			+ \frac{\bar{S}_{k}^2(Z_{k}-Z_{k+1})^2}{2} \\
			&\quad - \gamma\bar{S}_{k}\sin \bar{U}_{k} (\bar{\mathcal{U}}_{k}-\bar{\mathcal{U}}_{k+1})(Z_{k}-Z_{k+1}) 
			+ \delta_{3} (\bar{V}_{k}-\bar{V}_{k+1})\left((\bar{\mathcal{V}}_{k}-\bar{\mathcal{V}}_{k+1})-(Z_{k}-Z_{k+1})\right) \\[2mm]
			&\geq 2\tau \zeta_1 \bar{\mathcal{V}}_{k+1}^2 + 2\tau\gamma(\zeta_2+\zeta_1\gamma\sin^2 \bar{U}_{k})\bar{S}_{k}^2 \bar{\mathcal{U}}_{k+1}^2
			+ 4\tau \zeta_1\gamma\bar{S}_{k}\sin \bar{U}_{k}  \bar{\mathcal{V}}_{k+1} \bar{\mathcal{U}}_{k+1} 
			+ \tau \gamma \kappa_{2}\bar{S}_{k}\sin \bar{U}_{k} \bar{\mathcal{U}}_{k+1} \\
			&\quad + \tau \delta_{3} \left(\kappa_{1} \bar{V}_{k+1}^2 - \bar{\mathcal{V}}_{k+1}^2 + 2\zeta_1\bar{V}_{k+1}\bar{\mathcal{V}}_{k+1} -  \gamma \bar{S}_{k}\sin \bar{U}_{k}  \bar{\mathcal{V}}_{k+1}\bar{\mathcal{U}}_{k+1}
			+ 2\zeta_{1}\gamma\bar{S}_{k}\sin \bar{U}_{k}  \bar{V}_{k+1}\bar{\mathcal{U}}_{k+1} \right)
			\\
			&\quad - \epsilon\tau\left( \bar{V}_{k+1}^2 + \bar{\mathcal{V}}_{k+1}^2 + \bar{\mathcal{U}}_{k+1}^2 \right)  - C(\epsilon)\left(\tau (Z_{k+1}^2+1) +\tau^{-1} |\tilde{z}_{k+1}|^2 \right) 
			+ \frac{Z_{k+1}^2(\bar{S}_{k}^2-\bar{S}_{k+1}^2)}{2} \\
			&\quad
			- \bar{\mathcal{U}}_{k+1}Z_{k+1} \left( \gamma\bar{S}_{k}\sin \bar{U}_{k} - \gamma\bar{S}_{k+1}\sin \bar{U}_{k+1} \right)  \\
			&\geq C\tau \left( \bar{V}_{k+1}^2 + \bar{\mathcal{V}}_{k+1}^2 + \bar{\mathcal{U}}_{k+1}^2 \right) - C\left(\tau (Z_{k+1}^2+1)+ \tau^{-1} |\tilde{z}_{k+1}|^2 \right) + \frac{Z_{k+1}^2(\bar{S}_{k}^2-\bar{S}_{k+1}^2)}{2} \\
			&\quad 
			- \gamma\bar{\mathcal{U}}_{k+1}Z_{k+1} \left( \bar{S}_{k}\sin \bar{U}_{k} - \bar{S}_{k+1}\sin \bar{U}_{k+1} \right), 
		\end{align*}
		in which, by the mean value theorem, $\bar{U}_{k+1} - \bar{U}_{k} = \tau \bar{S}_{k} \bar{\mathcal{U}}_{k+1}$, and Young's inequality,  
		\begin{equation*}
			\frac{\left|Z_{k+1}^2(\bar{S}_{k+1}^2-\bar{S}_{k}^2)\right|}{2} 
			= \frac{\left|\gamma\sin \xi \cos \xi \, Z_{k+1}^2 (\bar{U}_{k+1}-\bar{U}_{k})\right|}{(1-\gamma\sin^2 \xi)^2}   
			\leq \epsilon \tau \bar{\mathcal{U}}_{k+1}^2 + C(\epsilon)\tau  Z_{k+1}^4,  
		\end{equation*}
		where $\xi$ is between $\bar{U}_{k}$ and $\bar{U}_{k+1}$. Similarly, 
		\begin{align*} 
			\left|\gamma\bar{\mathcal{U}}_{k+1}Z_{k+1} (\bar{S}_{k+1}\sin \bar{U}_{k+1}-\bar{S}_{k}\sin \bar{U}_{k}) \right|
			&\leq \frac{\gamma}{(1-\gamma)^{\frac{3}{2}}} |\bar{U}_{k+1}-\bar{U}_{k}| |Z_{k+1}| |\bar{\mathcal{U}}_{k+1}|. 
		\end{align*}
		By the triangle inequality, we obtain 
		\begin{align*}
			\tau|Z_{k+1}| &= \int_{t_{k}}^{t_{k+1}} |Z_{k+1}| \rmd s 
			\leq \int_{t_{k}}^{t_{k+1}} |z(\theta_s \omega)| \rmd s +\int_{t_{k}}^{t_{k+1}} |z(\theta_s \omega)-Z_{k+1}| \rmd s \\
			&\leq \int_{t_{k}}^{t_{k+1}} |z(\theta_s \omega)| \rmd s + \epsilon \tau + \frac{1}{\epsilon^3} \int_{t_{k}}^{t_{k+1}} |z(\theta_s \omega)-Z_{k+1}|^4 \rmd s, 
		\end{align*} 
		which, together with $\bar{U}_{k+1} = \bar{U}_{k} + \tau \bar{S}_{k} \bar{\mathcal{U}}_{k+1}$, $|\bar{U}_{k}|\leq2\pi$, and Young's inequality, leads to  
		\begin{align*}
			&|\bar{U}_{k+1}-\bar{U}_{k}| |\bar{\mathcal{U}}_{k+1}| |Z_{k+1}| \\
			\leq \, & \frac{1}{\tau} |\bar{U}_{k+1}-\bar{U}_{k}| |\bar{\mathcal{U}}_{k+1}| 
			\left( \int_{t_{k}}^{t_{k+1}} |z(\theta_s \omega)| \rmd s + \epsilon \tau + \frac{1}{\epsilon^3} \int_{t_{k}}^{t_{k+1}} |z(\theta_s \omega)-Z_{k+1}|^4 \rmd s \right) \\
			\leq \, & |\bar{S}_k| |\bar{\mathcal{U}}_{k+1}|^2 \int_{t_{k}}^{t_{k+1}} |z(\theta_s \omega)| \rmd s 
			+ \epsilon\tau|\bar{S}_k| |\bar{\mathcal{U}}_{k+1}|^2
			+ \frac{4\pi}{\epsilon^{3}\tau} |\bar{\mathcal{U}}_{k+1}| \int_{t_{k}}^{t_{k+1}} |z(\theta_s \omega)-Z_{k+1}|^4 \rmd s \\
			\leq \, & \frac{|\bar{\mathcal{U}}_{k+1}|^2}{(1-\gamma)^{\frac{1}{2}}} \int_{t_{k}}^{t_{k+1}} |z(\theta_s \omega)| \rmd s 
			+ \frac{\epsilon\tau|\bar{\mathcal{U}}_{k+1}|^2}{(1-\gamma)^{\frac{1}{2}}}
			+ \epsilon\tau |\bar{\mathcal{U}}_{k+1}|^2 
			+ \frac{4\pi^2}{\epsilon^{7}\tau^{2}}  \int_{t_{k}}^{t_{k+1}} |z(\theta_s \omega)-Z_{k+1}|^8 \rmd s. 
		\end{align*} 
		Thus we have 
		\begin{align*}
			&E^{\tau}_{k+1} + C\tau \left( \bar{V}_{k+1}^2 + \bar{\mathcal{V}}_{k+1}^2 + \bar{\mathcal{U}}_{k+1}^2 \right) \\ 
			\leq \,& E^{\tau}_{k}
			+ \tau \epsilon C(\gamma)\bar{\mathcal{U}}_{k+1}^2 
			+ C(\gamma) \bar{\mathcal{U}}_{k+1}^2  \int_{t_{k}}^{t_{k+1}} |z(\theta_s \omega)| \rmd s 
			+ C(\epsilon)\left(\tau (1+Z_{k+1}^4)+ \tau^{-1} |\tilde{z}_{k+1}|^2 \right)  \\ 
			& + C(\epsilon,\gamma) \tau^{-2} \int_{t_{k}}^{t_{k+1}} |z(\theta_s \omega)-Z_{k+1}|^8 \rmd s. 
		\end{align*}
		It follows from \eqref{BEM} that $\bar{\mathcal{U}}_{k+1}^2 \leq C (\bar{V}_{k}^2 + \bar{\mathcal{V}}_{k}^2 + \bar{\mathcal{U}}_{k}^2) + C\Delta W_{k}^2$. 
		In addition, for $\delta_{3}<\sqrt{\kappa_{1}}$, 
		\begin{equation*}
			CE^{\tau}_{k} - C Z_{k}^2
			\leq \bar{V}_{k}^2 + \bar{\mathcal{V}}_{k}^2 + \bar{\mathcal{U}}_{k}^2 
			\leq  CE^{\tau}_{k} + C Z_{k}^2, 
		\end{equation*}
		which, by letting $\epsilon$ sufficiently small, means that there exist constants $c_{5}, \bar{c}_{5}>0$ such that 
		\begin{equation}\label{mathcalG}
			\begin{aligned}
				E^{\tau}_{k+1}  &\leq \left(1-c_{5}\tau+\bar{c}_{5} \int_{t_{k}}^{t_{k+1}} |z(\theta_s \omega)| \rmd s\right) E^{\tau}_{k}  
				+ C \tau^{-2} \int_{t_{k}}^{t_{k+1}} |z(\theta_s \omega)-Z_{k+1}|^8 \rmd s  \\
				&\quad + C \tau \left(1+Z_{k+1}^4+Z_{k}^4+\Delta W_{k}^4\right) + C\int_{t_{k}}^{t_{k+1}} |z(\theta_s \omega)|^2 \rmd s  \\
				&=: \left(1-c_{5}\tau+\bar{c}_{5} \int_{t_{k}}^{t_{k+1}} |z(\theta_s \omega)| \rmd s\right) E^{\tau}_{k} + \mathcal{G}_{k+1} \\
				&\leq e^{-c_{5}\tau+\bar{c}_{5} \int_{t_{k}}^{t_{k+1}} |z(\theta_s \omega)| \rmd s} E^{\tau}_{k} + \mathcal{G}_{k+1}, 
			\end{aligned}
		\end{equation}
		where the inequality $1+x \leq e^{x}$, $\forall \, x\in\mathbb{R}$ is used.

		\section{Proof of Lemma \ref{lemma:Bk}}\label{appendix:A}
		Recall that $B_{k} = (\beta_{k}, \bar{\beta}_{k})^{\top} = \|B_{k}\|(\cos\psi_{k},\sin\psi_{k})^{\top}$.  For convenience, in the following, let $c_{k}=\cos\psi_{k}$ and $s_{k}=\sin\psi_{k}$. 
		\begin{proof}[Proof of Lemma \ref{lemma:Bk}]
			From \eqref{VXbeta}, we have  
			\begin{equation*}
				\beta_{k}=\beta_{k+1} - \tau \bar{\beta}_{k+1}  \quad \text{and} \quad  
				\bar{\beta}_{k}=\bar{\beta}_{k+1}  + \tau \mathcal{Z}_{k+1} \beta_{k+1} 
				+ 2\tau\zeta_2 \bar{\beta}_{k+1} - \sigma\beta_{k}\Delta W_{k} - \tau^2\mathcal{Z}_{k+1}\bar{\beta}_{k+1}, 
			\end{equation*}
			which yields that 
			\begin{equation}\label{Est:Bk2}
				\begin{aligned}
					\frac{\|B_{k}\|^2}{\|B_{k+1}\|^2} 
					&= 1 + \tau^2 s_{k+1}^2 - 2\tau c_{k+1} s_{k+1} (1- \mathcal{Z}_{k+1}) + 4 \tau \zeta_2 s_{k+1}^2 -  \frac{2\sigma\beta_k \bar{\beta}_{k+1}\Delta W_{k}}{\|B_{k+1}\|^2}  
					\\
					&\quad 
					- 2\tau^2 \mathcal{Z}_{k+1} s_{k+1}^2 + \frac{\sigma^2\beta_{k}^2\Delta W_{k}^2}{\|B_{k+1}\|^2} 
					+  \frac{\left(\tau \mathcal{Z}_{k+1} \beta_{k+1} + 2\tau \zeta_2 \bar{\beta}_{k+1}-\tau^2\mathcal{Z}_{k+1}\bar{\beta}_{k+1}\right)^2}{\|B_{k+1}\|^2} \\
					&\quad - \frac{2\sigma\beta_k \Delta W_{k}\left(\tau\mathcal{Z}_{k+1}\beta_{k+1}+2\tau\zeta_2\bar{\beta}_{k+1}-\tau^2\mathcal{Z}_{k+1}\bar{\beta}_{k+1}\right)}{\|B_{k+1}\|^2} \\
					&=:1 - \frac{2\sigma\beta_{k} \bar{\beta}_{k+1} \Delta W_{k}}{\|B_{k+1}\|^2}
					+ \bar{\mathbb{B}}_{k+1}
					=: 1 + \mathbb{B}_{k+1}.  
				\end{aligned}
			\end{equation}
			It follows from $\beta_{k}=\beta_{k+1}-\tau\bar{\beta}_{k+1}$ and Cauchy--Schwarz's inequality that 
			\begin{equation}\label{Est:mathbbB1} 
				|\bar{\mathbb{B}}_{k+1}| \leq C\tau \left(|\mathcal{Z}_{k+1}|^2+1\right) + C\Delta W_{k}^2,  \qquad 
				|\mathbb{B}_{k+1}|
				\leq C\tau \left(|\mathcal{Z}_{k+1}|^2+1\right) + C|\Delta W_{k}| 
				+ C\Delta W_{k}^2.  
			\end{equation}
			which verifies \eqref{Est:Bk1}.  In addition, it follows from \eqref{Est:Bk2} that 
			\begin{equation}\label{Est:beta and barbeta}
				\begin{aligned}
					&\frac{\beta_k \bar{\beta}_{k+1}}{\|B_{k+1}\|^2} 
					= \frac{\beta_k \left( \bar{\beta}_{k}-\tau\mathcal{Z}_{k+1}\beta_{k+1}-2\tau\zeta_2\bar{\beta}_{k+1} + \sigma\beta_{k}\Delta W_{k} + \tau^2\mathcal{Z}_{k+1}\bar{\beta}_{k+1}\right)}{\|B_{k+1}\|^2} \\
					&= \left(c_k s_k + \sigma c_k^2 \Delta W_{k}\right) \frac{\|B_{k}\|^2}{\|B_{k+1}\|^2} 
					- \frac{\tau\left(\beta_{k+1}-\tau\bar{\beta}_{k+1}\right) \left( \mathcal{Z}_{k+1}\beta_{k+1} + \left(2\zeta_2-\tau\mathcal{Z}_{k+1}\right)\bar{\beta}_{k+1}\right)}{\|B_{k+1}\|^2} \\ 
					&= \left(c_k s_k 
					+ \sigma c_k^2 \Delta W_{k}\right) (1+\mathbb{B}_{k+1}) 
					- \tau\left(c_{k+1}-\tau s_{k+1}\right) \left( \mathcal{Z}_{k+1}c_{k+1} + \left(2\zeta_2-\tau\mathcal{Z}_{k+1}\right)s_{k+1}\right). 
				\end{aligned}
			\end{equation}
			Substituting \eqref{Est:beta and barbeta} into \eqref{Est:Bk2} yields 
			\begin{align*}
				\frac{\|B_{k}\|^2}{\|B_{k+1}\|^2} 
				&= 1-2\sigma c_{k}s_{k}\Delta W_{k} + \tilde{\mathbb{B}}_{k+1}, 
			\end{align*}
			where, by \eqref{Est:mathbbB1}, 
			\begin{equation}\label{Est:mathbbB2}
				\begin{aligned}
					|\tilde{\mathbb{B}}_{k+1}| 
					&\leq C|\Delta W_{k}| \left( |\mathbb{B}_{k+1}| + |\Delta W_{k}|(1+|\mathbb{B}_{k+1}|) + \tau(1+\tau)(|\mathcal{Z}_{k+1}|+1) \right) + |\bar{\mathbb{B}}_{k+1}|  \\ 
					&\leq C\left(\Delta W_{k}^2 + \Delta W_{k}^4\right) 
					+ C\tau\left(|\mathcal{Z}_{k+1}|^4+1\right), 
				\end{aligned}
			\end{equation}
			which verifies \eqref{Est:Bk3}. 
			Moreover, by \eqref{VXbeta}, we obtain  
			\begin{equation}\label{mathbbD}
				\begin{aligned}
					\frac{\|B_{k+1}\|^2}{\|B_{k}\|^2} &= \frac{\left(\bar{\beta}_k-\tau\mathcal{Z}_{k+1}\beta_k+\sigma\beta_{k}\Delta W_{k}\right)^2 + \left(\beta_{k}(1+2\tau\zeta_{2})+\tau\left(\bar{\beta}_k-\tau\mathcal{Z}_{k+1}\beta_k+\sigma\beta_{k}\Delta W_{k}\right)\right)^{2}}{(1+2\tau\zeta_{2})^{2}(\beta_k^2+\bar{\beta}_k^2)} \\
					&= \frac{\left( s_{k}^2 + \left( \tau\mathcal{Z}_{k+1}c_k-\sigma c_{k}\Delta W_{k} \right)^2 - 2s_{k}\left( \tau\mathcal{Z}_{k+1}c_k-\sigma c_{k}\Delta W_{k} \right)\right) }{(1+2\tau\zeta_{2})^2} 
					+ c_{k}^2 \\
					&\quad 
					+ \frac{\tau^2\left(s_{k}- \tau\mathcal{Z}_{k+1}c_k+\sigma c_{k}\Delta W_{k} \right)^2}{(1+2\tau\zeta_{2})^2}  
					+ \frac{2\tau c_{k}\left(s_{k}- \tau\mathcal{Z}_{k+1}c_k+\sigma c_{k}\Delta W_{k} \right)}{1+2\tau\zeta_{2}}   \\
					&=: 1+\mathbb{D}_{k+1}, 
				\end{aligned}
			\end{equation}
			where 
			\begin{align*}
				|\mathbb{D}_{k+1}| &\leq \left( (1+2\tau\zeta_{2})^{-2}-1 \right) s_{k}^2 
				+ \left( \tau|\mathcal{Z}_{k+1}| + \sigma|\Delta W_{k}| \right)^2 
				+ 2 \left( \tau|\mathcal{Z}_{k+1}| + \sigma|\Delta W_{k}| \right) \\
				&\quad + \tau^2 \left( 1+\tau|\mathcal{Z}_{k+1}| + \sigma|\Delta W_{k}| \right)^2 
				+ 2\tau \left( 1+\tau|\mathcal{Z}_{k+1}| + \sigma|\Delta W_{k}| \right) \\
				&\leq C\tau \left(|\mathcal{Z}_{k+1}|^2+1\right) + C|\Delta W_{k}| 
				+ C\Delta W_{k}^2, 
			\end{align*}
			which completes the proof. 
		\end{proof}

		\section{Calculation of $c_k$ and $s_k$}\label{appendix:B}
		
		\begin{lemma}\label{lemma:B1}
			For any $k\geq0$, it holds that 
			\begin{align*}
				\frac{\|B_{k}\|-\|B_{k+1}\|}{\|B_{k+1}\|} = \frac{1}{2} \mathbb{B}_{k+1} -\frac{1}{8} \mathbb{B}_{k+1}^2 
				+  \frac{\|B_{k+1}\|^2\mathbb{B}_{k+1} ^3}{4(\|B_{k}\|+\|B_{k+1}\|)^2} 
				-  \frac{\|B_{k+1}\|^4\mathbb{B}_{k+1} ^4}{8(\|B_{k}\|+\|B_{k+1}\|)^4}, 
			\end{align*}
			where $\mathbb{B}_{k+1}$ is given by \eqref{Est:Bk2} and satisfies 
			\begin{equation}\label{mathbbB}
				\mathbb{B}_{k+1} = -2\tau c_{k+1} s_{k+1} (1- \mathcal{Z}_{k+1}) + 4 \tau \zeta_2 s_{k+1}^2 
				- 2 \sigma c_k s_k \Delta W_{k} 
				- \sigma^2c_{k}^2 \Delta W_{k}^2 (1 - 4s_{k}^2)
				+ \mathcal{B}_{1,k+1}, 
			\end{equation}
			with $\sup_{0\leq k \leq N}\mathbb{E}[|\mathcal{B}_{1,k}|^{p}] \leq C(p)\tau^{\frac{3p}{2}}$, $\forall p\geq1$. 
		\end{lemma}
		\begin{proof}
			It follows from \eqref{Est:Bk2} that $\|B_{k}\|^2 \|B_{k+1}\|^{-2} = 1+\mathbb{B}_{k+1}$. Thus 
			\begin{equation*}
				\frac{\|B_{k+1}\|}{\|B_{k}\|+\|B_{k+1}\|} 
				= \frac{1}{2} + \frac{\|B_{k+1}\|^2-\|B_{k}\|^2}{2(\|B_{k}\|+\|B_{k+1}\|)^2} 
				= \frac{1}{2} \left(1-  \frac{\mathbb{B}_{k+1}\|B_{k+1}\|^2}{\left(\|B_{k}\|+\|B_{k+1}\|\right)^2}\right),
			\end{equation*}
			which leads to 
			\begin{align*}
				\frac{\|B_{k}\|-\|B_{k+1}\|}{\|B_{k+1}\|} 
				&= \frac{\mathbb{B}_{k+1}}{2} -  \frac{\mathbb{B}_{k+1}^2\|B_{k+1}\|^2}{2\left(\|B_{k}\|+\|B_{k+1}\|\right)^2}
				= \frac{\mathbb{B}_{k+1}}{2} - \frac{\mathbb{B}_{k+1}^2}{8}  \left(1 -  \frac{\mathbb{B}_{k+1}\|B_{k+1}\|^2}{\left(\|B_{k}\|+\|B_{k+1}\|\right)^2}\right)^2 \\
				&= \frac{1}{2} \mathbb{B}_{k+1} -\frac{1}{8} \mathbb{B}_{k+1}^2 
				+  \frac{\|B_{k+1}\|^2\mathbb{B}_{k+1} ^3}{4(\|B_{k}\|+\|B_{k+1}\|)^2} 
				-  \frac{\|B_{k+1}\|^4\mathbb{B}_{k+1} ^4}{8(\|B_{k}\|+\|B_{k+1}\|)^4}. 
			\end{align*} 
			Next, we estimate $\mathbb{B}_{k+1}$. 
			Substituting \eqref{Est:Bk1} and \eqref{Est:Bk3} into \eqref{Est:beta and barbeta}, we have  
			\begin{align*}
				\frac{\beta_k \bar{\beta}_{k+1}}{\|B_{k+1}\|^2} 
				&= \left(c_k s_k + \sigma c_k^2 \Delta W_{k}\right) \frac{\|B_{k}\|^2}{\|B_{k+1}\|^2} \\
				&\quad - \frac{\left(\beta_{k+1}-\tau\bar{\beta}_{k+1}\right) \left( \tau\mathcal{Z}_{k+1}\beta_{k+1}+2\tau\zeta_2\bar{\beta}_{k+1} - \tau^2\mathcal{Z}_{k+1}\bar{\beta}_{k+1}\right)}{\|B_{k+1}\|^2} \\ 
				&= c_k s_k \left(1-2\sigma c_{k}s_{k}\Delta W_{k} + \tilde{\mathbb{B}}_{k+1}\right) 
				+ \sigma c_k^2 \Delta W_{k} (1+\mathbb{B}_{k+1}) \\
				&\quad - \left(c_{k+1}-\tau s_{k+1}\right) \left( \tau\mathcal{Z}_{k+1}c_{k+1}+2\tau\zeta_2s_{k+1} - \tau^2\mathcal{Z}_{k+1}s_{k+1}\right), 
			\end{align*}
			which, together with \eqref{Est:Bk2}, implies that 
			\begin{align*}
				\mathbb{B}_{k+1} 
				&= -2\tau c_{k+1} s_{k+1} (1- \mathcal{Z}_{k+1}) + 4 \tau \zeta_2 s_{k+1}^2 - 2 \sigma \Delta W_{k} \left( c_k s_k -2\sigma c_{k}^2s_{k}^2\Delta W_{k}  
				+ \sigma c_k^2 \Delta W_{k} \right) \\
				& \quad + \frac{\sigma^2\beta_{k}^2\Delta W_{k}^2}{\|B_{k}\|^2}(1+\mathbb{B}_{k+1}) +\tau^2 s_{k+1}^2 
				+ \left(\tau \mathcal{Z}_{k+1} c_{k+1} + 2\tau \zeta_2 s_{k+1}- \tau^2\mathcal{Z}_{k+1}s_{k+1}\right)^2 \\ 
				& \quad - 2\tau^2\mathcal{Z}_{k+1}s_{k+1}^2 
				- 2 \sigma \Delta W_{k} \left( c_k s_k \tilde{\mathbb{B}}_{k+1}+\sigma c_k^2\Delta W_{k}\mathbb{B}_{k+1}  \right) \\
				&\quad -2\sigma\Delta W_{k}\left(c_{k+1} - \tau s_{k+1}\right) \left(\tau\mathcal{Z}_{k+1}c_{k+1}+2\tau\zeta_2s_{k+1}-\tau^2\mathcal{Z}_{k+1}s_{k+1}\right) \\
				&=: -2\tau c_{k+1} s_{k+1} (1- \mathcal{Z}_{k+1}) + 4 \tau \zeta_2 s_{k+1}^2 - 2 \sigma c_k s_k \Delta W_{k} 
				- \sigma^2 c_{k}^2 \Delta W_{k}^2 ( 1 -4s_{k}^2) 
				+ \mathcal{B}_{1,k+1}, 
			\end{align*}
			where 
			\begin{align*}
				|\mathcal{B}_{1,k+1}|
				\leq C\Delta W_{k}^2 |\mathbb{B}_{k+1}| + 
				C\left(\tau^2+\tau|\Delta W_{k}|\right) \left(1+\mathcal{Z}_{k+1}^2\right)
				+ C|\Delta W_{k}| \left(  |\tilde{\mathbb{B}}_{k+1}| + |\Delta W_{k}| |\mathbb{B}_{k+1}| \right). 
			\end{align*}
			In view of \eqref{Est:mathbbB1}, \eqref{Est:mathbbB2}, and the moment boundedness of $|\mathcal{Z}_{k}|$, we obtain that for any $p\geq1$, 
			\begin{equation*}
				\sup\limits_{0\leq k \leq N}\mathbb{E}[|\mathcal{B}_{1,k}|^{p}] \leq C(p)\tau^{\frac{3p}{2}}, 
			\end{equation*}
			which completes the proof. 
		\end{proof}

		By using Lemma \ref{lemma:B1}, we next derive the equation satisfied by $\bm{\mathcal{C}}_{k} = \left(\eta_{k}, \bar{\eta}_{k}, c_{k}, s_{k}\right)^{\top}$. 
		\begin{lemma}\label{lemma:A3}
			For any $k\geq0$, it holds that 
			\begin{equation*}
				\bm{\mathcal{C}}_{k+1} = \bm{\mathcal{C}}_{k} + \tau \bm{\varPhi}(\bm{\mathcal{C}}_{k+1}) + \bm{\varPsi}(\bm{\mathcal{C}}_{k}) \Delta W_{k} + \varTheta(\bm{\mathcal{C}}_{k+1}) 
				+ (\tau-\Delta W_{k}^2) \tilde{\varTheta}(\bm{\mathcal{C}}_{k}), 
			\end{equation*} 
			where $\bm{\varPsi}(\bm{\mathcal{C}}_{k})=\left(0, \, \sigma, \, - \sigma c_{k}^2 s_{k}, \, \sigma c_{k}^3\right)^{\top}$, $\tilde{\varTheta}(\bm{\mathcal{C}}_{k})=\big(0, \, 0, \, \tfrac{\sigma^2}{2} c_{k}^5-\sigma^2c_{k}^3s_{k}^2, \, \tfrac{3\sigma^2}{2} c_{k}^4s_{k}\big)^{\top}$,  
			\begin{align*}
				\bm{\varPhi}(v,x,c,s)=\begin{pmatrix}
					x \\[1mm] -\kappa_{1} v - 2\zeta_2 x \\[1mm] 
					s -c^2 s (1- \kappa_{2}-2\zeta_1x-\kappa_{1}v) + 2\zeta_2 cs^2 + \sigma^2c^3s^2-\tfrac{\sigma^2}{2} c^5 \\[1mm]
					-c + c^3(1- \kappa_{2}-2\zeta_1x-\kappa_{1}v) - 2\zeta_2c^2s 
					-\frac{3\sigma^2}{2} c^4 s
				\end{pmatrix}, 
			\end{align*} 
			and $\varTheta(\bm{\mathcal{C}}_{k+1})=\big(0,\,0,\,\tilde{\mathcal{C}}_{k+1},\,\tilde{\mathcal{S}}_{k+1}\big)^{\top}$ with $\sup_{0\leq k \leq N}\mathbb{E}\big[\|\varTheta(\bm{\mathcal{C}}_{k})\|^{p}\big] \leq C(p) \tau^{\frac{3p}{2}}$. 
		\end{lemma}
		\begin{proof}
			The evolution of first two components $(\eta_{k},\bar{\eta}_{k})$ follows directly from the equations in \eqref{VXbeta}.
			Thus, in the following, we aim to derive the evolution of $(c_{k}, s_{k})$. 
			Recall that $(\beta_{k},\bar{\beta}_{k})^{\top} = \|B_{k}\|(c_{k},s_{k})^{\top}$. By using $\beta_{k+1}=\beta_{k}+\tau\bar{\beta}_{k+1}$ and  Lemma \ref{lemma:B1}, we obtain 
			\begin{align*}
				&c_{k+1} - c_{k} 
				= \frac{\beta_{k+1}-\beta_{k}}{\|B_{k+1}\|} + c_{k}  \frac{\|B_{k}\|-\|B_{k+1}\|}{\|B_{k+1}\|}  \\
				&= \frac{\tau\bar{\beta}_{k+1}}{\|B_{k+1}\|} + c_{k} \left( \frac{1}{2} \mathbb{B}_{k+1} -\frac{1}{8} \mathbb{B}_{k+1}^2 
				+  \frac{\|B_{k+1}\|^2\mathbb{B}_{k+1} ^3}{4(\|B_{k}\|+\|B_{k+1}\|)^2} 
				-  \frac{\|B_{k+1}\|^4\mathbb{B}_{k+1} ^4}{8(\|B_{k}\|+\|B_{k+1}\|)^4} \right). 
			\end{align*}
			Combining with \eqref{mathbbB}, we have 
			\begin{align*}
				c_{k+1} - c_{k} 
				&= \tau s_{k+1} -c_{k}  
				\left( \tau c_{k+1} s_{k+1} (1- \mathcal{Z}_{k+1}) 
				- 2\tau \zeta_2 s_{k+1}^2 
				+ \sigma c_{k}s_{k} \Delta W_{k} \right)\\
				&\quad - \frac{\sigma^2}{2}  c_{k}^3 \Delta W_{k}^2 (1-3s_{k}^2) 
				+ \mathcal{B}_{2,k+1}, 
			\end{align*}
			where $\mathcal{B}_{2,k+1}=  \frac{c_k \mathcal{B}_{1,k+1}}{2} -\frac{c_{k}}{8} (\mathbb{B}_{k+1}^2 - 4\sigma^2 c_{k}^2s_{k}^2\Delta W_{k}^2) + \frac{c_{k}\|B_{k+1}\|^2\mathbb{B}_{k+1} ^3}{4(\|B_{k}\|+\|B_{k+1}\|)^2} - \frac{c_{k}\|B_{k+1}\|^4\mathbb{B}_{k+1} ^4}{8(\|B_{k}\|+\|B_{k+1}\|)^4}$ and satisfies 
			\begin{align*}
				|\mathcal{B}_{2,k+1}| 
				\leq \frac{1}{2} |\mathcal{B}_{1,k+1}|
				+\frac{1}{8} \left|\mathbb{B}_{k+1}^2 - 4\sigma^2 c_{k}^2s_{k}^2\Delta W_{k}^2 \right| 
				+ \frac{1}{4} |\mathbb{B}_{k+1}|^3 
				+ \frac{1}{8} |\mathbb{B}_{k+1}|^4, 
			\end{align*}
			which means that $\sup_{0\leq k \leq N}\mathbb{E}[|\mathcal{B}_{2,k}|^{p}] \leq C(p)\tau^{\frac{3p}{2}}$, $\forall p\geq1$. 
			Consequently, 
			\begin{align*}
				&c_{k+1}-c_{k} \\
				=\,& \tau s_{k+1} -\tau c_{k}c_{k+1} s_{k+1} (1- \mathcal{Z}_{k+1}) +2\tau \zeta_2 c_{k}s_{k+1}^2  
				- \sigma c_k^2 s_k\Delta W_{k} - \frac{\sigma^2}{2} c_k^3 \Delta W_{k}^2 (c_k^2-2s_k^2)   
				+ \mathcal{B}_{2,k+1} \\ 
				=\,& \tau s_{k+1} -\tau c_{k+1}^2 s_{k+1} (1- \mathcal{Z}_{k+1}) + 2\tau \zeta_2 c_{k+1}s_{k+1}^2 + \sigma^2\tau c_{k+1}^3s_{k+1}^2-\frac{\sigma^2\tau}{2} c_{k+1}^5   
				- \sigma c_k^2 s_k\Delta W_{k}\\
				& +(\tau-\Delta W_{k}^2) \left(\frac{\sigma^2}{2} c_{k}^5-\sigma^2c_{k}^3s_{k}^2\right) + \tilde{\mathcal{C}}_{k+1}, 
			\end{align*}
			where $\sup_{0\leq k \leq N}\mathbb{E}[|\tilde{\mathcal{C}}_{k}|^{p}] \leq C(p)\tau^{\frac{3p}{2}}$, $\forall p\geq1$.

			We next consider the evolution of $s_{k}$, which is very similar to that of $c_{k}$. By \eqref{VXbeta}, $\beta_{k}=\beta_{k+1}-\tau\bar{\beta}_{k+1}$, and Lemma \ref{lemma:B1}, we have 
			\begin{align*}
				s_{k+1} - s_{k} 
				&= -\tau \mathcal{Z}_{k+1}c_{k+1} - 2\tau\zeta_2s_{k+1}
				- s_{k} \left( \tau c_{k+1} s_{k+1} (1- \mathcal{Z}_{k+1}) - 2\tau \zeta_2 s_{k+1}^2 \right) 	\\ 
				&\quad +\sigma\frac{\beta_{k}\Delta W_{k}}{\|B_{k+1}\|}  - \sigma c_k s_k^2 \Delta W_{k} 
				-\frac{\sigma^2}{2} c_{k}^2 s_{k} \Delta W_{k}^2 (1-3s_{k}^2) + \mathcal{B}_{3,k+1}, 
			\end{align*}
			where $\sup_{0\leq k \leq N}\mathbb{E}[|\mathcal{B}_{3,k}|^{p}] \leq C(p)\tau^{\frac{3p}{2}}$, $\forall p\geq1$. 
			It follows from $\beta_{k}=\beta_{k+1}-\tau\bar{\beta}_{k+1}$ that 
			\begin{align*}
				s_{k+1} - s_{k}
				&= -\tau c_{k+1} + \tau c_{k+1}^3(1-\mathcal{Z}_{k+1}) 
				+ \tau c_{k+1}s_{k+1}(1-\mathcal{Z}_{k+1})(s_{k+1}-s_{k}) \\
				& \quad - 2\tau\zeta_2c_{k+1}^2s_{k+1} 
				- 2\tau\zeta_2s_{k+1}^2 (s_{k+1}-s_{k})
				- \frac{\sigma^2}{2}c_{k}^2 s_{k} \Delta W_{k}^2 (1-3s_k^2)
				\\
				& \quad +\sigma (c_{k+1}-\tau s_{k+1}) \Delta W_{k} 
				- \sigma c_{k}s_{k}^2\Delta W_{k} 
				+ \mathcal{B}_{3,k+1}. 
			\end{align*}
			By using the equation satisfied by $c_{k}$, namely, 
			\begin{equation*}
				c_{k+1}=c_{k}+\tau \varPhi_{3}(\mathcal{Z}_{k+1},c_{k+1},s_{k+1})-\sigma c_{k}^2 s_{k}\Delta W_{k} +(\tau-\Delta W_{k}^2) \left(\frac{\sigma^2}{2} c_{k}^5-\sigma^2c_{k}^3s_{k}^2\right) + \tilde{\mathcal{C}}_{k+1}, 
			\end{equation*}
			where $\varPhi_{3}(z,c,s) = s -c^2 s (1-z) + 2\zeta_2 cs^2 + \sigma^2c^3s^2-\tfrac{\sigma^2}{2} c^5$, we have 
			\begin{align*}
				s_{k+1} - s_{k} 
				&= -\tau c_{k+1} + \tau c_{k+1}^3(1-\mathcal{Z}_{k+1}) - 2\tau\zeta_2c_{k+1}^2 s_{k+1}
				+\sigma c_{k}^3 \Delta W_{k} 
				-\frac{3\sigma^2}{2} c_{k}^4 s_{k}\Delta W_{k}^2  \\ 
				& \quad + \tau c_{k+1}s_{k+1}(1-\mathcal{Z}_{k+1})(s_{k+1}-s_{k}) 
				- 2\tau\zeta_2s_{k+1}^2 (s_{k+1}-s_{k}) 
				- \sigma\tau s_{k+1}\Delta W_{k} \\ 
				& \quad + \sigma \Delta W_{k} \left( \tau \varPhi_{3} +(\tau-\Delta W_{k}^2) (\tfrac{\sigma^2}{2} c_{k}^5-\sigma^2c_{k}^3s_{k}^2) + \tilde{\mathcal{C}}_{k+1} \right) 
				+ \mathcal{B}_{3,k+1} \\
				&= -\tau c_{k+1} + \tau c_{k+1}^3(1-\mathcal{Z}_{k+1}) - 2\tau\zeta_2c_{k+1}^2s_{k+1} 
				-\frac{3\sigma^2\tau}{2} c_{k}^4 s_{k} 
				+\sigma c_{k}^3 \Delta W_{k} \\
				& \quad + \frac{3\sigma^2}{2}(\tau-\Delta W_{k}^2)  c_{k}^4 s_{k} 
				+ \tilde{\mathcal{S}}_{k+1}, 
			\end{align*}
			where $\sup_{0\leq k \leq N}\mathbb{E}[|\tilde{\mathcal{S}}_{k}|^{p}] \leq C(p)\tau^{\frac{3p}{2}}$, $\forall p\geq1$. 	
		\end{proof}

		\bibliographystyle{plain}
		\bibliography{srb.bib}

	\end{document}